\documentclass[a4paper,10pt,reqno]{article}

\usepackage[english]{babel}
\usepackage{amsmath}
\usepackage{amsfonts}
\usepackage{amssymb}
\usepackage{amsthm}
\usepackage{mathrsfs}
\usepackage{float}
\usepackage{graphics}
\usepackage{color}
\usepackage{graphicx}
\usepackage[colorlinks=false]{hyperref}
\usepackage{verbatim}
\usepackage{vmargin}
\usepackage[T1]{fontenc}
\usepackage{times}
\usepackage{dsfont}
\usepackage{geometry}
\def\Z{\mathbb{Z}}
\def\Q{\mathbb{Q}}

\newcommand{\eps}{\varepsilon}

\newcommand{\la}{\lambda}

\newcommand{\ga}{{\gamma}}

\def\N{{\mathbb N}}
\renewcommand{\to}{\rightarrow}

\numberwithin{equation}{section}

\theoremstyle{plain}
\newtheorem{teor}{Theorem}[section]
\newtheorem{ese}[teor]{Example}
\newtheorem{prop}[teor]{Proposition}
\newtheorem{lem}[teor]{Lemma}
\newtheorem{cor}[teor]{Corollary}

\newcommand{\bdm}{\begin{displaymath}}
\newcommand{\edm}{\end{displaymath}}
\newcommand{\bpb}{\begin{prob}}
\newcommand{\epb}{\end{prob}}

\newcommand{\beq}{\begin{equation}}
\newcommand{\eeq}{\end{equation}}
\newcommand{\bem}{\begin{multline}}
\newcommand{\eem}{\end{multline}}
\newcommand{\bes}{\begin{ese}}
\newcommand{\ees}{\end{ese}}
\newcommand{\bde}{\begin{defi}}
\newcommand{\ede}{\end{defi}}
\newcommand{\bpr}{\begin{prop}}
\newcommand{\epr}{\end{prop}}
\newcommand{\ble}{\begin{lem}}
\newcommand{\ele}{\end{lem}}
\newcommand{\bte}{\begin{teor}}
\newcommand{\ete}{\end{teor}}
\newcommand{\bco}{\begin{cor}}
\newcommand{\eco}{\end{cor}}

\theoremstyle{definition}
\newtheorem{defi}[teor]{Definition}
\newtheorem{remark}[teor]{Remark}

\newcommand{\R}{\mathbb{R}}

\newcommand{\T}{\mathbb{T}}

\newcommand{\calI}{{\mathcal I}}
\newcommand{\calJ}{{\mathcal J}}

\newcommand{\MM}{{\mathcal M}}
\newcommand{\NN}{{\mathcal N}}

\newcommand{\calP}{{\mathcal P}}

\newcommand{\RR}{{\mathcal R}}

\newcommand{\TT}{{\mathcal T}}

\newcommand{\al}{\alpha}
\newcommand{\kkk}{\kappa}

\newcommand{\AAA}{\boldsymbol{A}}

\newcommand{{\resonance}}{relevant self-energy cluster }

\newcommand{\de}{\delta}
\newcommand{\ZZ}{\mathbb{Z}}
\newcommand{\OO}{\mathcal{O}}

\newcommand{\Zodd}{\mathbb{Z}_\mathrm{odd}}
\newcommand{\Lm}{\Lambda}

\newcommand{\Res}{\mathrm{Res}}
\newcommand{\Ker}{\mathrm{Ker}}
\newcommand{\odd}{\mathrm{odd}}
\newcommand{\ad}{\mathrm{ad}}

\begin{document}

\title{\bf Chaotic--like transfers of energy in Hamiltonian PDEs
}
\date{}

\author{\bf
 Filippo Giuliani$^{\dag}$,  Marcel Guardia$^{\dag}$,  Pau Martin$^{\dag}$, Stefano Pasquali$^{\dag}$
\\
\small
${}^\dag$ UPC \footnote{
This project has received funding from the European Research Council (ERC) under the European Union's Horizon 2020 research and innovation programme (grant agreement No 757802). P.M. has been partially funded by the Spanish Government MINECO-FEDER grant PGC2018-100928-B-I00. M. G and S. P. have been also partially supported by
the Spanish MINECO-FEDER Grant PGC2018-098676-B-100 (AEI/FEDER/UE) and by the Catalan Institution for Research
and Advanced Studies via an ICREA Academia Prize 2019.  F.G., M. G, P. M. and S.P. have  been also partially supported by the Catalan grant 2017SGR1049. S.P. acknowledges financial support from the Spanish ``Ministerio de Ciencia, Innovaci\'on y
Universidades'', through the Mar\'{\i}a de Maeztu Programme for Units of Excellence (2015-
2019) and the Barcelona Graduate School of Mathematics.
}
}

\maketitle

\begin{abstract}
We consider the nonlinear cubic Wave, the Hartree and the nonlinear cubic
Beam equations on $\mathbb{T}^2$ and we prove the existence of
different types of solutions which exchange energy between Fourier modes in certain
time scales. This exchange can be considered ``chaotic-like'' since either the
choice of activated modes or the time spent in each transfer
can be chosen randomly. The key point of the construction of those orbits is the
existence of heteroclinic connections between invariant objects and the
construction of symbolic dynamics (a Smale horseshoe) for the Birkhoff
Normal Form truncation of those equations.
%
\end{abstract}

\tableofcontents

\section{Introduction}


A fundamental question in nonlinear Hamiltonian Partial Differential Equations (PDEs) on compact manifolds is to understand how solutions can exchange energy among Fourier modes as time evolves.
A way to capture such behaviors is to analyze the invariant objects of the equation (or a ``good approximation of it''), such as periodic orbits or invariant tori, and to understand how they structure the global dynamics through their stable and unstable manifolds and their possible intersections.
This ``dynamical systems'' approach works very well, for instance, for PDEs on the torus $\mathbb{T}^n$. Such equations can be seen as infinite dimensional systems of ODEs for the Fourier coefficients and classical perturbative arguments can be adapted to the infinite dimensional context for the analysis of stability and instability phenomena. This approach has been classically applied to the analysis of stable motions, that is KAM Theory (the literature is huge, we refer to \cite{BertiBumi} for an overview on the subject and to the reference therein). However, its application to  exchange of energy phenomena is much more recent.

In the last decade there has been a lot of activity in building exchange of energy behaviors in different Hamiltonian PDEs almost exclusively  for the nonlinear Schr\"odinger equation. They can be classified into two groups. The first one are the so-called beating solutions \cite{GrebertV11,GT,GPT,HT, HausP17}. Those are orbits that are essentially supported on a finite numbers of modes and whose energy oscillates between those modes in a certain time range.

The other group are those addressing the problem of transfer of energy. That is, constructing orbits whose energy is transferred to increasingly higher modes as time evolves \cite{Bourgain96,Kuksin96,Kuksin97b,CKSTT,Hani12,Guardia14,GuardiaK12,HaniPTV15, HausProcesi,GuardiaHP16, Pocovnicu11, Pocovnicu12, Maspero18g, Delort10, GerardG10,  GerardG11}. Those are solutions whose dynamics is essentially supported  in a large number of modes and it is related to \emph{weak turbulence}. J. Bourgain considered this problem one of the key questions in Hamiltonian PDEs for the XXI century \cite{Bourgain00b}.

Most of  these results rely on analyzing certain truncations of the Hamiltonian PDEs (its first order Birkhoff normal form truncation) and building invariant objects for such models. Note that these first order Birkhoff normal forms are typically  non-integrable Hamiltonian systems (at least in dimension greater or equal than 2)  with very complicated dynamics. Nevertheless, restricted to suitably chosen invariant subspaces those models are integrable (they have ``enough'' first integrals in involution), and therefore one can have a very precise knowledge of their orbits in such invariant subspaces. Most of the results cited above strongly rely on this integrability on subspaces to  construct unstable motions and exchange of energy solutions. This is somewhat surprising from the point of view of (finite dimensional) dynamical systems where usually unstable motions and drifting orbits must rely on non-integrability and transverse homoclinic orbits.

Can one take advantage of the non-integrability and chaoticity  of a normal form truncation to construct new types of beating solutions? Can one exploit this chaoticity/non-integrability to build new type of dynamics in Hamiltonian PDEs? This is the goal of this paper. We consider three different PDEs, a nonlinear Wave equation, a nonlinear Beam equation and the Hartree equation (see \eqref{Wave}, \eqref{Beam} and \eqref{Hartree} below) and we are able to show the non-integrability and chaoticity (symbolic dynamics) of its Birkhoff normal form. This allows us to obtain different types of exchange of energy behaviors for the actual PDEs in some time scales. In particular,
\begin{itemize}
 \item Solutions which exchange energy in a chaotic-like way between a given set of modes. By chaotic-like we refer to orbits such that oscillate in being supported in two different sets of modes and the ``oscillation times'' can be chosen ``randomly'', see Theorem \ref{TeoWaveBeam} below for the precise statement.
\item Chaotic-like transfer of energy phenomenon: those orbits are essentially supported in a finite number of modes and the support is changing as follows. At each transition two modes get deactivated (their modulus becomes essentially constant) and we can choose randomly which new two modes are activated (their modulus starts oscillating) among certain set.
See Theorem \ref{thm:traveling_periodic_beating} below for the precise statement.
\end{itemize}
These results provide different types of beating solutions which are significantly different from the previous results \cite{GT,GPT,HT}. The beating solutions in these papers  exchange energy periodically in time and they rely on integrability and existence of action-angle variables. On the contrary,  in the present paper the oscillations can be ``randomly'' chosen: in the first one with respect to the time and in the second one with respect to the choice of activated modes.


Our second result leads to transfer of energy. However, the transfer does not involve arbitrarily high modes and therefore does not lead to growth of Sobolev norms.
The methods in \cite{CKSTT} for the construction of solutions exhibiting growth of the norms seem to fit very well for the NLS model  \cite{Hani12,Guardia14,GuardiaK12,HaniPTV15, HausProcesi,GuardiaHP16}. Nevertheless, it is not clear how to apply it to other PDEs.  We think that the present work could represent a first step to strengthen the  strategy in \cite{CKSTT} so that is  applicable to other PDEs by incorporating tools and mechanisms  inspired by the theory of Arnold diffusion. In Section \ref{sec:growth} we relate our results to the approach developed in \cite{CKSTT}.

The key point to  obtain the results in this paper is to consider certain first order truncations of the PDEs which can be treated as nearly integrable Hamiltonian systems. Then, one can apply classical methods in dynamical systems such as Melnikov Theory, shadowing arguments (Lambda lemma), hyperbolic invariant sets and symbolic dynamics.

\color{black}

\subsection{Main results}\label{sec:NLWMain}
Consider the completely resonant cubic nonlinear Wave equation on the $2$-dimensional torus
\begin{equation}\label{Wave}
u_{tt}-\Delta u+u^3=0 \qquad u=u(t, x), \quad t\in\mathbb{R}, \quad x\in\mathbb{T}^2
\end{equation}
and the cubic nonlinear Beam equation
\begin{equation}\label{Beam}
u_{tt}+\Delta^2 u+u^3=0\qquad u=u(t, x), \quad t\in\mathbb{R}, \quad x\in\mathbb{T}^2.
\end{equation}
We prove the existence of special beating solutions for such PDEs, namely solutions that exhibit transfer of energy between Fourier modes. Such solutions $u(t, x)$ are mainly Fourier supported on a finite set of $4$-tuple resonant modes
 \begin{equation}\label{def:lambdaset}
 \Lambda:=\{n^{(r)}_j\}^{r=1, \dots, N}_{j=1, \dots, 4}\subset\Z^2,
 \end{equation}
 with $N\geq 2$, in the sense that
 \[
 u(t, x)=\sum_{j\in\Lambda} a_j(t)\,e^{\mathrm{i} j\cdot x}+R(t, x)
 \]
 where $R(t, x)$ is small in some Sobolev norm. The transfers of energy between modes in $\Lambda$ are \emph{chaotic-like}, in the following sense. Either
\begin{itemize}
\item[(a)] one can prescribe a finite sequence of times $t_1, \dots, t_n$ and find a solution that exists for \emph{long but finite time} exhibiting transfers of energy among the modes in $\Lambda$ at the prescribed times $t_1, \dots, t_n$
\end{itemize}
 or
 \begin{itemize}
 \item[(b)] one can prescribe a sequence of resonant tuples $\{n^{(r_n)}_{j}\}_{n=1, \dots, k}\subseteq \Lambda$ and find a  solution and a sequence of times $t_1, \dots, t_k$ such that at time zero many  modes are "switched off" (modulus of the modes almost constant) and at times $t_n$ the modes $(n^{(r_n)}_{1},n^{(r_n)}_{2},n^{(r_n)}_{3},n^{(r_n)}_{4})$ are "switched on", in the sense that they start to exchange between them.
\end{itemize}
Those phenomena are consequence of the presence of (partially) hyperbolic, finite dimensional manifolds which are approximately invariant for the equations \eqref{Wave}, \eqref{Beam} and possess stable and unstable invariant manifolds that intersect transversally within some energy level.

We look for beating solutions in the following subspace
\[
\mathcal{U}_{\odd}:=\left\{ u=\sum_{j\in\mathbb{Z}_{\odd}^2} u_j \,e^{\mathrm{i} j \cdot x} \right  \}, \quad \mathbb{Z}^2_{\odd}:=\left\{ (j^{(1)},
j^{(2)})\in\mathbb{Z}^2 \,: \,\,
j^{(1)}\,\,\text{odd}\,\,,\,\,j^{(2)}\,\,\mbox{even}\right \},
\]
which is invariant under the flow of the equations \eqref{Wave}, \eqref{Beam} (see \cite{Procesi2010ANF}).
The origin of such subspace is an elliptic fixed point and the solutions of the variational equation
\[
\ddot{u}_j+\la_j^2 u_j=0 \quad j\in\mathbb{Z}^2_\odd
\]
where $\la_j=|j|$ (for the Wave equation \eqref{Wave}) and $\la_j=|j|^2$ (for the Beam equation \eqref{Beam}),
 are superposition of decoupled harmonic oscillators, hence all solutions are periodic/quasi-periodic/almost-periodic in time. In particular there is no transfer of energy between the linear modes when time evolves.
This implies that the existence of beating solutions (if any) depend on the presence of the nonlinearities.
To catch the nonlinear effects in a neighborhood of an elliptic equilibrium we perform a Birkhoff normal form analysis. Namely
we construct changes of coordinates\footnote{It is well known that the existence of such changes of coordinates cannot be always guaranteed because of the presence of small divisor problems and / or derivatives in the nonlinear terms. At this stage, one can consider  the normal form truncation as a formal ``good first order'' of the full equation. To show that is truly a good first order in the regions of the phase space that we consider, we adopt the strategy of performing a \emph{weak version} of the Birkhoff normal form which does not remove all the non-resonant terms but a finite number of them.} that transform the Hamiltonian of the equations \eqref{Wave}, \eqref{Beam} into a Hamiltonian of the form
\begin{equation}\label{def:BNFAbs}
K=K^{(2)}+K^{(4)}+\mathcal{R},
\end{equation}
where $K^{(i)}$ are homogenous terms of degree $i$ and $\mathcal{R}$ is a function that can be considered as a small perturbation. Then, one can consider  the truncated system
\begin{equation}\label{def:HamN}
\NN:=K^{(2)}+K^{(4)},
\end{equation}
called \emph{normal form} (see \eqref{def:commonham} below for the explicit formulas),  as a model which describes the effective dynamics of equations \eqref{Wave}, \eqref{Beam} for a certain range of times.

The normal form Hamiltonian $\NN$ possesses many finite-dimensional, symplectic, invariant subspaces of the form $V_\Lambda:=\{ u_j=0\,\, \,\,\forall j\notin \Lambda \}$, where $\Lambda\subset\mathbb{Z}^2_\odd$ is a finite set.
 We shall prove the following.
\begin{teor}\label{thm:geometric}
Let $N\geq 2$. There exist sets\footnote{Actually there exist ``many sets'' with such properties. See Remark \ref{rmk:Lambda} below.} $\Lambda\subset\mathbb{Z}^2_\odd$ of cardinality $4 N$ such that $V_\Lambda$ is invariant by the dynamics of $\NN$ and the following holds.
\begin{itemize}
\item[(i)] Let $N=2$. Then, the flow $\Phi_t$ associated to $\NN$ in $V_\Lambda$ has the following property. There exists a section $\Pi$ transverse to the flow $\Phi_t$ such that the induced Poincar\'e map
\[
 \calP:\mathcal{U}=\mathring{\mathcal{U}}\subset\Pi\to\Pi
\]
has an invariant set $X\subset \mathcal{U}$ which is homeomorphic to $\Sigma\times\T^5$ where $\Sigma=\N^\ZZ$ is the set of sequences of natural numbers. Moreover, the dynamics of $\calP:X\to X$ is topologically conjugated to the following dynamics
\[
 \widetilde\calP:\Sigma\times\T^5\to \Sigma\times\T^5, \qquad  \widetilde\calP(\omega, \theta)=(\sigma\omega, \theta+f(\omega))
\]
where $\sigma$ is the usual shift $(\sigma\omega)_k=\omega_{k+1}$ and $f:\Sigma\to\R^5$ is a continuous function.

Namely $\calP$ has a Smale horseshoe of infinite symbols as a factor.
\item[(ii)] There exist $N$ partially hyperbolic  $2(N+1)$-dimensional tori $\mathbb{T}_1, \dots, \mathbb{T}_N$ invariant for the restriction of the normal form Hamiltonian $\mathcal{N}$ at the subspace $V_\Lambda$ which have the following property. Take arbitrarily small neighborhoods $V_i$ of $\mathbb{T}_i$ and  any sequence $\{p_i\}_{i\geq 1}\subset\N^\N$. Then,  there exists an orbit $u(t)$ and a sequence of times $\{t_i\}_{i\geq 1}$ such that
\[
u(t_i)\in V_{p_i}.
\]
\end{itemize}
\end{teor}

\begin{remark}\label{rmk:Lambda}
The set $\Lambda\subset\ZZ^2$ is the union of $N$ resonant tuples (with certain properties). The ``shape'' of the resonant tuples in $\mathbb{Z}^2$ are different for the Beam and Wave Equations. For the Beam equation, as for the cubic nonlinear Schr\"odinger equation, are rectangles with vertices in  $\ZZ^2$. For the Wave equation are modes $n_1,n_2,n_3,n_4\in\mathbb{Z}^2$, which satisfy
\[
 n_1-n_2+n_3-n_4=0,\quad | n_1|-|n_2|+|n_3|-|n_4|=0.
\]
Those tuples form a parallelogram inscribed on an ellipse with foci at $F_1=0$ and $F_2=n_1+n_2$ and semi-major axis $a=(|n_1|+|n_2|)/2$.

Let us explain in which sense there are \emph{many sets $\Lambda\subset\ZZ^2$} for which Theorem \ref{thm:geometric} (and also Theorems \ref{TeoWaveBeam} and \ref{thm:traveling_periodic_beating} below) are satisfied. Theorem \ref{thm:geometric} relies on  proving the transverse intersection of certain invariant manifolds. This transversality is proven by  perturbative methods and, therefore, we need $\mathcal{N}|_\Lambda$ to be close to integrable. For the Wave \eqref{Wave} and Beam \eqref{Beam} equations this relies on choosing appropriate sets $\Lambda$.
The precise statement goes as follows. Fix $\eps>0$ (which will measure the closeness to integrability). Then, for any $R\gg 1$,  one can  choose the resonant tuples in the set $\Lambda$ generically in the annulus
 \[
  R(1-\eps)\leq |n|\leq R (1+\eps).
 \]
Generically means that one has to exclude the zero set of a finite number of algebraic varieties (and the number of those is independent of $\eps$ and $R$).
\end{remark}

\begin{figure} 
\begin{center}
\includegraphics[width=0.5\textwidth]{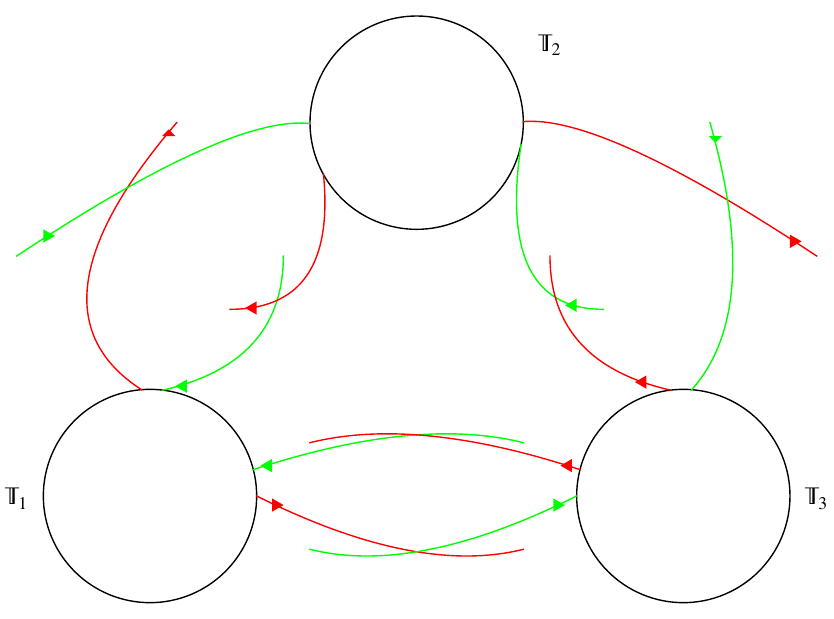}
\end{center}
\caption{Invariant tori with their stable (green) and unstable (red) invariant manifolds. This transition chain of tori allows (plus the Lambda lemma) gives the orbits of Item $(ii)$ of Theorem \ref{thm:geometric} which visit the invariant tori with any prescribed orbit}
\label{fig:SeqTori}
\end{figure}

The items $(a)$ and $(b)$ above are consequences respectively of items $(i)$ and $(ii)$ in Theorem \ref{thm:geometric}. Let us make some remark on the type of dynamics for the normal form Hamiltonian $\mathcal{N}$.
\begin{itemize}
 \item Item (i) of Theorem \ref{thm:geometric} gives the existence of invariant sets for the Birkhoff normal form truncation which possess chaotic dynamics. Such chaotic dynamics is obtained through the classical Smale horseshoe dynamics for a suitable Poincar\'e map. This invariant set is constructed in the neighborhood of homoclinic points to an invariant tori orbit (which becomes a periodic orbit for a suitable symplectic reduction). The (infinite) symbols codify the closeness to the invariant manifolds of the periodic orbit, and therefore the larger the symbol is the longer the return time to the section $\Pi$ is. In particular, one can construct orbits which take longer and longer time to return $\Pi$ for higher iterates.

 Even if the theorem, as stated, gives the existence of one invariant set, one actually can construct  a Smale horseshoe at each energy level.

 \item Item (ii) of Theorem \ref{thm:geometric} gives orbits which visit (possibly infinitely many times) a given set of invariant tori in any prescribed order. The construction of such orbits follows the classical strategy of Arnold Diffusion \cite{Arnold64}. That is, is a consequence of the existence of a chain of invariant tori (again periodic orbits in a suitable symplectic reduction) connected by transverse heteroclinic connections (see Figure \ref{fig:SeqTori}) plus a classical shadowing argument (Lambda lemma, see for instance \cite{FontichM98}).

 This is radically different from the approach in \cite{CKSTT, GuardiaK12}. In these papers, the authors consider the normal form associated to the nonlinear cubic Schr\"odinger equation. This normal form  has ``extra integrability'', due to the symmetries of the model,
and the considered heteroclinic orbits are not transverse. Therefore,  the associated  shadowing arguments are more delicate.  We refer to \cite{DelshamsSZ18} for a thorough analysis of non-transverse shadowing arguments. In particular, the authors of this paper show that the number of dimensions needed for the shadowing depend on the number of tori the orbits have to visit (what they called the \emph{dropping the dimension} mechanism).


 As for item (i) one can obtain the explained behavior at each energy level. Indeed, the invariant tori come in families parameterized by the energy level and therefore one can obtain this shadowing behavior at each energy level as well

\end{itemize}

Note that the knowledge of the orbits obtained in Theorem \ref{thm:geometric} is \emph{for all time}. If one adds the errors dropped from the original equation, that is $\mathcal{R}$ in \eqref{def:BNFAbs}, one can obtain orbits for equations \eqref{Wave}, \eqref{Beam} which follow the orbits of  Theorem \ref{thm:geometric} for some time scales. Next theorem gives solutions of equations \eqref{Wave} and \eqref{Beam} which (approximately) behave as those obtained in Item $(i)$ of Theorem \ref{thm:geometric}.

\begin{teor}\label{TeoWaveBeam}
Let $N=2$ and fix  $0<\varepsilon\ll 1$. Then for a large choice of sets $\Lambda=\{n_i\}_{i=1}^8\subset\mathbb{Z}^2$ as in \eqref{def:lambdaset}
there exists $\mathtt{T}_0\gg 1$ such that for all $\mathtt{T}\geq \mathtt{T}_0$ there exists $M_0>0$ such that for all $M\geq M_0$ there exists $\delta_0=\delta_0(M, \varepsilon, \mathtt{T})>0$ such that $\forall \delta\in (0, \delta_0)$ the following holds.

Choose any   $k\geq  1$ and any sequence $\{ m_j\}_{j=1}^k$ such that $m_j\geq M_0$ and\footnote{The condition $\sum_{j=1}^k m_j\leq M-k$ is just to ensure that the sequence $\{t_j\}_{j=1}^k$ defined below belongs to the interval  $[0, \delta^{-2} M \mathtt{T}]$.} $\sum_{j=1}^k m_j\leq M-k$.  Then, there exists a solution $u(t, x)$ of \eqref{Wave}, \eqref{Beam} for $t\in [0, \delta^{-2} M \mathtt{T}]$  of the form
\begin{equation*}
u(t, x)=\frac{\de}{\sqrt{2}}\sum_{i=1}^ 8 |n_i|^{-\kappa/2}\left(a_{n_i}( t)\,e^{\mathrm{i} n_i \cdot x}+\overline{a_{n_i}}(t)\,e^{-\mathrm{i} n_i \cdot x}\right)+R_1(t, x)
\end{equation*}
where $\kappa=1$ for the Wave equation \eqref{Wave} and $\kappa=2$ for the Beam equation \eqref{Beam}, and  $\sup_{t\in [0, \delta^{-2} M \mathtt{T}]}\lVert R_1 \rVert_{H^s(\T^2)}\lesssim_s \delta^{3/2}$ for all $s\geq 0$. The first order $\{ a_{n_i}\}_{i=1\dots 8}$
satisfies
\[
\begin{split}
\lvert a_{n_1}(t) \rvert^2&=\lvert a_{n_3} (t) \rvert^2=1- \lvert a_{n_2}(t) \rvert^2=1-\lvert a_{n_4} (t) \rvert^2,\\
\lvert a_{n_5}(t) \rvert^2&=\lvert a_{n_7} (t) \rvert^2=1- \lvert a_{n_6}(t) \rvert^2=1-\lvert a_{n_8} (t) \rvert^2,
\end{split}
\]
and has the following behavior.
\begin{itemize}
 \item \textbf{First resonant tuple (Periodic transfer of energy):} There exists a $\mathtt{T}$-periodic function $Q(t)$, independent of $\delta$ and satisfying $\min_{[0,\mathtt{T}]} Q(t) < \eps$ and $\max_{[0,\mathtt{T}]} |Q(t)| > 1-\eps$, such that
 \[\lvert a_{n_1}(t) \rvert^2= Q (\delta^{2} t)+R_2(t)\qquad \text{with} \qquad \sup_{t\in \mathbb{R}} \lvert R_2(t) \rvert\leq \varepsilon.\]
\item \textbf{Second resonant tuple (Chaotic-like transfer of energy):} There exists a sequence of times $\{t_j\}_{j=0}^{k}$ satisfying $t_0=0$ and
\[
 t_{j+1}=t_j+\de^{-2}\mathtt{T}\left(m_j+\theta_j\right)\qquad \text{with}\qquad \theta_j\in (0,1)
\]
such that
\[
 \lvert a_{n_5}( t_j) \rvert^2=\frac{1}{2}.
\]
Moreover, there exists another sequence $\{\bar t_j\}_{j=1\ldots k}$ satisfying $t_j<\bar t_j<t_{j+1}$ such that,
\begin{equation}\label{eq:TheoRandomTimeOsc1}
\begin{split}
\lvert a_{n_5}(t) \rvert^2&>\frac{1}{2}\qquad \text{for}\qquad t\in (t_j,\bar t_j)\\
\lvert a_{n_5}(t) \rvert^2&<\frac{1}{2}\qquad \text{for}\qquad t\in (\bar t_j,t_{j+1})
\end{split}
%
\end{equation}
and
\begin{equation}\label{eq:TheoRandomTimeOsc2}
\sup_{t\in (t_j,\bar t_j)}\lvert a_{n_5}(t) \rvert^2\geq 1-\varepsilon\qquad \text{ and }\qquad
\inf_{t\in (\bar t_j, t_{j+1})}\lvert a_{n_5}(t) \rvert^2\leq \varepsilon.
\end{equation}

\end{itemize}
\end{teor}

Note that the first order $\{\delta a_{n_i}\}_{i=1\ldots8}$ are the trajectories obtained in Theorem \ref{thm:geometric}--(i) which belong to the horseshoe. This phenomenon is genuinely nonlinear since for the linear equation the \emph{actions} $|a_{n_i}(t)|^2=\text{constant}$.

The first resonant tuple has a periodic beating behavior similar to \cite{GT}. On the contrary, the behavior of the second resonant tuple is radically different. The modulus of the modes $a_{n_i}$, $i=5,6,7,8$ ``oscillate'' from being   $\OO(\eps)$ to  being $\OO(\eps)$-close to 1. However,  the sequence of times $\{t_j\}$ in which all the modes in the tuple have the same modulus, that is
\[
\lvert a_{n_5}( t_j) \rvert^2=\lvert a_{n_6}( t_j) \rvert^2=\lvert a_{n_7}( t_j) \rvert^2= \lvert a_{n_8}( t_j) \rvert^2=\frac{1}{2},
\]
(and the modulus of $a_{n_5}$ and $ a_{n_7}$ is increasing) can be chosen randomly as any (large enough) integer multiple of $\mathtt{T}$.


Finally, let us explain the role of the constant $\mathtt{T}$ in the theorem. To build the horseshoe in Theorem \ref{thm:geometric}, we apply a symplectic reduction to $\mathcal{N}|_\Lambda$ (see \eqref{def:HamN}) which leads to a 2 degree of freedom Hamiltonian. For this Hamiltonian we construct a periodic orbit with transverse invariant homoclinic orbits. The time $\mathtt{T}$ is the period of this periodic orbit and can be taken arbitrarily big.

Now we state the second main result of this paper, which gives solutions of equations \ref{Wave} and \ref{Beam} which (approximately) behave as those obtained in Item $(ii)$ of Theorem \ref{thm:geometric}.
\begin{teor}\label{thm:traveling_periodic_beating}
Let $N\geq 2$,  $k\gg 1$, $0<\eps\ll 1$.  Then for a large choice of a set $\Lambda:=\{n^{(r)}_j\}^{r=1, \dots, N}_{j=1, \dots, 4} \subset\mathbb{Z}^2$ as in \eqref{def:lambdaset} there exist $\delta_0>0$, $T>0$,  such that for any $\de\in (0,\de_0)$ and any  sequence $\omega=(\omega_1, \dots, \omega_k), \omega_i\in \{1, \dots, N  \}$, there exists
 a solution $u(t, x)$ of the \eqref{Wave}, \eqref{Beam}
 of the form
\begin{equation*}
u(t, x)=\frac{\de}{\sqrt{2}}\sum_{n\in \Lambda} |n_i|^{-\kappa/2}\left(a_{n}(t)\,e^{\mathrm{i} n \cdot x}+\overline{a_{n}}(t)\,e^{-\mathrm{i} n \cdot x}\right)+R_3(t, x) \qquad t\in [0, \delta^{-2} T]
\end{equation*}
where $\kappa=1$ for the Wave equation \eqref{Wave} and $\kappa=2$ for the Beam equation \eqref{Beam}, $ \sup_{t\in[0, \delta^{-2} T]} \| R_3(t, x) \|_{H^s(\mathbb{T}^2)}\lesssim_s \delta^{3/2}$ for all $s\geq 0$,
and the  first order $\{ a_{n}\}_{n\in\Lambda}$,
has the following behavior:

There exist  some $\al_p,\beta_p$ satisfying
\[
\al_p<\beta_p<\al_{p+1} \qquad \text{ and }\qquad \beta_{p}-\alpha_p\gtrsim |\ln\eps|, \quad p=1, \dots, k
\]
such that, if one splits the time interval as $[0, \delta^{-2} T]=I_1\cup J_{1, 2}\cup I_2\cup J_{2, 3}\cup\dots \cup J_{k-1, k }\cup I_k$
with
\[
I_p=[\delta^{-2}\al_p, \delta^{-2}\beta_p],\quad J_{p,p+1}=[\delta^{-2}\beta_p,\delta^{-2}\al_{p+1}],
\]
such that $\{a_{n}\}_{n\in\Lambda}$ satisfies:
\begin{itemize}
\item In the \textbf{beating-time} intervals $I_p$, there exists $t_p>0$ such that
\begin{align*}
\sup_{t\in I_p} \Big| |a_{n^{(\omega_p)}_1}(t)|^2-Q(\delta^2 t-t_p)   \Big|&\le\eps \\
\sup_{t\in I_i}  |a_{n^{(\omega_i)}_1}(t)|^2
& \leq \varepsilon \qquad &\text{ for }&\qquad i\neq p,
\end{align*}
where $Q(t)$ is the periodic function given by Theorem \ref{TeoWaveBeam}.
\item In the \textbf{transition-time} intervals $J_{p, p+1}$,
\begin{align*}
\sup_{t\in J_i} |a_{n^{(\omega_i)}_1}(t)|^2&\geq 1-\eps \qquad \quad&\text{ for }&\qquad i= 1, \dots, N,\\
\end{align*}
\end{itemize}
and $|a_{n^{(\omega_i)}_1}(t)|^2=|a_{n^{(\omega_i)}_3}(t)|^2$ , $|a_{n^{(\omega_i)}_r}(t)|^2=1-|a_{n^{(\omega_i)}_{1}}(t)|^2$ with $r=2, 4$.
\end{teor}
The solutions obtained in this theorem are approximations of those obtained in Item $(ii)$ of Theorem \ref{thm:geometric} and possess two different regimes. The orbits of Theorem \ref{thm:geometric} are obtained by shadowing a sequence of invariant tori (periodic orbits for a suitable symplectic reduction) connected by transverse heteroclinic orbits. Then, what we call \emph{beating-time} intervals are the time intervals where the orbit is in a small neighborhood of each of the periodic orbits. In this regime, (the moduli of) some modes oscillate periodically, whereas the others are at rest.  The \emph{transition-time} intervals correspond to time intervals in which the orbit is ``traveling'' along a heteroclinic orbit and is ``far'' from all periodic orbits. In this regime, all modes undergo a drastic change to drift along the heteroclinic connection.

\paragraph{Hartree equation.} Similar results hold true also for  the Hartree equation
\begin{equation}\label{Hartree}
\mathrm{i} u_t=\Delta u+(V\star \lvert u \rvert^2)\,u\qquad u=u(t, x), \quad t\in\mathbb{R}, \quad x\in\mathbb{T}^2
\end{equation}
with a convolution potential $V(x)=\sum_{j\in\mathbb{Z}^2} V_j\,e^{\mathrm{i} j \cdot x}$ such that
\begin{equation}\label{def:potential}
V\colon \mathbb{T}^2\to\mathbb{R}, \quad V(x)=V(-x)
\end{equation}
and assuming the following hypothesis. Once fixed the set $\Lambda\subset\mathbb{Z}^2$, the Fourier coefficients $V_j$ of the  potential with $j=n_1-n_2$ for some $n_1, n_2\in\Lambda$ satisfy
\begin{equation}\label{cond:potential}
\quad V_j=1+\eps\gamma_j \quad \text{with}\quad \eps\ll 1.
\end{equation}
Assume that  the coefficients $\gamma_j$ satisfy a non-degeneracy condition which is of codimension 1 and take $\eps$ small enough. Then, the Hartree equation has solutions of the form
\[
 u(t, x)=\de\sum_{n\in \Lambda} a_{n}(t)\,e^{\mathrm{i} n \cdot x}+R(t, x)
\]
where the first order $\{a_{n}\}$ and the remainder $R$ satisfy the statements given either in  Theorem \ref{TeoWaveBeam} (where $R\rightsquigarrow R_3$) or \ref{thm:traveling_periodic_beating} (where $R\rightsquigarrow R_4$) .

Since the obtaining of such behaviors for the Hartree equation is the same as for Wave and Beam equations, in Sections \ref{sec:weak_normal_form}--\ref{sec:transfer} we prove the results together for the three equations.


\paragraph{Comments to Theorems \ref{TeoWaveBeam}, \ref{thm:traveling_periodic_beating}}
\begin{itemize}
\item {\bf Smale Horseshoes in PDEs:} Theorem \ref{thm:geometric} provides a Smale Horseshoe for the Birkhoff normal form. This invariant set is partially hyperbolic and partially elliptic if considered in the whole infinite dimensional phase space. This is the reason why, a priori, this invariant set is not persistent for the full equations \ref{Wave}\ref{Beam}, \ref{Hartree}. As far as the authors know, the existence of Smale horseshoes in Hamiltonian PDEs has been mostly obtained by adding dissipation to the equation which make these sets become fully hyperbolic (see \cite{HolmesM81,BertiC02,BatelliF05}). See \cite{BatelliG08}, for an infinite dimensional Hamiltonian system with a Smale horseshoe.

\item {\bf Beating partially hyperbolic quasiperiodic tori:} The Smale horseshoe obtained in Theorem \ref{thm:geometric} possesses a dense set of periodic orbits. Even if the horseshoe may not persist for the
 equations \ref{Wave},\ref{Beam},, \ref{Hartree}, KAM Theory should give the persistence of these periodic orbits. In \cite{HausP17}, the authors prove the existence of beating KAM
 Tori. The tori in \cite{HausP17} are elliptic whereas those coming from the horseshoe would be partially elliptic and partially hyperbolic.


\item {\bf Non-integrability of $\mathcal{N}|_\Lambda$ in \eqref{def:HamN}:} Theorem \ref{thm:geometric} (and therefore Theorems \ref{TeoWaveBeam} and \ref{thm:traveling_periodic_beating}) relies on the fact that $\mathcal{N}|_\Lambda$ is not integrable and admits invariant tori with transverse homoclinic orbits. On the other hand, the Birkhoff normal form truncation associated to the cubic Nonlinear Schr\"odinger equation
\begin{equation*}
 i u_t=\Delta u-|u|^2 u,\quad x\in \mathbb{T}^2
\end{equation*}
is such that $\mathcal{N}|_\Lambda$ is integrable. Therefore, the invariant manifolds of the invariant tori coincide and one cannot construct the orbits given in  Theorems \ref{TeoWaveBeam} and \ref{thm:traveling_periodic_beating} for this equation (at least not with the tools used in the present paper).

\item {\bf Weak Birkhoff normal form:} We point out that the reduction to the resonant model $\mathcal{N}|_\Lambda$ is obtained by means of a weak version of the Birkhoff normal form procedure around elliptic fixed points, which is described in Section \ref{sec:weak_normal_form}. This is needed when we deal with the Wave equation \eqref{Wave}. Indeed, even if this PDE is semilinear (it has bounded nonlinearity) the resonant interactions between the linear frequencies of oscillation produce small divisor problems making the full normal form procedure not convergent.
This approach is well established in the KAM theory for quasi-linear PDEs on the circle (see for instance \cite{KdVAut}, \cite{FGP}).

\item {\bf Defocusing and Focusing equations:} To simplify the exposition, the theorems above only refer to the defocusing equations \eqref{Wave} and \eqref{Beam}. However, it can be checked that the sign of the nonlinearity does not play any role and therefore, Theorems \ref{thm:geometric}, \ref{TeoWaveBeam} and \ref{thm:traveling_periodic_beating} also apply to the focusing equations
 \[
  u_{tt}-\Delta u-u^3=0,\qquad  u_{tt}+\Delta^2 u-u^3=0.
 \]
\end{itemize}

%
%


%

\subsection{Transfer of energy and growth of Sobolev norms}\label{sec:growth}
The solutions of the Wave equation \eqref{Wave}/Beam equation \eqref{Beam}/Hartree equation \eqref{Hartree} obtained in Theorem \ref{thm:traveling_periodic_beating} undergo certain transfer of energy between modes. Unfortunately, such transfer of energy do not lead to growth of Sobolev norms \cite{Bourgain00b, CKSTT, GuardiaK12}.

We would like to devote this section to  relate our results to that of \cite{CKSTT}.
In \cite{CKSTT}, the authors obtain orbits undergoing growth of Sobolev norms for the defocusing nonlinear Schr\"odinger equation on $\mathbb{T}^2$. One of the key points of their proof is to construct, for the Birkhoff normal form truncation, a  chain of invariant tori (periodic orbits in certain symplectic reduction, named \emph{toy model}) which are connected by \emph{non-transverse} heteroclinic orbits.  To obtain such connections, they strongly rely on the following fact. Even if this toy model  is not integrable, it is integrable once restricted to certain invariant subspace (what can be called \emph{two generations model} following \cite{CKSTT}). Then the orbits undergoing growth of Sobolev norms are well approximated by orbits which shadow (follow closely) this chain of periodic orbits.

If one wants to use their ideas to obtain similar behavior in other equations such as the Wave \eqref{Wave}, Beam  \eqref{Beam} and Hartree \eqref{Hartree} equations, one has to face several challenges.

First of all, in these  equations, the two generations model is not integrable (for the Hartree equation it is not for a generic potential). This is not surprising. Indeed, typically (at least in finite dimensional Hamiltonian systems) unstable motion (Smale horseshoes, Arnold diffusion) is related to non-integrability. Still, even if non-integrability should ``help '' to achieve growth of Sobolev norms it makes the analysis considerably more difficult. The present paper is a first attempt to understand this regime (for the two generations model).

The models we consider are carefully chosen so that they are close to integrable and therefore  can be analyzed through perturbative methods. Unfortunately, for the Wave and Beam equation, to be close to integrable we have to choose the modes in $\Lambda$ with very similar modulus and therefore it seems difficult to use the analysis done in this paper to construct orbits undergoing growth of Sobolev norms. For the Hartree equation, one should expect that the ideas developed in this paper could lead to growth of Sobolev norms for a generic potential satisfying \eqref{def:potential}, \eqref{cond:potential}.

A second fundamental difference between NLS and the PDEs considered in this paper is about the chain of tori connected by  heteroclinic connections considered in \cite{CKSTT}. Such structure is not \emph{structurally stable} in the following sense: to have such heteroclinic connections one certainly needs that the connected invariant tori belong to the same level of energy (and to the samel level of other first integrals that the finite dimensional reduction possesses). This does not happen to be the case in other equations besides NLS. Indeed, for the Hartree equation \eqref{Hartree} with a generic potential $V$ the tori considered in \cite{CKSTT} belong to different level of energy and the same happens for the Wave and Beam equations for a generic choice of resonant tuples.

Therefore, to achieve growth of Sobolev norms for those equation one certainly needs to consider other invariant objects. The tori considered in  Theorem \ref{thm:geometric} are radically different from those in \cite{CKSTT}. These tori come in families of higher dimension which are transverse to the first integrals. Moreover,
 they are indeed connected by heteroclinic orbits. These connections are transverse and, therefore, they   are robust. We believe that  such objects could play a role if one wants to implement \cite{CKSTT} to other PDEs.\\

\section*{Acknowledgments}
The authors warmly thank Massimiliano Berti and Michela Procesi for useful discussions and comments.

This project has received funding from the European Research Council (ERC) under the European Union's Horizon 2020 research and innovation programme (grant agreement No 757802). P.M. has been partially funded by the Spanish Government MINECO-FEDER grant PGC2018-100928-B-I00. M. G and S. P. have been also partially supported by
the Spanish MINECO-FEDER Grant PGC2018-098676-B-100 (AEI/FEDER/UE) and by the Catalan Institution for Research
and Advanced Studies via an ICREA Academia Prize 2019.  F.G., M. G, P. M. and S.P. have  been also partially supported by the Catalan grant 2017SGR1049. S.P. acknowledges financial support from the Spanish ``Ministerio de Ciencia, Innovaci\'on y
Universidades'', through the Mar\'{\i}a de Maeztu Programme for Units of Excellence (2015-
2019) and the Barcelona Graduate School of Mathematics.

\label{sec:Heuristics}

\section{Heuristics and description of the paper}\label{sec:heuristics}

The general argument we use in the proofs of Theorems~\ref{TeoWaveBeam} and~\ref{thm:traveling_periodic_beating} follows some of the ideas in the literature~\cite{CKSTT,GuardiaK12,Guardia14,HausProcesi,GuardiaHP16}. The steps are the following. First, a weak Birkhoff normal form procedure simplifies the infinite dimensional Hamiltonian defined by the PDE, removing some non-resonant terms. Second, the normal form is truncated. The truncated normal form admits finite dimensional invariant subspaces. Third, a choice of these subspaces is made, defining a finite dimensional approximation of the PDE, that we call \emph{resonant model}. Some particular finite dimensional orbit of the finite dimensional model is found. Fourth and final, a true solution of the original PDE, close to the finite dimensional one for long enough time, is found. We will use this scheme, with particular choices in each step, particularly when considering the finite dimensional model.

In \cite{CKSTT, HausProcesi}, in the third step, the particular orbit found in the resonant model is obtained relying on the fact that the resonant model is integrable. More precisely, some invariant manifolds of different hyperbolic objects, coincide. Our approach is essentially different, because our resonant model is non-integrable in the sense that the invariant manifolds of several invariant objects - fixed points or periodic orbits - intersect transversally. We take advantage of the non-integrable dynamics of the finite dimensional model to obtain solutions of the truncated normal form with prescribed behavior;
indeed, non-integrable dynamics is richer than the integrable one. More details are given below.

As a matter of fact, the proofs of Theorems~\ref{TeoWaveBeam} and~\ref{thm:traveling_periodic_beating} share all these common ingredients and only differ in the  finite dimensional phenomena arisen by  non-integrability.

Let us give more details concerning our implementation of the strategy.

\paragraph{Step 1: }Each of the PDEs under consideration has a Hamiltonian structure. Let us denote by $H$  the Hamiltonian.
Given a \emph{complete} (see Definition~\ref{def:complete}) finite subset $\Lambda\subset\mathbb{Z}^2$ of resonant modes, to be chosen later, a \emph{weak normal form} scheme is applied to the Hamiltonian. This weak normal form only ``removes'' a finite number of monomials of degree $4$ of the Hamiltonian. Hence it is well defined (the normal form transformation is defined by the flow of a system of ODEs). The monomials to be killed are related to the set $\Lambda$. Although the normal form procedure is not complete and many non-resonant terms of degree $4$ are left untouched, for suitable $\Lambda$  the truncated normal form admits a finite dimensional invariant subspace supported on $\Lambda$. This is done in Section~\ref{sec:weak_normal_form}.

Once the Hamiltonian is written in the normal form coordinates, we consider the truncated normal form, disregarding the terms of degree $6$ or more. We call this truncated normal form the \emph{resonant model}.

\paragraph{Step 2: } This step is the core of the paper and can be divided as follows.
\begin{itemize}
 \item \emph{Construction of the Set $\Lambda$} (Section \ref{sec:resonantmodel4}): The set $\Lambda$ is chosen in such a way that its associated subspace of modes (see $V_{\Lambda}$ in \eqref{def:splitting}) is invariant by the flow of the resonant model, but of course satisfies other requirements. Its precise definition depends on the PDE model we consider, but all three instances (Wave, Beam and Hartree equations) of the set $\Lambda$ share some common features. They have exactly $4N$ elements which, using the terminology introduced in~\cite{CKSTT}, encompass two generations. The elements of the set $\Lambda$ are organized in groups of four,
pairwise disjoint, each of them forming a parallelogram. The choice of the modes is such that each individual parallelogram
is invariant. It also happens that the dynamics of a single parallelogram is integrable, that is, if the rest of the modes are at $0$, the dynamics of the four modes in a parallelogram is integrable. At this point is where our choice of the modes differs from other examples in the literature.

\item \emph{The dynamics of the finite dimensional model} (Sections \ref{sec:resonantmodel}, \ref{sec:Proofthm1}, \ref{sec:transfer}):
First, we choose our modes in such a way that the dynamics of the resonant model is close to integrable, where closeness to integrability is measured through some parameter $\varepsilon$.  The nearly integrability is obtained choosing properly the modes in $\Lambda$. The unperturbed system (where $\varepsilon = 0$) possesses certain invariant objects, namely  hyperbolic fixed points and hyperbolic periodic orbits, whose invariant manifolds form heteroclinic or homoclinic separatrices. Our second (generic) condition on the modes is sufficient to ensure that these heteroclinic or homoclinic manifolds split for small $\varepsilon \neq 0$, giving rise to horseshoes and instability phenomena from which we deduce the existence of certain types of orbits. The splitting of these manifolds is measured by means of a suitable set of Melnikov integrals \cite{Melnikov63}.

\item \emph{The infinite symbols Smale horseshoe} (Section \ref{sec:Proofthm1}): The orbits in Theorem~\ref{TeoWaveBeam} give rise from a horseshoe of infinite symbols that can be constructed close to a hyperbolic periodic orbit whose invariant manifolds intersect transversally. The construction of this horseshoe follows the ideas in~\cite{Moser01}. In this horseshoe, each symbol encodes the time to pass close to the periodic orbit, which then becomes random. The horseshoe can be described as follows. Let $\Gamma=\{1,2,3, \dots\}$ be a denumerable set of symbols and
\[
\Sigma = \{s= (\dots, s_1, s_0, s_1, \dots) \mid s_i \in \Gamma, \; i \in \N\},
\]
the space of bi-infinite \emph{sequences}, with the product topology. Notice that, unlike what happens when $\Gamma$ is a finite set, $\Sigma$ is not compact. The shift $\sigma: \Sigma \to \Sigma$ is the homeomorphism
on $\Sigma$ defined by $(\sigma(s))_{i} = s_{i-1}$. Following the construction of Moser in~\cite{Moser01}, given a hyperbolic periodic orbit whose invariant manifolds intersect transversally, it is possible to find a set of coordinates --- one of the coordinates is time, in $\T$ ---, a suitable section $\mathcal{S}$ that defines a return map $\phi$ and a set $Q$ in this section with $\phi(Q) = Q$, such that there exists a homeomorphism $\tau: \Sigma \to Q$ satisfying $\phi \circ \tau = \tau \circ \sigma$. The set $Q$ is in fact the intersection of forward and backward images by $\phi$ of a set of disjoint closed bands $\{\mathcal{V}_j, \; j \in \N\}$, where the index $j$ denotes precisely the time between to consecutives passes through $\mathcal{S}$ and hence measures the distance to one of the invariant manifolds of the set $\mathcal{V}_j$. In this way, $\mathcal{V}_j$ tends to the invariant manifold when $j$ tends to infinity. The set $Q$ is not compact because the return map is not defined in the invariant manifolds.

\item \emph{Shadowing of a sequence of periodic orbits} (Section \ref{sec:transfer}): the orbits in Theorem~\ref{thm:traveling_periodic_beating}
travel along a chain of periodic orbits connected by transverse heteroclinic orbits, following the diffusion mechanism
described originally by Arnold~\cite{Arnold64}. This mechanism consists of a sequence - finite or infinite - of partially hyperbolic periodic orbits\footnote{These periodic orbits are not fully hyperbolic since the system is Hamiltonian: the tangent to the periodic orbit and its conjugate direction are not hyperbolic.}, $\{\TT_i\}_{i\in I}$, $I\subset\mathbb{N}$, such that the unstable manifold of $\TT_i$, $W^u(\TT_i)$, intersects \emph{transversally} the stable manifold of $\TT_{i+1}$, $W^s(\TT_{i+1})$. Here, since the system we are considering is autonomous, \emph{transversally} means transversality in the energy level, which  implies that the intersection of the manifolds is, locally, a single heteroclinic orbit. If a nondegeneracy condition is met, this transversality is sufficient to have a Lambda Lemma that implies that
$W^u(\TT_{i+1}) \subset \overline{W^u(\TT_{i})}$ (see~\cite{FontichM01}), which in turn implies that for any $i,j \in I$,
$i< j$, $W^u(\TT_{j}) \subset \overline{W^u(\TT_{i})}$. One can then choose arbitrary small neighborhoods of the tori $\TT_i$ and orbits that visit these neighborhoods according to an increasing sequence of times.

It is worth to remark that the orbits found in the resonant model do exist for any positive time. In the case of the horseshoe with infinite symbols, one obtains orbits that arrive closer and closer to the periodic orbit, in randomly chosen times. In the case  of the diffusion orbits, one obtain solutions that wander along the chain of periodic orbits for any positive time, and can be chosen to arrive closer and closer to each periodic orbit.
\end{itemize}

\paragraph{Step 3:} The last step of the proof consists in finding a true solution of each PDE shadowing for long enough time the chosen solution of the resonant model. This is accomplished by a standard Gronwall and bootstrap argument. This relies on the Approximation argument given in Section~\ref{sec:weak_normal_form}  with the analysis of the dynamics of the Birkhoff normal form truncation of  Sections ,\ref{sec:Proofthm1}, \ref{sec:transfer}.

\section{Weak Birkhoff normal form}
\label{sec:weak_normal_form}

\subsection{Hamiltonian formalism}
In this section we show that
the Hamiltonian PDEs \eqref{Wave}, \eqref{Beam} and \eqref{Hartree} have a
Hamiltonian of the same form in an appropriate set of coordinates. We consider
spaces of functions defined on $\mathbb{T}^2$, hence it is convenient to use the
Fourier representation $u(x)=\sum_{j\in\mathbb{Z}^2} u_j\,e^{\mathrm{i} j \cdot
x}$.\\
Let us denote by $\mathcal{P}$ the phase space and $\Omega$ a symplectic form on it. The vector field $X_H$ of a Hamiltonian $H$ is uniquely determined by the formula $d H(u)[\cdot]=\Omega(X_H(u), \,\cdot)$.
\paragraph{Hamiltonian structure of equation \eqref{Hartree}}
Let us consider $\mathcal{P}:=H^1(\mathbb{T}^2; \mathbb{C})\times
H^1(\mathbb{T}^2; \mathbb{C})$ equipped with the symplectic form
$\Omega:=\mathrm{i} du \wedge d \bar{u}=\mathrm{i} \sum_{j\in\mathbb{Z}^2} d
u_j\wedge d\overline{u_j}$. If $V$ satisfies \eqref{def:potential}, the
equation \eqref{Hartree} is given by $\partial_tu=X_H(u,\overline{u})$
where\begin{equation}\label{def:HartreeHam}
\begin{aligned}
H(u, \bar{u}) &=\frac{1}{(2\pi)^2}\left(\int_{\mathbb{T}^2} |\nabla u|^2\,dx+
\frac{1}{2}\int_{\mathbb{T}^2}(V(x)\star \lvert u \rvert^2)\,\lvert u
\rvert^2\,dx\right)\\
&=\sum_{j\in\mathbb{Z}^2} \lvert j \rvert^2 \lvert u_j \rvert^2+\sum_{j_1-j_2+j_3- j_4=0} V_{j_1-j_2} u_{j_1}\,\overline{u_{j_2}}\, u_{j_3}\,\overline{u_{j_4}}.
\end{aligned}
\end{equation}

\paragraph{Hamiltonian structure of equations \eqref{Wave}, \eqref{Beam}} In the following we use the parameter $\kappa\in \{ 1, 2\}$ to treat both cases at the same time. More precisely, $\kappa=1$ if we refer to the Wave equation \eqref{Wave} or $\kappa=2$ when we consider the Beam equation \eqref{Beam}.
By setting $v:=\dot{u}$, we can express
these equations as the following system of two first order equations
\begin{equation}\label{systemwave}
\begin{cases}
\dot{u}=v,\\
\dot{v}=(-1)^{\kappa+1}\Delta^{\kappa} u-u^3.
\end{cases}
\end{equation}
We recall the subset $\mathbb{Z}^2_{\odd}:=\{ (j^{(1)},
j^{(2)})\in\mathbb{Z}^2 \,: \,\,
j^{(1)}\,\,\mbox{odd}\,\,,\,\,j^{(2)}\,\,\mbox{even} \}$. The subspace
\begin{equation}
\mathcal{U}_{\odd}:=\{ (u, v)\in H^{\kappa}(\mathbb{T}^2;
\mathbb{R})\times L^2(\mathbb{T}^2; \mathbb{R}),
\,\,u=\sum_{j\in\mathbb{Z}^2_{\odd}} u_j\,e^{\mathrm{i}\,j \cdot x},
\,\,v=\sum_{j\in\mathbb{Z}^2_{\odd}} v_j\,e^{\mathrm{i}\,j \cdot x} \}
\end{equation}
is invariant for \eqref{systemwave}.
Since $(0, 0)\notin \mathbb{Z}^2_{\odd}$ the change of variables $\Xi (u,
v)=(\Psi, \overline{\Psi})$ defined by
\begin{equation}\label{def:complexcoordinates}
\Psi:=\frac{1}{\sqrt{2}} \Big( |D|^{\kappa/2}
u-\mathrm{i} |D|^{-\kappa/2} v \Big), \qquad \overline{\Psi}:=\frac{1}{\sqrt{2}}  \Big(|D|^{\kappa/2} u+\mathrm{i} |D|^{-\kappa/2} v
\Big) \qquad \lvert D \rvert:=(-\Delta)^{1/2},
\end{equation}
is well defined  on $\mathcal{U}_{{\odd}}$ and it transforms the
system \eqref{systemwave} into the following one
\begin{equation}\label{def:newsystemwave}
\begin{cases}
-\mathrm{i}\dot{\Psi}=|D|^\kappa \Psi+\frac{1}{4} |D|^{-\kappa/2}\left( \left(
|D|^{-\kappa/2}\left( \Psi+\overline{\Psi} \right) \right)^3 \right)\\[2mm]
\mathrm{i}\dot{\overline{\Psi}}=|D|^\kappa \overline{\Psi}+\frac{1}{4}
|D|^{-\kappa/2}\left( \left( |D|^{-\kappa/2}\left( \Psi+\overline{\Psi} \right) \right)^3 \right).
\end{cases}
\end{equation}
The vector field in \eqref{def:newsystemwave} is Hamiltonian with respect to the $2$-form $\Omega:=\mathrm{i} d\Psi\wedge d\overline{\Psi}$ and Hamiltonian
\[
H(\Psi,\overline{\Psi}):=\frac{1}{(2\pi)^2}\,\left[ \int_{\mathbb{T}^2} |D|^{\kappa}
\Psi\,\overline{\Psi}\,dx+\frac{1}{4}\int_{\mathbb{T}^2}
\left(|D|^{-\kappa/2}\left(\frac{\Psi+\overline{\Psi}}{\sqrt{2}}\right)
\right)^4\,dx \, \right].
\]
  By considering the Fourier expansion
$
\Psi=\sum_{j\in\mathbb{Z}^2} a_j\,e^{\mathrm{i} j\cdot x},
$
we can consider $\Omega=\mathrm{i} \sum_{j\in\mathbb{Z}^2_{\odd}} d a_j\wedge d
\overline{a_j}$ and
\begin{equation}\label{def:hamwavebeam}
\begin{aligned}
H &=\sum_{j\in\mathbb{Z}^2_{\odd}} |j|^{\kappa}
\,a_j\,\overline{a_j}+\frac{1}{16} \sum_{\substack{j_i\in\mathbb{Z}^2_{\odd},\\
j_1+j_2+j_3+j_4=0}} \frac{(a_{j_1}+\overline{a_{-j_1}})
(a_{j_2}+\overline{a_{-j_2}})
(a_{j_3}+\overline{a_{-j_3}})(a_{j_4}+\overline{a_{-j_4}}) }{(|j_1|
\,|j_2|\,|j_3|\,|j_4|)^{\kappa/2}}.
\end{aligned}
\end{equation}
%
We observe that the Hamiltonians \eqref{def:HartreeHam} and \eqref{def:hamwavebeam} have the form\footnote{Here
we use the standard notation $a_j^+=a_j$, $a_{j}^-= \overline{a_{-j}}$.}
\begin{equation}\label{def:commonham}
H=H^{(2)}+H^{(4)}=\sum_{j\in\mathbb{Z}^2_*} \omega(j)\, a_j\,\overline{a_j}+\sum_{\substack{j_i\in\mathbb{Z}^2_*, \sigma_i\in\{ \pm \},\\ \sigma_1 j_1+\sigma_2 j_2+\sigma_3 j_3+\sigma_4 j_4=0}} C^{\sigma_1 \sigma_2 \sigma_3 \sigma_4}_{j_1 j_2 j_3 j_ 4}\, a_{j_1}^{\sigma_1}\,a_{j_2}^{\sigma_2}\,a_{j_3}^{\sigma_3}\,a_{j_4}^{\sigma_4},
\end{equation}
where
\begin{itemize}
\item (Hartree): $\mathbb{Z}^2_*=\mathbb{Z}^2$, $\omega(j)=\lvert j \rvert^2$, and the coefficients $C^{\sigma_1 \sigma_2 \sigma_3 \sigma_4}_{j_1 j_2 j_3 j_ 4}$ are defined as
\begin{equation}\label{def:Cj1}
C^{+ - + -}_{j_1 j_2 j_3 j_4}=C^{-+ - +}_{j_1 j_2 j_3 j_4}=V_{j_1-j_2}, \qquad C^{\sigma_1 \sigma_2 \sigma_3 \sigma_4}_{j_1 j_2 j_3 j_ 4}=0 \,\,\,\mbox{otherwise}.
\end{equation}
\item  ($\kappa=1$ Wave,  $\kappa=2$ Beam):
$\mathbb{Z}^2_*=\mathbb{Z}^2_{\odd}$, $\omega(j)=\lvert j \rvert^{\kappa}$ and, for 
 $(j_1, j_2, j_3, j_4) $ such that $\sigma_1 j_1+\sigma_2 j_2+\sigma_3 j_3+\sigma_4j_4=0$, we have
\begin{equation}\label{def:Cj2}
\begin{aligned}
&C^{\sigma_1 \sigma_2 \sigma_3 \sigma_4}_{j_1 j_2 j_3 j_ 4}=\frac{1}{16 \left( |j_1|| j_2
|| j_3|| j_ 4| \right)^{\kappa/2}}.
\end{aligned}
\end{equation}
%
\end{itemize}
We remark that (using \eqref{cond:potential} for the Hartree equation) the coefficients $C^{\sigma_1\sigma_2\sigma_3\sigma_4}_{j_1 j_2 j_3 j_4}$ are such that
\begin{equation}\label{def:unifbdd}
\sup_{j_1, j_2, j_3, j_4\in\mathbb{Z}_*^2} | C^{\sigma_1\sigma_2\sigma_3\sigma_4}_{j_1 j_2 j_3 j_4} |\le 2.
\end{equation}

\subsection{Weak Birkhoff normal form}

In this section we apply a Birkhoff normal form argument to the Hamiltonian \eqref{def:commonham}.
We consider the symplectic form $\Omega=\mathrm{i} \sum_{j\in\mathbb{Z}_*^2} d a_j\wedge d \overline{a_j}$. \\
We denote by $\ad_{H^{(2)}}$ the adjoint action of the Hamiltonian $H^{(2)}$. If
$F=\sum_{\sigma_1 j_1+\dots+\sigma_n j_n=0} F^{\sigma_1\dots \sigma_n}_{j_1\dots
j_n} a_{j_1}^{\sigma_1}\dots a_{j_n}^{\sigma_n}$ is a homogenous, momentum
preserving Hamiltonian of degree $n$ we have then
\[
\ad_{H^{(2)}}[F]:=\{ H^{(2)}, F \}=\sum_{\sigma_1 j_1+\dots+\sigma_n j_n=0}
\left(\sum_{i=1}^n \sigma_i \omega(j_i) \right)  F^{\sigma_1\dots
\sigma_n}_{j_1\dots j_n} a_{j_1}^{\sigma_1}\dots a_{j_n}^{\sigma_n}.
\]
We denote by $\Pi_{\Ker(H^{(2)})}$ the projection on the kernel of
$\ad_{H^{(2)}}$. The $n$-tuples $(\sigma_i, j_i)_{i=1}^n$ such that
\[
\sum_{i=1}^n \sigma_i \omega(j_i)=0, \qquad \sum_{i=1}^n \sigma_i \,j_i=0
\]
 are called $n$-resonances. Since there are no regularity issues in what follows
we decide to work on the phase space of analytic sequences. We fix $\rho>0$ and define
\[
W_{\rho}:=\left\{ a=(a_j)_{j\in\mathbb{Z}^2_*}\in\ell^1 : \lVert a
\rVert_{\rho}:=\sum_{j\in\mathbb{Z}^2_*} \lvert a_j \rvert\,e^{\rho |j |}
<\infty \right\}.
\]
We denote by $\mathcal{B}_{\rho}(\delta)$ the open ball of radius $\delta>0$
centered at the origin of $W_{\rho}$. We use the notation $A \lesssim B$ to denote $A\le C\, B$ where $C>0$ is a constant possibly depending on the fixed $\rho$.

Let $\Lambda$ be a finite subset of $\mathbb{Z}_*^2$. We consider the following
splitting $W_{\rho}=V_{\Lambda}+Z_{\Lambda}$ with
\begin{equation}\label{def:splitting}
V_{\Lambda}:=V_{\Lambda, \rho}=\left\{ a\in W_{\rho} : a_j=0\,\,\mbox{if} \,\,j\notin
\Lambda \right\}, \quad Z_{\Lambda}:=Z_{\Lambda, \rho}=\left\{  a\in W_{\rho}:
a_j=0\,\,\mbox{if} \,\,j\in \Lambda \right\}.
\end{equation}
We define
\[
\mathcal{S}_{n, k}:=\left\{(\sigma_i, j_i)_{i=1}^n :\,\,\sum_{i=1}^n \sigma_i j_i=0\,\,\mbox{such that the number of indices} \,\,j_i\notin\Lambda\,\, \mbox{is exactly}\,\, k\right\}.
\]
Given a homogenous $n$-degree, momentum preserving Hamiltonian $F=\sum_{\sigma_1 j_1+\dots+\sigma_n j_n=0} F^{\sigma_1\dots \sigma_n}_{j_1\dots j_n} a_{j_1}^{\sigma_1}\dots a_{j_n}^{\sigma_n}$, we denote by $F^{(n, k)}$ the projection of $F$ onto the monomials $a^{\sigma_1}_{j_1}\dots a_{j_n}^{\sigma_n}$ with exactly $k$ indices $j_i\notin \Lambda$. Thus $F^{(n, k)}$ is the part of the Hamiltonian $F$ which is Fourier supported on $\mathcal{S}_{n, k}$.\\
We denote by
\[
\mathcal{S}_{n, \le k}:=\cup_{i=1}^k \mathcal{S}_{n, i}, \quad \mathcal{S}_{n, \geq k}:=\cup_{i=k}^n \mathcal{S}_{n, i}.
\]
We refer to $F^{(n, \le k)}$ (and $F^{(n, \geq k)}$) the part of the Hamiltonian $H$ which is Fourier supported on $\mathcal{S}_{n, \le k}$ (and $\mathcal{S}_{n, \geq k}$).
\begin{remark}\label{rem:finiteham}
Since we assume that $\Lambda$ is finite,  the preservation of
momentum implies that the Hamiltonians $F^{(n, \le 1)}$ have compact Fourier
support.
\end{remark}
\begin{defi}
\label{def:complete}
We say that a subset $\Lambda\subset \mathbb{Z}^2$ is \emph{complete} if the following holds: given a $4$-resonance $(\sigma_i, j_i)_{i=1}^4$ we have that if $j_1, j_2, j_3\in \Lambda$ then $j_4\in\Lambda$.
\end{defi}

\begin{prop}[Weak Birkhoff normal form]\label{WBNF}
Fix $\rho>0$. Let $\Lambda\subset\mathbb{Z}^2_*$ be finite and complete
and consider the Hamiltonian $H$ in \eqref{def:commonham}.
Then,
\begin{itemize}
\item[$(i)$] There exists $\delta_1>0$ small such that $\forall \delta\in (0, \delta_1)$ there exists an analytic change of coordinates
$\Gamma\colon \mathcal{B}_{\rho}(\delta)\subset W_\rho\to
\mathcal{B}_{\rho}(2\delta)$ such that
\begin{equation}\label{birkhoffcoord}
H\circ \Gamma=H^{(2)}+\Pi_{\Ker(H^{(2)})}H^{(4, 0)}+H^{(4, \geq 2)}+\mathcal{R}
\end{equation}
where $\mathcal{R}$  satisfies
\[
\lVert X_{\mathcal{R}}(a) \rVert_{\rho}\lesssim \lVert a
\rVert_{\rho}^5\,\qquad \text{ for all }\quad a\in\mathcal{B}_{\rho}(\delta).
\]
\item[$(ii)$] Moreover, the map $\Gamma$ is close to the identity, i.e. $\lVert
\Gamma(a)-a\rVert_{\rho}\lesssim \lVert a \rVert_{\rho}^3\,$ for all
$a\in\mathcal{B}_{\rho}(\delta)$.
\end{itemize}
\end{prop}
\begin{proof}
Let us consider the $4$-degree homogenous Hamiltonian
\[
F=\sum_{\substack{(\sigma_i, j_i)\in \mathcal{S}_{4, \le 1},\\ \sigma_1 j_1+\sigma_2 j_2+\sigma_3 j_3+\sigma_4 j_4=0}} F_{j_1 j_2 j_3 j_4}^{\sigma_1 \sigma_2 \sigma_3 \sigma_4} \,a^{\sigma_1}_{j_1}\,a^{\sigma_2}_{j_2}\,a^{\sigma_3}_{j_3}\,a^{\sigma_4}_{j_4}
\]
where
\begin{equation}\label{omologeq}
F_{j_1 j_2 j_3 j_4}^{\sigma_1 \sigma_2 \sigma_3 \sigma_4}=\begin{cases}
-\frac{\mathrm{i} C^{\sigma_1 \sigma_2 \sigma_3 \sigma_4}_{j_1 j_2 j_3 j_ 4}}{\sigma_1\omega(j_1)+\sigma_2\omega(j_2)+\sigma_3\omega(j_3)+\sigma_4\omega(j_4)} \qquad \sigma_1\omega(j_1)+\sigma_2\omega(j_2)+\sigma_3\omega(j_3)+\sigma_4\omega(j_4) \neq 0\\
0 \qquad \qquad \qquad\qquad  \quad\qquad \ \qquad \qquad   \sigma_1\omega(j_1)+\sigma_2\omega(j_2)+\sigma_3\omega(j_3)+\sigma_4\omega(j_4) = 0.
\end{cases}
\end{equation}
The function $F$ solves the homological equation
\begin{equation}\label{homoeq}
\{ H^{(2)}, F\}+H^{(4, \le 1)}=\Pi_{\Ker(H^{(2)})} H^{(4, \le 1)}.
\end{equation}
 By Remark \ref{rem:finiteham} the vector field generated by $F$ has just a finite number of non zero components. Hence $\Phi_F^t$ is the flow of a ODE with a smooth vector field. 
 We call $\Gamma=(\Phi_F^t)_{|_{t=1}}$ the time-one flow map of $F$.\\
 By Remark \ref{rem:finiteham} the denominators in \eqref{omologeq} have a uniform lower bound, hence by \eqref{def:unifbdd} the coefficients defined in \eqref{omologeq} are uniformly bounded. Then by Young's inequality it is easy to see that $\lVert X_F(a) \rVert_{\rho}\lesssim \lVert a \rVert^3_{\rho} $ for all $a\in W_{\rho}$.\\
This implies that for $\delta>0$ small enough $\lVert \Phi_F^t(a)\rVert_{\rho}\le 2 \lVert a \rVert_{\rho}$ for all $t\in[0, 1]$. Thus $\Gamma$ maps $\mathcal{B}_{\rho}(\delta)$ to $\mathcal{B}_{\rho}(2\delta)$ and
\[
\lVert \Gamma(a)-a \rVert_{\rho}\le \sup_{s\in [0, 1]}\lVert X_F(\Phi_F^s(a)) \rVert_{\rho}\le \sup_{s\in[0, 1]} \lVert \Phi_F^s(a) \rVert^3_{\rho}\lesssim \lVert a \rVert_{\rho}^3.
\]
So we have proved item $(ii)$. After the change of coordinates $\Gamma$ the Hamiltonian \eqref{def:commonham} transforms into
\begin{align*}
H\circ \Gamma &=H+\{ H, F\}+\int_0^1 (1-t) \{\{ H, F\}, F\}\circ \Phi^t_F\,dt\\
&=H^{(2)}+\Big(H^{(4, \le 1)}+ \{ H^{(2)}, F \}\Big)+H^{(4, \geq 2)}+\{ H^{(4)}, F \}+\int_0^1 (1-t) \{\{ H, F\}, F\}\circ \Phi^t_F\,dt.\\
&\stackrel{ \eqref{homoeq}}{=}H^{(2)}+\Pi_{\Ker(H^{(2)})} \,H^{(4, \le
1)}+H^{(4, \geq 2)}+\{H^{(4)}, F\}+\int_0^1 (1-t) \{\{ H, F\}, F\}\circ
\Phi^t_F\,dt.
\end{align*}
Then, the completeness of $\Lambda$ implies that
\[
 \Pi_{\Ker(H^{(2)})}H^{(4, \leq 1)}=\Pi_{\Ker(H^{(2)})}H^{(4, 0)}.
\]
Moreover, we can take $\mathcal{R}:=\{H^{(4)}, F\}+\int_0^1 (1-t) \{\{ H, F\},
F\}\circ \Phi^t_F\,dt$. We observe that $\{ H^{(4)}, F\}$ is a homogenous
Hamiltonian of degree $6$ . Regarding the integral term, we have that $\Phi^t_F$
is smooth and $\{\{ H, F \}, F\}$ is the sum of two homogenous Hamiltonians of
degree at least $6$, hence it is an analytic function on
$\mathcal{B}_{\rho}(\delta)$ that can be Taylor expanded at $a=0$. The first
term of the Taylor expansion of the vector field is a polynomial of degree $5$
and the remainder is smaller in a sufficiently small neighborhood of the origin.
 Again, by the uniform boundness of the coefficients of $H$ and $F$,
one can obtain the estimate in item $(i)$ by using Young's inequality. 
\end{proof}

Let us consider the time-dependent change of coordinates
\begin{equation}\label{rotatingcoord1}
\Psi(t):\quad \mathtt{a}_j \to a_j=\mathtt{a}_j\,e^{\mathrm{i} \omega(j)\,t},
\quad \mathtt{a}=(\mathtt{a}_j)_{j\in \mathbb{Z}_*^2}\in W_{\rho}.
\end{equation}
%
This change of coordinates leaves resonant monomials unchanged, that is 
$\Pi_{\Ker(H^{(2)})} F\circ \Psi=\Pi_{\Ker(H^{(2)})} F$.
Then, we have that
\begin{equation}\label{Hres}
\begin{aligned}
&H\circ\Gamma\circ\Psi=H_{\Res}+\mathcal{R}'(t), \qquad
H_{\Res}:=\Pi_{\mathrm{Ker}(H^{(2)})} \,H^{(4, 0)}+\mathcal{Q}(t)
\\
 & \mathcal{Q}:=H^{(4, \geq 2)}\circ \Psi(t), \quad
\mathcal{R}'=\mathcal{R}\circ \Psi(t).
\end{aligned}
\end{equation}
%
Moreover, the functions $\mathcal{Q}$ and $\mathcal{R}'$ satisfy
\begin{equation}\label{eq:EstimatesQRprime}
\lVert X_{\mathcal{Q}}(a) \rVert_{\rho}\lesssim \lVert a
\rVert_{\rho}^3
\qquad \text{ and }\qquad 
\lVert X_{\mathcal{R}'}(a) \rVert_{\rho}\lesssim \lVert a
\rVert_{\rho}^5\,\qquad \text{ for all }\quad a\in\mathcal{B}_{\rho}(\delta).
\end{equation}

Now, if one considers a complete set $\Lambda$ (see Definition
\ref{def:complete}), the associated subspace $V_{\Lambda}$ (see
\eqref{def:splitting})   is
left invariant by $X_{H_{\Res}}$. Moreover, on $V_\Lambda$,
$X_{H_{\Res}}=X_{\Pi_{\mathrm{Ker}(H^{(2)})} \,H^{(4, 0)}}$. This Hamiltonian
is scaling invariant in the sense that if $r(t)$ is a trajectory of this vector
field
\begin{equation}\label{def:scaling}
 r^\de(t)=\de r(\de^2 t)
\end{equation}
also is. Taking $\de\ll 1$, in certain time
scales,  this trajectory  $r^\de(t)$ stays close to the trajectory
of the Hamiltonian \eqref{Hres}  with the same initial condition.

\begin{prop}\label{approximationargument}
Let $T_0$ be a positive number and consider a solution $r(t)$  of the
Hamiltonian system $H_{\Res}$  in \eqref{Hres} such that it is defined for $t\in [0,T_0]$ and $r(0)\in
V_{\Lambda}$.
%

Then there exists $\delta_2=\delta_2(T_0)\le \delta_1$ (where $\delta_1$ is given in Proposition \ref{WBNF}) such that the following holds:
for all $0<\delta\le \delta_2$ the rescaled solution $r^\de$ of the Hamiltonian
$H_{\Res}$ given by \eqref{def:scaling} and the solution
$u(t)$ of the Hamiltonian system
$H_{\Res}+\mathcal{R}'$ with initial condition $u(0)=r^{\delta}(0)=\de r(0)$
satisfy
\begin{equation}
\lVert r^{\delta}(t)-u(t) \rVert_{\rho}\lesssim \delta^2 \qquad \forall t\in [0,  \delta^{-2} T_0].
\end{equation}
\end{prop}
\begin{proof}
We define $\xi:=u-r^{\delta}$. Then, $\xi$ satisfies
\[
 \dot{\xi}=\mathcal{Z}_0(t)+\mathcal{Z}_1(t)\xi+\mathcal{Z}_{21}(t, \xi)+\mathcal{Z}_{22}(t, \xi)
 \]
where
\begin{align*}
& \mathcal{Z}_0(t):=X_{\mathcal{R}'}(r^{\delta}(t)),\\
&\mathcal{Z}_1(t)=D X_{H_{\Res}}(r^{\delta}(t)),\\
&\mathcal{Z}_{21}(t, \xi):=X_{H_{\Res}}(r^{\delta}(t)+\xi(t))-X_{H_{\Res}}(r^{\delta}(t))-DX_{H_{\Res}}(r^{\delta}(t))\xi,\\ &\mathcal{Z}_{22}(t, \xi):=X_{\mathcal{R}'}(r^{\delta}(t)+\xi(t))-X_{\mathcal{R}'}(r^{\delta}(t)).
\end{align*}
By the estimate in item $(i)$ of Proposition \ref{WBNF}, the fact that $H_{\Res}$ is a homogenous Hamiltonian of degree $4$ and \eqref{eq:EstimatesQRprime}  we have the following estimates
\begin{equation}\label{bounds}
\begin{aligned}
&\lVert \mathcal{Z}_0(t) \rVert_{\rho}\lesssim \lVert r^{\delta} \rVert_{\rho}^5 , \quad  \lVert \mathcal{Z}_1(t) \rVert_{\rho}\lesssim \lVert r^{\delta} \rVert_{\rho}^2 \lVert \xi \rVert_{\rho},\\
&\lVert \mathcal{Z}_{21}(t, \xi) \rVert_{\rho}\lesssim  \lVert r^{\delta} \rVert_{\rho} \lVert \xi \rVert^2_{\rho}+\lVert \xi \rVert^3_{\rho}, \qquad \lVert \mathcal{Z}_{22}(t, \xi) \rVert_{\rho}\lesssim \lVert r^{\delta} \rVert^4_{\rho} \lVert \xi \rVert_{\rho}.
\end{aligned}
\end{equation}
Now we use a bootstrap argument to conclude the proof. We assume temporarily that $\lVert \xi(t) \rVert\lesssim \delta^2$ for $t\in[0, \delta^{-2} T_0]$. We already know that this is true for $t=0$ since $\xi(0)=0$. Then by Minkowsky inequality, the fact that $\lVert r^{\delta} \rVert_{\rho}\lesssim \delta$ and \eqref{bounds} we have that
\begin{equation}\label{stimaderiv}
\frac{d}{d t}\lVert \xi \rVert_{\rho}\le \lVert \mathcal{Z}_0(t)
\rVert_{\rho}+\lVert \mathcal{Z}_1(t)\xi \rVert_{\rho}+\lVert\mathcal{Z}_{21}(t,
\xi)\rVert_{\rho}+\lVert\mathcal{Z}_{22}(t, \xi)\rVert_{\rho}
\lesssim \delta^5+\delta^2 \lVert \xi \rVert_{\rho}.
\end{equation}
Thus integrating \eqref{stimaderiv} and by using Gronwall lemma we get
\[
\lVert \xi(t) \rVert_{\rho}\lesssim \delta^5 t+\delta^7 e^{\delta^2 t} \int_0^t
s\,e^{-\delta^2\,s}\,ds= \delta^5 t+\delta^3 e^{\delta^2 t} (1-e^{-\delta^2
t}(1+\delta^2 t))
\]
and since $0<t\le \delta^{-2} T_0$ we have
\[
\lVert \xi(t) \rVert_{\rho}\lesssim  \delta^3 T_0+\delta^3 e^{ T_0}.
\]
Since $T_0$ is independent from $\delta$, we can choose $\delta$ small enough
such that $\lVert \xi(t) \rVert_{\rho}\lesssim\delta^{5/2}$ for $t\in [0,
\delta^{-2} T_0]$. Since this bound is stronger than the bootstrap assumption we
can drop such hypothesis and the proof is concluded.
\end{proof}


\section{Reduction to the resonant model}\label{sec:resonantmodel4}

\subsection{Lambda set}

We introduce a suitable finite and complete (see Definition \ref{def:complete})
resonant set of modes  $\Lambda\subset\mathbb{Z}^2$, whose construction is based
on the ideas of \cite{CKSTT}.  This set is constructed such that the associated
subspace
\begin{equation}\label{def:SLambda}
V_{\Lambda}:=\{ a\in W_{\rho} : a_j=0 \quad \forall j\notin\Lambda\}
\end{equation}
is invariant under the flow associated to the Hamiltonian $H_\Res$ in
\eqref{Hres}. Later on we study the dynamics of the Hamiltonian $H_{\Res}$
restricted to initial data supported on $V_{\Lambda}$.

First we introduce the set of resonant tuples for the
nonlinear Beam equation (those of the Hartree equation are a subset of it),
\begin{align}
\AAA_{bh} &:= \left\{ (n_1,n_2,n_3,n_4) \in (\mathbb{Z}^2)^4:
n_1 \pm n_2\pm n_3\pm n_4=0, |n_1|^2\pm |n_2|^2\pm |n_3|^2\pm |n_4|^2=0
\right\}, \label{def:Abh}
\end{align}
and a subset of it, which is going to be used to build the set $\Lambda$,
\begin{align}
\widetilde\AAA_{bh} &:= \left\{ (n_1,n_2,n_3,n_4) \in (\mathbb{Z}^2)^4:
n_1-n_2+n_3-n_4=0, |n_1|^2-|n_2|^2+|n_3|^2-|n_4|^2=0 \right\}. \label{def:Abhwt}
\end{align}
Analogously, one can define the resonant tuples for the Wave equation and the
corresponding associated subset
\begin{equation}\label{def:Aw}
\begin{aligned}
\AAA_w &:= \{ (n_1,n_2,n_3,n_4) \in (\mathbb{Z}^2)^4: n_1\pm n_2\pm n_3\pm
n_4=0, |n_1|\pm |n_2|\pm |n_3|\pm |n_4|=0 \},\\
\widetilde\AAA_w &:= \{ (n_1,n_2,n_3,n_4) \in (\mathbb{Z}^2)^4:
n_1-n_2+n_3-n_4=0,
|n_1|-|n_2|+|n_3|-|n_4|=0 \}.
\end{aligned}
\end{equation}

Let $N \geq 2$ be an integer, and let $\AAA$ be either $\AAA_{bh}$ or $\AAA_w$
(analogously be  $\widetilde\AAA$ either $\widetilde\AAA_{bh}$ or
$\widetilde\AAA_w$). We define a set $\Lambda \subset \mathbb{Z}^2$ which
consists of two  disjoint generations, $\Lambda = \Lambda_1 \cup \Lambda_2$,
$|\Lambda_1|=|\Lambda_2|=2N$. Define a \emph{nuclear family} to be a set $(n_1,
n_2, n_3, n_4) \in \widetilde\AAA$ whose elements are ordered, such that $n_1$
and $n_3$ (known as the \emph{parents}) belong to the first generation
$\Lambda_1$, and $n_2$ and $n_4$ (known as the \emph{children}) belong to the
second generation $\Lambda_{2}$.
Note that if $(n_1, n_2, n_3, n_4)$ is a nuclear family, then so are $(n_1, n_4, n_3, n_2)$, $(n_3, n_2, n_1, n_4)$ and $(n_3, n_4, n_1, n_2)$. These families are called trivial permutations of the family $(n_1, n_2, n_3, n_4)$.
The first conditions to impose on the set $\Lambda$ were already imposed in the
paper \cite{CKSTT}.
\begin{itemize}
\item[$1_{\Lambda}$] (\emph{Closure}) If $n_1, n_2, n_3 \in \Lambda$ and there
exists $n\in\Z^2$ such that  $(n_1, n_2, n_3, n) \in \widetilde\AAA$ (or any
permutation of it), then $n \in \Lambda$. In other words, if three members of
a nuclear family are in $\Lambda$, so is the fourth one. This is a rephrasing
of the completeness condition (see Definition \ref{def:complete}).
\item[$2_{\Lambda}$] (\emph{Existence and uniqueness of spouse and children})
For any $n_1 \in \Lambda_1$, there exists a unique nuclear family $(n_1, n_2, n_3, n_4)$ (up to trivial permutations) such that $n_1$ is a parent of this family. In particular, each $n_1 \in \Lambda_1$ has a unique spouse $n_3 \in \Lambda_1$ and has two unique children $n_2, n_4 \in \Lambda_{2}$ (up to permutation).
\item[$3_{\Lambda}$] (\emph{Existence and uniqueness of sibling and parents})
For any $n_2 \in \Lambda_{2}$, there exists a unique nuclear family $(n_1, n_2, n_3, n_4)$ (up to trivial permutations) such that $n_2$ is a child of this family. In particular each $n_2 \in \Lambda_{2}$ has a unique sibling $n_4 \in \Lambda_{2}$ and two unique parents $n_1, n_3 \in \Lambda_1$ (up to permutation).
\item[$4_{\Lambda}$] (\emph{Faithfulness})
Apart from the nuclear families,
$\Lambda$ does not contain any other set $(n_1,n_2,n_3,n_4)\in\AAA$.
\end{itemize}

%

In the next two propositions we construct a set $\Lambda$ for the three
considered PDEs. In some of the cases we need further conditions.
%
\begin{prop}
\label{prop:LambdaSetHartreeHet}
Let $N \geq 2$ and take $\AAA=\AAA_{bh}$. Then there exists a set $\Lambda \subset \mathbb{Z}^2$, with $\Lambda = \Lambda_1 \cup \Lambda_2$ and $|\Lambda_j|=2N$, which satisfies
properties $1_{\Lambda}$--$4_{\Lambda}$ and the following additional property:
any $n_k,n_{k}',n_h,n_h'\in\Lambda$ such that $n_k \neq n_h$ and $n_k' \neq n_h'$     satisfy
\begin{equation} \label{nondegLambda}
n_k - n_h \neq n_{k}' - n_{h}'.
\end{equation}
\end{prop}

\begin{prop} \label{prop:LambdaSetWaveHom}
Let $N \geq 2$ and take $\AAA=\AAA_{w}, \AAA_{bh}$. Then there exists a set $\Lambda
\subset \Zodd^2$, with $\Lambda = \Lambda_1 \cup \Lambda_2$
and $|\Lambda_j|=2N$, which satisfies conditions $1_{\Lambda}$--$4_{\Lambda}$
and the following additional condition. Take any
$n,n'\in\Lambda$, then
\begin{equation} \label{eq:Wave}
 |n| \neq |n'|.
\end{equation}
Moreover, if one takes $0<\eps\ll 1$, there exists $R=R(\eps)\gg1$ so that $\Lambda$ can be chosen to  satisfy also
%
\begin{equation}\label{eq:Wave2}
\left| |n| - R \right| < R\eps, \; \;  \text{ for all }\quad n \in \Lm.
\end{equation}
\end{prop}

Let us make some comments on the extra conditions imposed on $\Lambda$ in these
propositions. Condition \eqref{eq:Wave} below is required to apply Melnikov
Theory in Section \ref{sec:resonantmodel}. Condition
\eqref{eq:Wave2} is used to obtain Hamiltonian systems on $V_{\Lambda}$
(see \eqref{def:SLambda}) which are close to integrable for the Beam and Wave
equations. For the Beam and Wave equation we also require that the first
component of the modes in $\Lambda$ is odd. This is fundamental in the
approximation argument (Proposition \ref{approximationargument}) to avoid
interactions with the mode $n=0$ which is not elliptic.

We defer the proof of the above propositions to the Appendix
\ref{app:LambdaSet}.

\begin{lem}\label{lem:HresLambda}
Consider the Hamiltonian \eqref{def:commonham} given by the equations
\eqref{Hartree}, \eqref{Beam}, \eqref{Wave} and the associated $H_{\Res}$ in
\eqref{Hres} and the set $\Lambda$ obtained in Propositions
\ref{prop:LambdaSetHartreeHet} and
\ref{prop:LambdaSetWaveHom}.
Then $V_{\Lambda}$ is invariant
and the restriction of $H_{\Res}$ to $V_{\Lambda}$ (see
\eqref{def:SLambda}) has the following form
\begin{equation}\label{def:HResOnLambda}
\begin{aligned}
&\left.(H_{\Res})\right|_{{V_{\Lambda}}}(\{\mathtt{a}_{n}\}_{n\in\Lambda})\\
&=\frac{3}{8}\sum_{\substack{j_i\in \Lambda,\\ j_1-j_2+j_3-j_4=0,\\
|j_1|^{\kkk}-|j_2|^{\kkk}+|j_3|^{\kkk}-|j_4|^{\kkk}=0}} C_{j_1\dots j_4}\,
\mathtt{a}_{j_1}\,\overline{\mathtt{a}_{j_2}}\,\mathtt{a}_{j_3}\,\overline{\mathtt{a}_{j_4}}\\
= &\, \frac{3}{8}\sum_{n\in\Lambda} C_{nnnn}\,\lvert \mathtt{a}_n \rvert^4+ \frac{3}{4}\sum_{i\neq j, n_i, n_j\in\Lambda} C_{n_i n_j n_i n_j} \lvert \mathtt{a}_{n_i} \rvert^2\,\lvert \mathtt{a}_{n_j} \rvert^2 \\
&\,+ \frac{3}{4}\sum_{k=1}^N \,(C_{n_{4k-3} n_{4k-2} n_{4k-1} n_{4k}}+C_{n_{4k-3} n_{4 k} n_{4k-1} n_{4k-2}}+C_{n_{4k-1} n_{4k-2} n_{4k-3} n_{4 k}}+C_{n_{4k-1} n_{4 k} n_{4k-3} n_{4k-2}}) \times \\
&\; \; \; \; \; \; \; \; \; \; \; \; \times \,\mathrm{Re}(\mathtt{a}_{n_{4k-3}}\,\overline{\mathtt{a}_{n_{4k-2}}}\,\mathtt{a}_{n_{4k-1}}\,\overline{\mathtt{a}_{n_{4 k}}})
\end{aligned}
\end{equation}
with $C_{j_1 j_2 j_3 j_4}=C^{+ - + -}_{j_1 j_2 j_3 j_4}$ (see \eqref{def:commonham}, \eqref{def:Cj1}, \eqref{def:Cj2}), namely
\[
\begin{aligned}
\kkk&=2, \qquad
C_{j_1 j_2 j_3 j_4}&=&\,V_{j_1-j_2}=1+\OO(\eps)\qquad \text{(Hartree)}\\
\kkk&=1, \qquad
C_{j_1 j_2 j_3 j_4}&=&\,\frac{1}{16\sqrt{\lvert j_1 \rvert \lvert j_2 \rvert
\lvert j_3 \rvert \lvert j_4
\rvert}}=\frac{1}{R^2}\left(1+\OO(\eps)\right)\qquad \text{(Wave)}\\
\kkk&=2,\qquad
C_{j_1 j_2 j_3 j_4}&=&\,\frac{1}{16\lvert j_1 \rvert \lvert j_2 \rvert \lvert j_3
\rvert \lvert j_4 \rvert}=\frac{1}{R^4}\left(1+\OO(\eps)\right)\qquad
\text{(Beam)}.
\end{aligned}
\]
Therefore,  these coefficients satisfy $C_{j_1 j_2 j_3 j_4}\neq 0$.
\end{lem}
\begin{proof}
The particular  form of Hamiltonian $(H_{\Res})_{|_{V_{\Lambda}}}$ is
a direct consequence of the Properties $1_\Lm$--$4_\Lm$ satisfied by the set
$\Lm$ and the definition of $H_\Res$ in \eqref{Hres}. The definition of the
coefficients $C_{j_1 j_2 j_3 j_4}$ is given in \eqref{def:Cj1},\eqref{def:Cj2}
and their estimates are consequence of \eqref{def:potential}, \eqref{cond:potential} and
\eqref{eq:Wave2}.
\end{proof}

We use the symmetries of the Hamiltonian \eqref{def:HResOnLambda} to remove some of the monomials by a gauge transformation. Indeed, since the mass $M:=\sum_{n\in\Lambda} \lvert \mathtt{a}_n \rvert^2$ is a conserved quantity for $(H_{\Res})_{|_{V_{\Lambda}}}$, we can consider the change of coordinates and time reparametrization
\begin{equation}\label{fase}
\alpha_n=\mathtt{a}_n\,e^{\mathrm{i} G t}\qquad \text{and}\qquad t=-(8/3)\mathtt{g}\,\tau
\qquad\text{with}\qquad  G= \frac{3}{4}\,\mathtt{g}\,M,
\end{equation}
for some $\mathtt{g}\in\mathbb{R}$ to be chosen. The new system is Hamiltonian with respect to
%
%
\begin{equation}\label{toymodel}
\begin{aligned}
& \widetilde{H}_{\Res}(\{\alpha_{n}\}_{n\in\Lambda})\\
&:=\sum_{n\in\Lambda} \lvert \alpha_{n} \rvert^4+
\sum_{n\in\Lambda} \lvert \alpha_{n} \rvert^4 \left(1-\frac{C_{n\, n\, n\, n}}{\mathtt{g}}\right)+2\sum_{n_i, n_j\in\Lambda, \,i\neq j} \lvert \alpha_{n_i} \rvert^2\,\lvert \alpha_{n_j} \rvert^2 \left(1-\frac{C_{n_i n_j n_i n_j}}{\mathtt{g}}\right)\\
&-\frac{2}{\mathtt{g}}\,\sum_{k=1}^N \,(C_{n_{4k-3} n_{4k-2} n_{4k-1} n_{4k}}+C_{n_{4k-3} n_{4 k} n_{4k-1} n_{4k-2}}+C_{n_{4k-1} n_{4k-2} n_{4k-3} n_{4 k}}+C_{n_{4k-1} n_{4 k} n_{4k-3} n_{4k-2}}) \times \\
&\; \; \; \; \; \; \; \; \; \; \; \; \times \,\mathrm{Re} (\alpha_{n_{4k-3}}\,\overline{\alpha}_{n_{4k-2}}\,\alpha_{n_{4k-1}}\,\overline{\alpha}_{n_{4 k}}).
\end{aligned}
\end{equation}
Choosing the  constant $\mathtt{g}$ in \eqref{fase} as
\begin{equation}\label{def:g}
\mathtt{g}=\,1\qquad \text{(Hartree)},\qquad
\mathtt{g}=\,\frac{1}{R^2}\qquad \text{(Wave)}\qquad
\mathtt{g}=\,\frac{1}{R^4}\qquad \text{(Beam)},
\end{equation}
%
then the Hamiltonian system \eqref{toymodel} takes the following form
%
\begin{equation}\label{hamN}
\begin{split}
\widetilde{H}_{\Res}(\alpha_{n_1}, \dots, \alpha_{ n_{2N} })=&\,\sum_{k=1}^{4N}
\lvert \alpha_{n_k} \rvert^4\,+2\,\eps  \sum_{1 \leq i,j \leq 4N}
A_{i , j} \,\lvert \alpha_{n_i} \rvert^2\,\lvert \alpha_{n_j} \rvert^2\\
&\,- 8 \,
\sum_{h=1}^N \mathtt{C}_h \,\mbox{Re}(  \alpha_{n_{4h-3}} \,
\overline{ \alpha_{n_{4h-2}} } \, \alpha_{n_{4h-1}}\, \overline{ \alpha_{n_{4h} } }
),
\end{split}
\end{equation}
where  $A = (A_{i , j}) \in \mathbb{R}^{4N \times 4N}$ is a symmetric  matrix given by
%
%
\begin{equation}\label{def:coeffmatrixA}
\eps A_{j, j}:=\frac{1}{2}-\frac{C_{n_j n_j n_j n_j}}{2 \mathtt{g}}, \qquad
\eps A_{i, j}:=1-\frac{C_{n_i n_j n_i n_j}}{\mathtt{g}} \quad i\neq j,
\end{equation}
and $(\mathtt{C}_h)_{h=1,\ldots,N}$ satisfies
\begin{equation*}
\mathtt{C}_{h}=1+\mathcal{O}( \eps ), \; \; \forall \, h=1, \dots, N.
\end{equation*}
The equations of motion read as
%
\begin{equation}\label{eq:ModelSLambda}
\mathrm{i}\,\,\dot{\alpha}_{n_k}=\left(2\lvert \alpha_{n_k} \rvert^2 + 2\eps
\sum_{1 \leq r \leq 2N} A_{k , r} \lvert \alpha_{n_r}
\rvert^2\right)\,\alpha_{n_k}-8\,\mathtt{C}_h
\,\alpha_{n_i}\,\overline{\alpha_{n_l}}\,\alpha_{n_j},
\end{equation}
where $k, i, l, j\in\{4(h-1)+1, 4(h-1)+2, 4(h-1)+3, 4h\}$ give the four modes forming a resonant tuple, $l+k$ is even (namely, following \cite{CKSTT}, $n_k$ and $n_l$ belong to the same generation) and $h \in \{1,\ldots,N\}$.

\subsection{Invariant subspaces and first integrals of the resonant model}
The system associated to Hamiltonian $\widetilde{H}_{\Res}$ in \eqref{hamN} has
large dimension and it is not integrable. Nevertheless, the properties of the
set $\Lambda$ and the particular form of the Hamiltonian $\widetilde{H}_{\Res}$ ensure that the system associated to
$\widetilde{H}_{\Res}$ has several invariant subspaces, where one can easily
analyze the dynamics.
We devote this section to analyze these invariant subspaces and the  first integrals of $\widetilde{H}_{\Res}$.

Let us split $\Lm$ both as $\Lm=\Lm_1\cup\Lm_2=\RR_1\cup\ldots\cup \RR_N$. The first
splitting refers to the two generations and the second refers to the $N$
four--wave resonances used to define $\Lambda$ (see \eqref{def:Abh} and
\eqref{def:Aw}).

Associated to this set we can consider the following invariant subspaces (recall \eqref{def:SLambda})
\begin{equation*}
V_{\Lambda_i}=\left\{\alpha\in V_{\Lambda}: \alpha_j=0\,
\text{ for }\, j\not\in \Lambda_i\right \}
\end{equation*}
and, for $\{i_1,\ldots,i_k\}\subset \{1,\ldots, N\}$, $1<k<N$,
\begin{equation*}
V_{i_1,\ldots,i_k}=\left\{\alpha\in V_{\Lambda}:
\alpha_j=0\,
\text{ for }\, j\not\in \RR_{i_1}\cup\ldots\cup\RR_{i_k}\right \}.
\end{equation*}
One can easily check that all those subspaces are invariant under the flow
associated to equation \eqref{eq:ModelSLambda}. Let us study the corresponding
dynamics.

For $V_{\Lambda_1}$ (and analogously for $V_{\Lambda_2}$)
one obtains the equation
%
\[
 \mathrm{i}\,\,\dot{\alpha}_{n_k}=\left(2\lvert \alpha_{n_k} \rvert^2 +
2\varepsilon \sum_{1 \leq r \leq 2N ,
n_r\in\Lambda_1} A_{k , r} \lvert \alpha_{n_r}
\rvert^2\right)\,\alpha_{n_k}\qquad
\text{ for }\,{\alpha}_{n_k}\in\Lambda_1.
\]
Therefore, on $V_{\Lambda_1}$, $|\alpha_{n_r}|^2$ are constants of
motion and the phase space is foliated by invariant tori
\begin{equation}\label{def:HypTori}
 \T_{I_1,\ldots I_k}=\left\{\alpha\in V_{\Lambda_1}: |\alpha_{n_k}|=I_k\right\} \quad \mbox{where} \quad I_j>0, \quad j=1, \dots, k.
\end{equation}
It can be checked (see
Section \ref{sectionsymplecticreduction} below) that these invariant tori are
hyperbolic and thus have stable and unstable invariant manifolds.

The dynamics on $V_{i_1,\ldots,i_k}$ is just given as well by
equation \eqref{eq:ModelSLambda} just considering the interactions between the
modes in the rectangles $\RR_{i_1}, \ldots\RR_{i_k}$.

Hamiltonian $\widetilde{H}_{\Res}$ in \eqref{hamN} has the
first integrals
\begin{equation}\label{comN}
\begin{aligned}
S^{(k,+)}_{i,j}= \lvert \alpha_{n_{4(k-1)+i}} \rvert^2 + \lvert \alpha_{n_{4(k-1)+j}}
\rvert^2 \qquad i+j\equiv 1 \,\,(\mbox{mod}\,\, 2), \quad i, j\in \{1, 2, 3,
4\}, \,\,\, k \in \{1, \ldots, N \},\\
S^{(k,-)}_{i,j}=\lvert \alpha_{n_{4(k-1)+i}} \rvert^2 - \lvert \alpha_{n_{4(k-1)+j}}
\rvert^2 \qquad i+j\equiv 0 \,\,(\mbox{mod}\,\, 2), \quad i, j\in \{1, 2, 3,
4\}, \,\,\, k \in \{1, \ldots, N \}.
\end{aligned}
\end{equation}
These constants of motion are in involution. They are not functionally
independent but it can be easily checked that the subset of first integrals
\begin{equation}\label{comN2}
S^{(k,-)}_{1,3}, S^{(k,-)}_{2,4}, S^{(k,+)}_{3,4},\quad \,\,\, k
\in \{1, \ldots, N \}
\end{equation}
is functionally independent in the open set $\{\alpha_{n}\neq
0:n\in\Lambda\}\subset V_{\Lambda}$. Certainly they are
not functionally independent on the invariant subspaces
$V_{\Lambda_1}, V_{\Lambda_2}$ (and in particular are not
functionally independent at the tori $\T_{I_1,\ldots I_k}$).

\subsection{The symplectic reduction}\label{sectionsymplecticreduction}
We use the first integrals \eqref{comN2} to perform a symplectic reduction to
the Hamiltonian \eqref{hamN}. It can be applied in the open set $\{\alpha_{n}\neq
0:n\in\Lambda\}\subset V_{\Lambda}$ where the first integrals are
functionally independent.  In this domain, all modes are different from zero
and thus one can consider
%
%
symplectic polar coordinates $(\theta,I) \in \mathbb{T}^{4N} \times
(0,+\infty)^{4N}$, given by
\begin{equation} \label{AcAn}
\begin{aligned}
\alpha_{n_k} = \sqrt{I_k} e^{i \theta_k}.
\end{aligned}
\end{equation}
%
In these coordinates the Hamiltonian \eqref{hamN} takes the form
%
\begin{equation}\label{hamaa}
\begin{aligned}
H(\theta, I)=& \langle I, I \rangle+2\varepsilon\, \langle A I, I \rangle\\
&-8 \sum_{h=1}^N \mathtt{C}_h \sqrt{ I_{4(h-1)+1} \, I_{4(h-1)+2} \, I_{4(h-1)+3} \, I_{4h} } \,\cos(\theta_{4(h-1)+1}-\theta_{4(h-1)+2}+\theta_{4(h-1)+3}-\theta_{4h})
\end{aligned}
 \end{equation}
 and the symplectic form $\Omega_{|_{V_{\Lambda}}}$ becomes the standard one $d\theta\wedge d I=\sum_{k=1}^{4 N} d\theta_k\wedge d I_k$. The Hamiltonian system \eqref{hamaa} has $4 N$ degrees of freedom. We perform a symplectic reduction that leads to an $N$ degrees of freedom system. In particular first  we consider the restriction of \eqref{hamN} to
\begin{equation}\label{subspaceVN}
\mathcal{V}:= \bigcap_{k=1}^N \left\{ S^{(k,-)}_{1 , 3} =
S^{(k,-)}_{2 ,4}= 0 \right\},
\end{equation}
and then we further reduce it to the manifold
\begin{equation}\label{subspaceWN}
\mathcal{W}:= \bigcap_{k=1}^N  \left\{ S^{(k,+)}_{3,4}=1
\right\}  \cap \mathcal{V}.
\end{equation}
We adopt the following notation: we denote by $\mathbf{0}_n$ the null matrix of dimension $n\times n$ and by $\mathrm{I}_n$ the identity matrix of dimension $n\times n$.
We consider the symplectic linear change of variable $\Psi\colon \mathbb{T}^{4N} \times \mathbb{R}^{4N} \to  \mathbb{T}^{4N} \times \mathbb{R}^{4N}$ defined by
\begin{equation*}
\begin{pmatrix}
\theta\\
I
\end{pmatrix}
=\Psi \begin{pmatrix}
\phi\\
J
\end{pmatrix}
\end{equation*}
with, for $h=0\ldots N-1$,
\begin{equation*}
\begin{pmatrix}
\theta_{4h+1}\\
\theta_{4h+2}\\
\theta_{4h+3}\\
\theta_{4h+4}
\end{pmatrix}
=\begin{pmatrix}
1 & 0 & 0 & 0\\
0 & 1 & 0 & 0\\
-1 & 0 & 1 & 0\\
0 & -1 & 0 & 1
\end{pmatrix}
\begin{pmatrix}
\phi_{4h+1}\\
\phi_{4h+2}\\
\phi_{4h+3}\\
\phi_{4h+4}
\end{pmatrix}
,\qquad
\begin{pmatrix}
I_{4h+1}\\
I_{4h+2}\\
I_{4h+3}\\
I_{4h+4}
\end{pmatrix}
=\begin{pmatrix}
1 & 0 & 1 & 0\\
0 & 1 & 0 & 1\\
0 & 0 & 1 & 0\\
0 & 0 & 0 & 1
\end{pmatrix}
\begin{pmatrix}
J_{4h+1}\\
J_{4h+2}\\
J_{4h+3}\\
J_{4h+4}
\end{pmatrix}
\; =: \; \tilde{\mathtt{B}}
\begin{pmatrix}
J_{4h+1}\\
J_{4h+2}\\
J_{4h+3}\\
J_{4h+4}
\end{pmatrix}.
\end{equation*}
%
%
%
We consider the restriction of the new Hamiltonian $H\circ \Psi$ at the invariant submanifold $\mathcal{V}$ defined in \eqref{subspaceVN}, which corresponds, in the coordinates \eqref{AcAn}, to the subspace
\begin{equation*}
\{  J_{4k+1}=J_{4k+2}=0 : k=0, \dots, N-1\}.
\end{equation*}
The new Hamiltonian does not depend on the angles $\{\phi_i\}_{i\in\{4k+1,4k+2 : k=0,\dots, N-1\}}$ and it reads as
%
%
\begin{align*}
\mathsf{H}\left(\{\phi_{4k+3},\phi_{4k+4},J_{4k+3},J_{4k+4}\}_{k=0}^{N-1}\right)=& 2 \sum_{h=0}^{N-1} \sum_{k=3}^4 J_{4h+k}^2 \\
&+2\varepsilon \, \sum_{0 \leq h,h' \leq N-1} \sum_{\substack{i=3, 4 \\ k=3, 4}} (\mathtt{B}^T A \mathtt{B})_{4h+i, 4h'+k} J_{4h+i} \, J_{4h'+k} \\
&-8 \sum_{h=0}^{N-1} \mathtt{C}_{h+1} \,J_{4h+3}\,J_{4h+4}\,\cos(\phi_{4h+3}-\phi_{4h+4}),
\end{align*}
where
\begin{align*}
 \mathtt{B} &= \begin{pmatrix}
 \tilde{\mathtt{B}} & {\mathbf{0}}_4 & \cdots & \cdots & {\mathbf{0}}_4 \\
 {\mathbf{0}}_4 & \tilde{\mathtt{B}} & {\mathbf{0}}_4 & \cdots & {\mathbf{0}}_4 \\
 \vdots & {\mathbf{0}}_4 & \tilde{\mathtt{B}} & {\mathbf{0}}_4 & \cdots \\
 \vdots & \vdots & \ddots & \ddots  & \vdots \\
 {\mathbf{0}}_4 & \cdots & \cdots & \cdots & \tilde{\mathtt{B}}
 \end{pmatrix}\in \mathbb{R}^{4N \, \times \, 4N}.
\end{align*}
The second symplectic reduction is obtained by considering the symplectic linear change of variable $\Phi \colon \mathbb{T}^{2N} \times \mathbb{R}^{2N} \to  \mathbb{T}^{2N} \times \mathbb{R}^{2N}$ as \[\left(\{\phi_{4k+3}\}_{k=0}^{N-1},\{\phi_{4k+4}\}_{k=0}^{N-1}, \{J_{4k+3}\}_{k=0}^{N-1}, \{J_{4k+4}\}_{k=0}^{N-1}\right)=\Phi(\{K_k\}_{k=1}^N,\{\tilde{K}_1\}_{k=1}^N,\{\psi_k\}_{k=1}^N,\{\tilde{\psi}_1\}_{k=1}^N)\]
defined by
\begin{equation*}
\begin{pmatrix}
\phi_{4k+3}\\
\phi_{4k+4}
\end{pmatrix}=\begin{pmatrix}
1 & 1 \\
0 & 1
\end{pmatrix}
\begin{pmatrix}
\psi_{k+1}\\
\tilde{\psi}_{k+1}\end{pmatrix},\qquad
\begin{pmatrix}
J_{4k+3}\\
J_{4k+4}
\end{pmatrix}
=\begin{pmatrix}
1 & 0 \\
-1 & 1
\end{pmatrix}\begin{pmatrix}
K_{k+1}\\
\tilde{K}_{k+1}
\end{pmatrix},\qquad k=0\ldots N-1.
\end{equation*}

After the reparametrization of time $t\mapsto -4 \,t$, the restriction of the transformed Hamiltonian $\mathsf{H}\circ \Phi$ to the subspace
\begin{equation*}
\mathcal{W}=\mathcal{V} \bigcap_{k=1}^N  \{ S^{(k,+)}_{3,4}=1 \}=
\bigcap_{k=1}^N \{ \tilde{K}_k=1\}
\end{equation*}
is given (up to constants) by
\begin{equation}\label{redhamN}
\begin{aligned}
\mathcal{H}(\psi_1, \ldots, \psi_N, K_1, \ldots, K_N)=& \sum_{j=1}^N K_j (1-K_j)(1 +2  \,\cos(\psi_j))\\
&+\varepsilon\Big[ \sum_{j=1}^N a_{j} K_j +\sum_{j=1}^N b_{j} K^2_j  + \sum_{i, j=1, i< j}^N d_{i j} K_i\,K_j
+\sum_{h=1}^N c_j K_j(1-K_j) \,\cos(\psi_j)\Big]
\end{aligned}
\end{equation}
where the coefficients $a_j$, $b_j$ and $d_j$ can be written in terms of the entries of the matrix $A$ in \eqref{hamN} in the following way
\begin{equation} \label{coeffN}
\begin{aligned}
a_j := -&\sum_{r=1}^N \Big[ A^{(j, r)}_{1, 2}+A^{(j, r)}_{1, 4}+A^{(j, r)}_{ 2, 3}+A^{(j, r)}_{3, 4}-\big( A^{(j, r)}_{2, 2}+ A^{(r, j)}_{2, 4}+ A^{(r, j)}_{4, 2}+ A^{(j, r)}_{4, 4}  \big)\Big]  \\
b_j := -&\Big[ A^{(j, j)}_{1, 1}+2 A^{(j, j)}_{1, 3}+A^{(j, j)}_{3, 3}-2 \big( A^{(j, j)}_{1, 2} +A^{(j, j)}_{1, 4}+A^{(j, j)}_{2, 3}+A^{(j, j)}_{3, 4}\big)+A^{(j, j)}_{2, 2}+2 A^{(j, j)}_{2, 4}+A^{(j, j)}_{4, 4}\Big]  \\
d_{i j} := -&\Big[  A^{(i, j)}_{1, 1}+ A^{(i, j)}_{1, 3}+A^{(i, j)}_{3, 1}+A^{(i, j)}_{3, 3}+A^{(i, j)}_{2, 2}+A^{(i, j)}_{2, 4}+A^{(i, j)}_{4, 2}+A^{(i, j)}_{4, 4}\\
-&\big( A^{(i, j)}_{1, 2}+A^{(i, j)}_{2, 1}+A^{(i, j)}_{1, 4}+A^{(i, j)}_{4, 1}+A^{(i, j)}_{2, 3}+A^{(i, j)}_{3, 2}+A^{(i, j)}_{3, 4}+A^{(i, j)}_{4, 3} \big)\Big], \\
\end{aligned}
\end{equation}
with $A^{(i, j)}_{n, m}:=A_{4(i-1)+n, 4(j-1)+m}$, $n, m\in \{ 1, 2, 3, 4  \}$, $i, j\in \{ 1, \dots, N \}$ and
\begin{equation}
c_j:= \frac{2}{\varepsilon}(\mathtt{C}_j-1).
\end{equation}
Recall that $A$ is symmetric, hence $d_{i j}=d_{j i}$.
\begin{remark}\label{rem:Kcartcoord}
We point out that the variables $K_i$ in \eqref{redhamN} reads, in the coordinates $\{ \alpha_j \}_j$ (see \eqref{fase}) as
\[
K_i:=|\alpha_{n_{4 (i-1)+1}}|^2=|\alpha_{n_{4 (i-1)+3}}|^2=1-|\alpha_{n_{4 (i-1)+2}}|^2=1-|\alpha_{n_{4 (i-1)+4}}|^2 \quad \forall i=1, \dots, N.
\]
\end{remark}
It can be easily seen that the hyperplanes $\{K_j=0\}$, $\{K_j=1\}$ are
invariant under the Hamiltonian \eqref{redhamN}. Indeed one can understand the
Hamiltonian \eqref{redhamN} as defined on the product sphere $(S^2)^N$ by
``blowing down'' the sets $\{K_j=0\}$, $\{K_j=1\}$ to a point in each sphere.
That is, one can consider local coordinates
\begin{equation}\label{def:cartesian}
x_j=\sqrt{2K_j}\cos\frac{\psi_j}{2}, \quad y_j=\sqrt{2K_j}\sin\frac{\psi_j}{2}
\end{equation}
which blow down $\{K_j=0\}$. Then, the
Hamiltonian \eqref{redhamN} becomes
\begin{equation}\label{redhamNCartesian}
\begin{aligned}
\mathcal{H}(x_1, \ldots, x_N, y_1, \ldots, y_N)=& \frac{1}{2}\sum_{j=1}^N
\left(3x_j^2-y_j^2\right)-\frac{1}{4}\sum_{j=1}^N
\left(3x_j^2-y_j^2\right)\left(x_j^2+y_j^2\right)\\
&+\varepsilon\Bigg[ \frac{1}{2}\sum_{j=1}^N a_{j} \left(x_j^2+y_j^2\right)
+\frac{1}{4}\sum_{j=1}^N b_{j} \left(x_j^2+y_j^2\right)^2 \\&
+
\frac{1}{4}\sum_{i, j=1, i< j}^N d_{i j}
\left(x_i^2+y_i^2\right)\left(x_j^2+y_j^2\right)\\&
+\frac{1}{4}\sum_{h=1}^N c_j
\left(x_j^2-y_j^2\right)\left(2-x_j^2-y_j^2\right)\Bigg].
\end{aligned}
\end{equation}
From the particular form of this Hamiltonian, it is clear that $\{x_j=y_j=0\}$ is
invariant under the associated flow. In particular the point
\begin{equation}\label{def:saddledown}
P_-=\left\{x_j=0, y_j=0,\quad j=1\ldots N\right\},
\end{equation}
is a saddle (for small $\eps$) with $N$ dimensional stable and unstable manifolds.
One can analogously blow down $\{K_j=1\}$ by considering the coordinates
\[
 x_j=\sqrt{2(1-K_j)}\cos\frac{\psi_j}{2}, \quad
y_j=\sqrt{2(1-K_j)}\sin\frac{\psi_j}{2}
\]
and one also obtains that, for $\eps$ small enough, $P_+=\{x_j=0, y_j=0,\,\, j=1\ldots N\}$ is a saddle  with $N$ dimensional stable and unstable manifolds. This saddle is the ``blow down'' of $\{K_1=\ldots=K_N=1\}$.

\section{Dynamics of the resonant model}\label{sec:resonantmodel}

The reduced Hamiltonian \eqref{redhamN} for $N=2$ is of the form
\begin{equation}\label{calham0}
\begin{aligned}
\mathcal{H}(\varepsilon; \psi_1, \psi_2, K_1, K_2)=&\mathcal{H}_0(\psi_1, \psi_2, K_1, K_2)+\varepsilon \mathcal{H}_1(\psi_1, \psi_2, K_1, K_2)\\
\mathcal{H}_0(\psi_1, \psi_2, K_1, K_2)=&\,\mathcal{H}_0^{(1)}(\psi_1, K_1)+\mathcal{H}_0^{(2)} (\psi_2, K_2)\\
\mathcal{H}_0^{(1)}(\psi_1, K_1)= &\,K_1(1-K_1)(1+2\cos(\psi_1))\\
\mathcal{H}_0^{(2)}(\psi_2, K_2)=&\,K_2(1-K_2)(1+2\cos(\psi_2))\\
\mathcal{H}_1(\psi_1, \psi_2, K_1, K_2)=&\,a_1K_1+b_1K_1^2+a_2 K_2+b_2 K_2^2\\
&+c_1K_1(1-K_1)\cos(\psi_1)+c_2K_2(1-K_2)\cos(\psi_2) +d_{12}K_1K_2.
\end{aligned}
\end{equation}
Note that the only term which couples the two unperturbed Hamiltonians $\mathcal{H}_0^{(1)}$, $\mathcal{H}_0^{(2)}$ is $d_{12}K_1K_2$. The Hamiltonian $\mathcal{H}$ is reversible with respect to the involution
\begin{equation}\label{anglesymmetry}
\Upsilon(\psi_1, \psi_2, K_1, K_2)=(-\psi_1,-\psi_2, K_1, K_2).
\end{equation}
%
%
\subsection{Unperturbed dynamics ($\varepsilon=0$)}\label{sec:unperturbeddynamics}

For $\eps=0$, the Hamiltonian system $\mathcal{H}_0$ is the product of the two uncoupled $1$-d.o.f systems with  Hamiltonian $\mathcal{H}_0^{(i)}$, $i=1, 2$ and therefore it is integrable.
We analyze the dynamics given by $\mathcal{H}_0^{(i)}$. We analyze it only for $\mathcal{H}_0^{(1)}$ since both Hamiltonians are equal.
%
%
The associated equations of motion  are given by
\begin{equation*}
\begin{split}
\dot{\Psi}_1=&\,(1-2 K_1) (1+2\cos(\psi_1))\\
\dot{K}_1=&\,2\sin(\psi_1)\,K_1\,(1-K_1).
\end{split}
\end{equation*}
 The sets $\{  K_1=0\}$ and $\{K_1=1\}$ are $\mathcal{H}_0^{(1)}$-invariant $1$-dimensional tori which correspond to the hyperbolic  tori \eqref{def:HypTori} after symplectic reduction  and correspond to saddles in proper ``blow down'' coordinates (see \eqref{def:cartesian}, \eqref{def:saddledown}). The sets $\{  K_1=0\}$ and $\{K_1=1\}$ possess the hyperbolic equilibrium points $\left(\pm \Psi_{*}, 0\right)$, $\left( \pm \Psi_{*}, 1 \right)$
  with
  \begin{equation}\label{def:psi}
  \Psi_*=2\pi/3.
  \end{equation}
  Such equilibria are hyperbolic with eigenvalues  $\pm \sqrt{3}$. Their invariant manifolds outside of  $\{  K_1=0\}$ and $\{K_1=1\}$ correspond to the invariant manifolds of the saddles $P_\pm$ in  \eqref{def:saddledown}.



The tori $\{  K_1=0\}$ and $\{K_1=1\}$ are on the same energy level $\mathcal{H}_0^{(1)}=0$ and the saddles  $\left(\pm \Psi_{*}, 0\right)$ and  $\left( \pm \Psi_{*}, 1 \right)$ are
 connected through the heteroclinic orbits
 \begin{equation*}
 (\psi_1(t), K_1(t))=\left(\pm \Psi_*, \dfrac{1}{1+e^{\mp \sqrt{3} t}}\right)
 \end{equation*}
 (see Figure \ref{fig:PhaseSpaceH0}).

\begin{figure} 
\begin{center}
\includegraphics[width=0.7\textwidth]{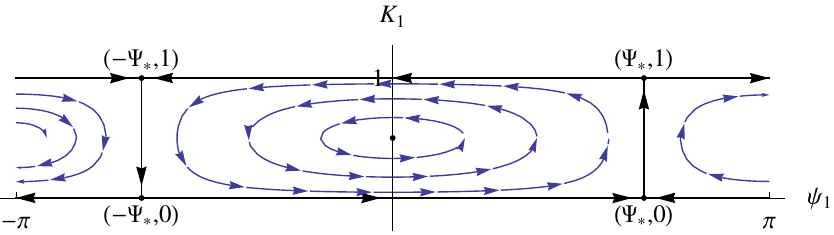}
\end{center}
\caption{{\small Phase space of the Hamiltonian $\mathcal{H}_0^{(1)}$ in \eqref{calham0}} }
\label{fig:PhaseSpaceH0}
\end{figure}

 For $0<K_1<1$, the dynamics of the Hamiltonian $\mathcal{H}_0^{(1)}$ can be also analyzed easily. Consider the ``half'' of the phase space $ (-2\pi/3,2\pi/3)\times (0, 1)\subset\T\times (0, 1)$ limited by the heteroclinic orbits (the other ``half'' is symmetric). It has an elliptic points at $(\psi_1,K_1)=(0,1/2)$ and the rest is foliated by periodic orbits
 \begin{equation}\label{def:Gammah}
\mathtt{P}_h:=\{\mathcal{H}_0^{(1)}=h\} \qquad \mbox{ with} \qquad h\in (0,3/4).
 \end{equation}
  When $h\to 0$, the periodic orbits ``tend'' to the sequence of heteroclinics and $K_1=0,1$ and therefore their period $\mathtt{T}_h\to+\infty$.


%

%
%
%
%
Hence, the dynamics  of the 2-dof Hamiltonian $\mathcal{H}_0$ in \eqref{calham0} has the following features. The invariant tori
 \begin{equation}\label{tori}
\mathbb{T}_0:=\{ K_1=K_2=0\}, \qquad \mathbb{T}_1:=\{ K_1=K_2=1 \}
\end{equation}
are two invariant Lagrangian tori for the system \eqref{calham0}. They possess the equilibrium points
\begin{equation}\label{def:saddles}
\mathfrak{e}_+^{(0)}:=(\Psi_{*}, \Psi_{*} , 0, 0), \quad \mathfrak{e}_+^{(1)}:=(\Psi_{*}, \Psi_{*} , 1, 1), \quad \mathfrak{e}_-^{(1)}:=(-\Psi_{*}, -\Psi_{*} , 1, 1), \quad \mathfrak{e}_-^{(0)}:=(-\Psi_{*}, -\Psi_{*} , 0, 0)
\end{equation} connected by the following heteroclinic manifolds
\begin{equation}\label{gammaplus}
\begin{split}
\gamma_{+}(\tau_1, \tau_2):=(\psi_1^+(\tau_1), \psi_2^+(\tau_2), K_1^+(\tau_1), K_2^+(\tau_2))&=\left(\Psi_{*}, \Psi_{*}, \dfrac{1}{1+e^{-\sqrt{3}\tau_1}},   \dfrac{1}{1+e^{-\sqrt{3} \tau_2}}  \right),\\
\gamma_{-}(\tau_1, \tau_2):=(\psi_1^-(\tau_1), \psi_2^-(\tau_2), K_1^-(\tau_1), K_2^-(\tau_2))&=\left(-\Psi_*, -\Psi_*, \dfrac{1}{1+e^{\sqrt{3}\tau_1}},   \dfrac{1}{1+e^{\sqrt{3}\tau_2}}  \right).
\end{split}
\end{equation}
In particular $\gamma_+$ connects the  points $\mathfrak{e}_+^{(0)}$, $\mathfrak{e}_+^{(1)}$ and $\gamma_-$ connects $\mathfrak{e}_-^{(1)}$ with $\mathfrak{e}_-^{(0)}$. The trajectories in the heteroclinic manifolds are just given by $\gamma_{\pm}(\tau_1+t, \tau_2+t)$, $t\in\mathbb{R}$.

The $4$-dimensional phase space of Hamiltonian $\mathcal{H}$ in \eqref{calham0} with $\eps=0$ has several three-dimensional invariant subspaces setting either $K_1$ or $K_2$ equal to $0$ or $1$, and two dimensional invariant subspaces setting either $(\psi_1,K_1)$ or $(\psi_2,K_2)$ at one of the saddles. Thus, one can define the hyperbolic periodic orbits (recall \eqref{def:Gammah})
\begin{align}\label{def:PO:ext}
\mathtt{P}^{\sigma,s}_h := \{ (\psi_1,\psi_2,K_1,K_2) : (\psi_1,K_1) \in \mathtt{P}_h, \,\psi_2= \sigma \Psi_*,\, K_2=k \}, \; \; \sigma = \pm, \; \; s =0,1,
\end{align}
and one could define analogously the other ones placing them at the other saddles.

For  $\varepsilon=0$, the Hamiltonian system \eqref{calham0} possesses two  $2$-dimensional heteroclinic manifolds $W^{s}(\mathfrak{e}_{+}^{(0)})=W^u(\mathfrak{e}_{+}^{(1)})$, $W^{u}(\mathfrak{e}_{-}^{(0)})=W^s(\mathfrak{e}_{-}^{(1)})$. They are certainly not robust under perturbations. We show that, under a generic non-degeneracy condition, those heteroclinic manifolds break down when $0<\varepsilon\ll 1$ creating transverse intersections between some of the stable and unstable invariant manifolds.


%
%

\subsection{Non-integrable dynamics $(\varepsilon>0)$: $2$ resonant tuples}

For $\eps>0$, the tori
$\mathbb{T}_0$ and $\mathbb{T}_1$ in \eqref{tori} are still invariant and they still possess  saddles which are $\eps$--close to the unperturbed saddles $\mathfrak{e}_\pm^{(j)}$, $j=0,1$. These saddles have 2-dimensional stable and unstable invariant manifolds.

\begin{remark}\label{rmk:SaddlesPerturbed}
Abusing notation, we also denote by $\mathfrak{e}_\pm^{(j)}$, $j=0,1$ the saddles of the perturbed Hamiltonian \eqref{calham0} with $0<\eps\ll 1$, which are $\eps$-close to those defined by \eqref{def:saddles}.
\end{remark}

%

%
%
%
%
%
\begin{teor}\label{thm:SplittingHomo2Rect}
Consider the Hamiltonian \eqref{calham0} and assume that
 \begin{equation}\label{tildekappas}
d_{12}\neq 0\qquad \text{(see \eqref{calham0})}.
\end{equation}
Then, there exists $\varepsilon_0>0$ such that for all $\varepsilon\in (0, \varepsilon_0)$,
 the invariant manifolds $W_{\varepsilon}^{u}(\mathfrak{e}_{+}^{(0)})$  and $W_{\varepsilon}^s(\mathfrak{e}_{-}^{(0)})$ of the saddles \eqref{def:saddles} of the Hamiltonian \eqref{calham0} intersect transversally along an orbit (within the energy level).
\end{teor}

Note that this theorem is not a classical perturbative result. Indeed, for $\eps=0$ the saddles $\mathfrak{e}_{\pm}^{(0)}$ did not have any connection since  their invariant manifolds coincided with those of $\mathfrak{e}_{\pm}^{(1)}$ along heteroclinic connections. Therefore, the prove of Theorem \ref{thm:SplittingHomo2Rect} is not a direct consequence of Melnikov Theory (is not a theorem about persistence, is a theorem about \emph{new} heteroclinic connections). Thus, we prove this theorem in two steps. First in Section \ref{sec:HeteroMelnikov} we apply Melnikov Theory to prove the existence of transverse (within the energy level) heteroclinic connections between  $\mathfrak{e}_{+}^{(0)}$ and  $\mathfrak{e}_{+}^{(1)}$ (under certain conditions). Then, in Section \ref{sec:HomoMelnikov}, we use this analysis to prove the existence of the connections given in Theorem \ref{thm:SplittingHomo2Rect} through a suitable modification of Melnikov Theory.

\subsubsection{Transversal heteroclinic orbits to saddles}\label{sec:HeteroMelnikov}

The first step to prove Theorem \ref{thm:SplittingHomo2Rect} is to prove the existence of heteroclinic intersections between the saddles  $\mathfrak{e}_{\pm}^{(0)}$ and $\mathfrak{e}_{\pm}^{(1)}$. This step is certainly not necessary to obtain homoclinic intersections. Nevertheless, it will make considerably easier the computation of the Melnikov function associated to the homoclinic intersections.
To obtain the mentioned heteroclinic intersections, one certainly needs that the saddles belong to the same energy level, that is, $\mathcal{H}(\mathfrak{e}_{\pm}^{(0)})=\mathcal{H}(\mathfrak{e}_{\pm}^{(1)})$. By \eqref{calham0} this condition is equivalent to
\begin{equation}\label{elc}
a_1+b_1+a_2+b_2+d_{12}=0.
\end{equation}
%
%
%

\begin{prop}\label{teoetero}
The Hamiltonian \eqref{calham0} possesses four hyperbolic fixed points $\mathfrak{e}_{\pm}^{(0)}$, $\mathfrak{e}_{\pm}^{(1)}$ such that the following holds. If \eqref{elc} is satisfied and
\begin{equation}
\label{concord}
(a_1+b_1)(a_2+b_2)>0,
\end{equation}
there exists $\varepsilon_0>0$ such that for $\varepsilon\in (0, \varepsilon_0)$ the manifolds $W_{\varepsilon}^{u}(\mathfrak{e}_+^{(0)})$ and $W_{\varepsilon}^{s}(\mathfrak{e}_+^{(1)})$ intersect transversally along orbits (within the energy level). The same happens for  $W_{\varepsilon}^{s}(\mathfrak{e}_-^{(0)})$ and $W_{\varepsilon}^{u}(\mathfrak{e}_-^{(1)})$.
\end{prop}
See in Figure \ref{fig:Heteros} an example of heteroclinic connections. We devote the rest of the section to prove Proposition \ref{teoetero}.

\begin{figure} 
\begin{center}
\includegraphics[width=0.7\textwidth]{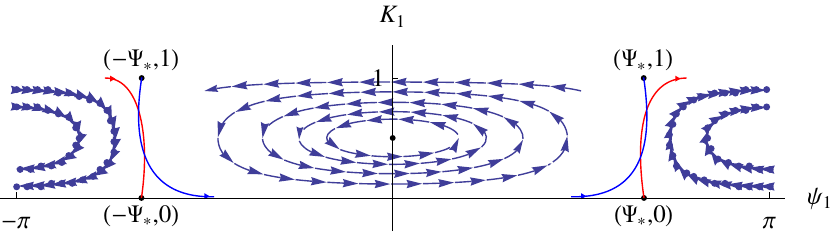}
\end{center}
\caption{Transverse heteroclinic orbits for $\eps$ small enough.}
\label{fig:Heteros}
\end{figure}

\begin{proof}[Proof of Proposition \ref{teoetero}]
Thanks to the symmetry \eqref{anglesymmetry} of the system \eqref{calham0}, one of the intersections implies the other one.
We just deal with the first one.

Consider a compact subset $\mathbf{K}$ of $\mathbb{R}^2$.
Let $\vec{\tau}=(\tau_1, \tau_2)\in \mathbf{K}$ and $m>0$. We consider the line
\begin{equation*}
\Sigma(\vec{\tau})=\{\gamma_+(\vec{\tau})+r\,\nabla \mathcal{H}^{(1)}_0 (\gamma_+(\vec{\tau})),\,\,\,r\in (-m, m)\},
\end{equation*}
which passes through $\gamma_+(\vec{\tau})$ and it is orthogonal to $\{\mathcal{H}^{(1)}_0=\mathcal{H}^{(1)}_0(\gamma_+(\vec{\tau}))\}$ at $\gamma_+(\vec{\tau})$. Since the system has two degrees of freedom and energy conservation, it is enough to measure the distance along this line. It would be equivalent to consider $\mathcal{H}^{(2)}_0$.
Since $\gamma_+(\vec{\tau})\in W_{0}^s(\mathfrak{e}_+^{(1)})=W_{0}^u(\mathfrak{e}_+^{(0)})$,
if we consider $\varepsilon$ small enough we can ensure that $\Sigma(\vec{\tau})$ intersects transversally $W_{\varepsilon}^s(\mathfrak{e}_+^{(1)})$ and $W_{\varepsilon}^u(\mathfrak{e}_+^{(0)})$ at just one point, $q_{\varepsilon}^s=q_{\varepsilon}^s(\vec{\tau})$ and $q_{\varepsilon}^u=q_{\varepsilon}^u(\vec{\tau})$ respectively.
Then,  the distance between the invariant manifolds in $\Sigma(\vec{\tau})$ is given by
\begin{equation}\label{distsign}
d(\vec{\tau}):=\left\langle \dfrac{ \nabla \mathcal{H}^{(1)}_0 (\gamma_+(\vec{\tau}))}{\lVert    \nabla \mathcal{H}^{(1)}_0 (\gamma_+(\vec{\tau}))    \rVert}, q_{\varepsilon}^s(\vec{\tau})-q_{\varepsilon}^u(\vec{\tau}) \right\rangle.
\end{equation}
%
%
Application of the classical Melnikov Theory gives the following result.
\begin{lem}\label{melnikovlemma}
The function $d(\vec{\tau})$ introduced in \eqref{distsign} satisfies
\begin{equation*}
d(\vec{\tau})=\dfrac{\varepsilon}{\lVert    \nabla \mathcal{H}^{(1)}_0 (\gamma_+(\vec{\tau}))    \rVert} \,\mathcal{M_+}(\vec{\tau})+\mathcal{O}_{C^1(\mathbf{K})}(\varepsilon^2), \qquad \vec{\tau}\in\mathbf{K},
\end{equation*}
where
\begin{equation}\label{melint}
\begin{aligned}
&\mathcal{M_+}(\vec{\tau}):= \int_{\mathbb{R}} \{ \mathcal{H}^{(1)}_0 ,  \mathcal{H}_1\} \circ \Phi^t_{\mathcal{H}_0}(\gamma_{+}(\vec{\tau}))\,dt = \int_{\mathbb{R}} \{ \mathcal{H}^{(1)}_0 ,  \mathcal{H}_1\} \circ (\gamma_{+}(\tau_1+t,\tau_2+t))\,dt
\end{aligned}
\end{equation}
is the so-called \emph{Melnikov function} (see \cite{Melnikov63}).
\end{lem}

Since the Hamiltonian system \eqref{calham0} is autonomous, the Melnikov function  $\mathcal{M_+}$ depends just on the one-dimensional variable $\tau_1-\tau_2$. That is, there exists a function $\MM_+^{(0)}:\R\to\R$ such that
\[
 \MM_+(\tau_1,\tau_2)=\MM^{(0)}_+(\tau_1-\tau_2).
\]
By Lemma \ref{melnikovlemma}, we will deduce Theorem \ref{teoetero} by proving that there exists a non-degenerate zero of the function $\mathcal{M}^{(0)}_+$ in \eqref{melint}.

It is convenient to introduce the Melnikov potential $\mathcal{L}_{+}\colon\mathbb{R}^2\to\mathbb{R}$, since it is usually easier to compute. It is defined, up to constants, as a primitive of the Melnikov function, namely
\begin{equation*}
\partial_{\tau_1} \mathcal{L}_+ (\vec{\tau})=\mathcal{M}_+ (\vec{\tau}).
\end{equation*}
We have
\begin{equation*}
\mathcal{L}_+(\vec{\tau})=\int_{\mathbb{R}}\mathcal{H}_1\circ \Phi^t_{\mathcal{H}_0}\left(\gamma_+(\vec{\tau})\right)\,dt=\int_{\mathbb{R}}\mathcal{H}_1\left(\gamma_+(\tau_1+t,\tau_2+t)\right)\,dt.
\end{equation*}
Recall that we are assuming \eqref{elc}, which implies $\mathcal{H}_1(\mathfrak{e}_+^{(0)})=\mathcal{H}_1(\mathfrak{e}_+^{(1)})=0$. Therefore, the integrand decays exponentially to zero as $t\to\pm\infty$.


The Melnikov potential satisfies  $\mathcal{L}_+(\vec{\tau})= \mathcal{L}^{(0)}_+(\tau_0)$ where $\tau_0:=\tau_1-\tau_2$ and $ \mathcal{L}^{(0)}_+$ is called \emph{reduced Melnikov potential}.
Then,
\[
\partial_{\tau_0} \mathcal{L}_{+}^{(0)}(\tau_0)=\mathcal{M}^{(0)}_+(\tau_0).
\]
Hence we shall look for non-degenerate critical points of $\mathcal{L}_{+}^{(0)}$, which correspond to non-degenerate zeros of $\mathcal{M}_{+}^{(0)}$.
The following lemma concludes the proof of Proposition \ref{teoetero}.

\begin{lem}\label{lemma:HeteroMelnPot}
There exists a constant $\tilde{\eta}\in\mathbb{R}$ such that  the reduced Melnikov potential  $\mathcal{L}^{(0)}_+$ is given by
\begin{equation}\label{nice}
 \mathcal{L}^{(0)}_+(\tau_0) =\tau_0\,\,\dfrac{(a_1+b_1)\,e^{-\sqrt{3}\tau_0}+(a_2+b_2)}{ 1-e^{-\sqrt{3} \tau_0} }+\tilde{\eta}.
\end{equation}
Therefore, provided \eqref{concord} is satisfied, it possesses a non-degenerate critical point.
\end{lem}

\begin{remark}\label{rem}
Note that $\Upsilon\gamma_+(\vec{\tau})=\gamma_-(-\vec{\tau})$, $i=1, 2$ (see \eqref{gammaplus}) where $\Upsilon$ is the involution introduced in \eqref{anglesymmetry}.Then $\mathcal{L}_+(\vec{\tau})=\mathcal{L}_-(-\vec{\tau})$ and $\mathcal{L}^{(0)}_+(\tau_0)=\mathcal{L}^{(0)}_-(-\tau_0)$ . Therefore, if \eqref{concord} holds, $\mathcal{L}_-^{(0)}$ has a non-degenerate critical point.
\end{remark}

\begin{proof}[Proof of Lemma \ref{lemma:HeteroMelnPot}]
Using the definition of $\mathcal{H}_0$ in \eqref{calham0}, \eqref{elc},
one can write $\mathcal{L}_+$ as
\begin{equation*}
\begin{aligned}
\mathcal{L}_+(\vec{\tau})&= (a_2+b_2)\int_{\mathbb{R}} K_1^+(t+\tau_1)(1-K_2^+(t+\tau_2))\,dt +(a_1+b_1)\int_{\mathbb{R}} K_2^+(t+\tau_2)(1-K_1^+(t+\tau_1))\,dt+\tilde{\eta}\\
&=(a_2+b_2)\int_{\mathbb{R}} K_1^+(s+\tau_0)(1-K_2^+(s))\,ds +(a_1+b_1)\,\int_{\mathbb{R}} K_2^+(s)(1-K_1^+(s+\tau_0))\,ds+\tilde{\eta}\\
&=:\mathcal{L}_+^{(0)}(\tau_0),
\end{aligned}
\end{equation*}
where $\tau_0=\tau_1-\tau_2$ and the constant $\tilde{\eta}\in\R$ is given by
\begin{equation*}
\begin{split}
\tilde{ \eta}:=&\,\int_\R \left(b_1 K_1(t)(K_1(t)-1)+b_2 K_2(t)(K_2(t)-1)\right)\,dt\\
 &+\,\int_\R \left( c_1 K_1(t)(1-K_1(t))\cos\Psi_*+c_2 K_2(t)(1-K_2(t))\cos\Psi_* \right)\,dt.
\end{split}
\end{equation*}
For $i, j=1, 2$ we have (recall \eqref{gammaplus})
 \begin{equation}\label{def:IntegralForMelnikov}
\begin{split}
 \int_{\mathbb{R}} K^+_i(t+\tau_i) (1-K^+_j(t+\tau_j))\,dt&= \int_{\mathbb{R}} \dfrac{e^{-\sqrt{3}(t+\tau_j)}}{(1+e^{-\sqrt{3}(t+\tau_i)})(1+e^{-\sqrt{3}(t+\tau_j)}) }  \,dt\\
 &=(\tau_i-\tau_j) \dfrac{1}{1-e^{-\sqrt{3}(\tau_i-\tau_j)}},
\end{split}\end{equation}
which gives \eqref{nice}.
Therefore,  we have that
\begin{align*}
\lim_{\tau_0\to +\infty} \partial_{\tau_0} \mathcal{L}^{(0)}_+(\tau_0)=a_2+b_2, \qquad \lim_{\tau_0\to -\infty} \partial_{\tau_0}\mathcal{L}^{(0)}_+(\tau_0)=-(a_1+b_1) \label{limelle+}.
\end{align*}
If $(a_1+b_1)(a_2+ b_2)>0$ (see \eqref{concord}) the reduced Melnikov potential $\mathcal{L}^{(0)}_{+}$ has at least one critical point.
Moreover,
\begin{equation*}
\,\partial_{\tau_0}^2\mathcal{L}^{(0)}_+(\tau_0)=-(a_1+b_1+a_2+b_2)
\dfrac{\sqrt{3}}{4}\left( 2-\sqrt{3} \tau_0\coth\left(\dfrac{\sqrt{3}\tau_0}{2} \right)\right) {\mathrm{csch}^2\left(  \dfrac{\sqrt{3} \tau_0}{2} \right) } .
\end{equation*}
By \eqref{concord} this function has constant sign since
\begin{align*}
 2-\sqrt{3}\tau_0\,\coth\left(\dfrac{\sqrt{3}\tau_0}{2} \right)< 0 \qquad \forall \tau_0\neq 0,\\
 \lim_{\tau_0\to 0} \left( 2-\sqrt{3}\tau_0\coth\left(\dfrac{\sqrt{3}\tau_0}{2} \right)\right) {         \mathrm{csch}^2\left(  \dfrac{\sqrt{3} \tau_0}{2} \right) }  =-\frac{2}{3}.
\end{align*}
Therefore $\mathcal{L}_{+}^{(0)}$ is either convex or concave (depending on the sign of $a_1+b_1+a_2+b_2$) and its critical points are non-degenerate.
\end{proof}


\end{proof}
\subsubsection{Transversal homoclinic orbits to saddles: Proof of Theorem \ref{thm:SplittingHomo2Rect}}\label{sec:HomoMelnikov}

We use the computation of the heteroclinic Melnikov potential in Lemma \ref{lemma:HeteroMelnPot} to prove the existence of  homoclinic transversal intersections given by Theorem \ref{thm:SplittingHomo2Rect}.
Since the Hamiltonian \eqref{calham0} with $\eps=0$ does not have connections between $\mathfrak{e}_{\pm}^{(0)}$, we  cannot apply directly Melnikov Theory to obtain such connections for $\eps>0$. Instead, we exploit the usual technique of considering a modified unperturbed Hamiltonian and using two parameters $\eps$ and $\delta$.

We consider  the Hamiltonian
\begin{equation}\label{bfham0}
\begin{aligned}
\mathcal{H} =\mathbf{H}_0+\varepsilon\mathbf{H}_1, \quad \mathbf{H}_0(\psi_1, \psi_2, K_1, K_2)& =\mathbf{H}_0^{(1)}+\mathbf{H}_0^{(2)},\\
\mathbf{H}_0^{(1)}(\psi_1, \psi_2, K_1, K_2)&=K_1(1-K_1)(1+2\, \cos(\psi_1))-\delta K_1^2,\\
\mathbf{H}_0^{(2)}(\psi_1, \psi_2, K_1, K_2)&=K_2(1-K_2)(1+2\, \cos(\psi_2))-\delta K_2^2
\end{aligned}
\end{equation}
and
\begin{equation}\label{trueperturbation}
\begin{aligned}
\mathbf{H}_1(\psi_1, \psi_2, K_1, K_2):=&\,d_{12}\, K_1 \,K_2\\
&+ a_1 K_1+(b_1+1) K_1^2+c_1\,K_1(1-K_1) \cos(\psi_1) \\
&+a_2 K_2+(b_2+1) K_2^2+c_2\,K_2(1-K_2) \cos(\psi_2) .
\end{aligned}
\end{equation}
If one takes $\de=\eps$, this Hamiltonian coincides with \eqref{calham0}. Nevertheless, for now we consider $\delta$ and $\eps$ independent parameters. Later one we will take $\de=\eps$.

If $\delta=0$, then the dynamics of $\mathbf{H}_0$ is the same described in Section \ref{sec:unperturbeddynamics}.
If $\delta\neq 0$, the tori defined in \eqref{tori} are $\mathbf{H}_0$-invariant; moreover, they belong to different energy levels, since
\[
\mathbf{H}_{0_{| \mathbb{T}_0 }}=0, \qquad \mathbf{H}_{0_{| \mathbb{T}_1 }}=-2\delta.
\]
The equilibrium points contained in $\mathbb{T}_0$ are the saddles $\mathfrak{e}_{\pm}^{(0)}$ defined in \eqref{def:saddles}.
 Now we compute the heteroclinic manifold that connects (forward in time) $\mathfrak{e}^{(0)}_+$ with $\mathfrak{e}^{(0)}_-$ (see Figure \ref{fig:PhaseSpaceHomoclinic}). Such orbit corresponds to a homoclinic to the saddle $P_-$ in \eqref{def:saddledown} (expressed in  the ``blow down'' coordinates \eqref{def:cartesian}).

\begin{lem}\label{lemma:UnperturedHomo2Rect}
 The saddles $\mathfrak{e}^{(0)}_+$ with $\mathfrak{e}^{(0)}_-$ of Hamiltonian $\mathbf{H}_0$ in \eqref{bfham0} are connected by a two-dimensional heteroclinic manifold parameterized as
\begin{equation}\label{gammazero}
\begin{aligned}
\gamma_0(\vec{\tau}):&=(\gamma_0^{(1)}(\tau_1), \gamma_0^{(2)}(\tau_2))=(\psi_1^{(0)}(\tau_1), \psi_2^{(0)}(\tau_2), K_1^{(0)}(\tau_1), K_2^{(0)}(\tau_2)),\\
\psi_j^{(0)}(\tau_j):&=2\,\arctan (\Lambda (\tau_j)), \quad  K_j^{(0)}(\tau_j)=\frac{1}{1-\frac{\delta}{3}(1-2\cosh(\sqrt{3}\tau_j))}\quad j=1, 2,
\end{aligned}
\end{equation}
where $\Lambda(t):=-\sqrt{3}\,\, \tanh \left( \dfrac{\sqrt{3}}{2} t \right)$.
\end{lem}

\begin{proof}
Using that $\mathbf{H}_0^{(1)}$ is zero when restricted to $\mathbb{T}_0$ we get
\begin{equation}\label{def:kappa1}
K_1=\frac{1+2 \cos(\psi_1)}{1+2 \cos(\psi_1)+\delta}.
\end{equation}
When the angle $\psi_1\in [-\Psi_*, \Psi_*]$ the numerator in \eqref{def:kappa1} is positive. Hence $K_1\in (0, 1)$ if $\delta>0$.
Plugging \eqref{def:kappa1} in the equation for $\psi_1$ we have
\[
\dot{\psi}_1=-(1+2\cos(\psi_1)),
\]
which leads to
\begin{equation}\label{psi1}
\psi_1(t)=2\,\arctan(\Lambda (t)), \quad \Lambda(t)=-\sqrt{3}\,\, \tanh \left( \dfrac{\sqrt{3}}{2} t \right).
\end{equation}
By using \eqref{def:kappa1} and the trigonometric identity  $\cos(2 \arctan (x))=(1-x^2)/(1+x^2)$ we have
\begin{equation}\label{def:kappasimple}
K_j^{(0)}(t)=\frac{1}{1-\frac{\delta}{3}(1-2\cosh(\sqrt{3}t))}.
\end{equation}
Reasoning in the same way for $(\psi_2, K_2)$ we get that
the homoclinic orbit to $\mathbb{T}_0$ is given by \eqref{gammazero}.
\end{proof}

\begin{figure}
\begin{center}
\includegraphics[width=0.7\textwidth]{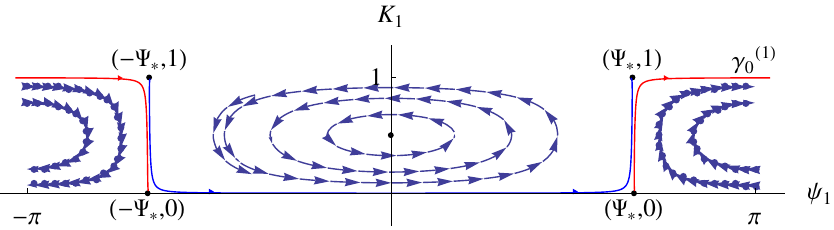}
\end{center}
\caption{{\small Phase space restricted to the $(\psi_1,K_1)$-coordinates for the Hamiltonian $\mathcal{H}$ in \eqref{bfham0}-\eqref{trueperturbation}.} }\label{fig:PhaseSpaceHomoclinic}
\end{figure}

By reasoning as in the proof of Theorem \ref{teoetero} we have that the distance between the manifolds in a suitable section is given by
\begin{equation}\label{dist2}
d(\vec{\tau})=\dfrac{\varepsilon}{\lVert    \nabla \mathbf{H}^{(1)}_0 (\gamma_0(\vec{\tau}))    \rVert} \,\mathcal{M}_0(\vec{\tau})+\mathcal{O}_{C^1(\mathbf{K})}(\varepsilon^2), \qquad \vec{\tau}\in\mathbf{K},
\end{equation}
where the Melnikov function is given by
\begin{equation}\label{melfunz}
\mathcal{M}_0(\vec{\tau})=\int_{\mathbb{R}} \{ \mathbf{H}^{(1)}_0, \mathbf{H}_1\}\circ \Phi^t_{\mathbf{H}_0}(\gamma_0(\vec{\tau}))\,dt.
\end{equation}
It can be easily checked that the $\mathcal{O}_{C^1(\mathbf{K})}(\varepsilon^2)$ are uniform for $\de$ small enough.

The associated  Melnikov potential is
\begin{align*}
\mathcal{L}_0(\vec{\tau})&=\int_{\mathbb{R}}  \mathbf{H}_1\circ \Phi^t_{\mathbf{H}_0}(\gamma_0(\vec{\tau}))\,dt    \label{melpot}
={d_{12}} \int_{\mathbb{R}} K^{(0)}_1(t+\tau_1) \,K^{(0)}_2(t+\tau_2)\,dt+\eta_*
\end{align*}
where
\begin{equation*}
\begin{split}
\eta_*=&\int_{\mathbb{R}} \Big(a_1 K^{(0)}_1(t)+(b_1+1) (K^{(0)}_1(t))^2+c_1\,K^{(0)}_1(t)(1-K^{(0)}_1(t)) \cos\psi^{(0)}_1(t)\Big)\,dt\\
&+\int_{\mathbb{R}} \Big(a_2 K^{(0)}_2(t)+(b_2+1) (K^{(0)}_2(t))^2+c_2\,K^{(0)}_2(t)(1-K^{(0)}_2(t)) \cos\psi^{(0)}_2(t) \Big)\,dt.
\end{split}
\end{equation*}
As before we consider the reduced Melnikov potential
\begin{equation}\label{def:melpotred}
\mathcal{L}_0^{(0)}(\tau_0)={d_{12}} \int_{\mathbb{R}} K^{(0)}_1(s+\tau_0) \,K^{(0)}_2(s)\,ds+\eta_*.
\end{equation}
%
%

%
We want to deduce that $\mathcal{L}_0^{(0)}$ has non-degenerate critical points by using the information on the Melnikov potentials \eqref{nice} of the heteroclinic case.

\begin{prop}\label{propfond}
Fix an interval $\mathcal{I}\subset\mathbb{R}$. There exists $\de_0>0$ such that $\forall \de\in (0, \de_0)$ there exists  a real number $\eta$ and a constant $\nu_0>0$ such that, for $\tau_0\in \mathcal{I}$,
\begin{equation*}
 \mathcal{L}^{(0)}_0(\tau_0)=\eta+d_{1 2}\,\tau_0\,\coth\left( \frac{\sqrt{3}\tau_0}{2} \right)+\mathcal{O}_{C^2(\mathcal{I})}(\de^{\nu_0}).
\end{equation*}
\end{prop}

The proof of this proposition is deferred to Section \ref{def:prop59}.

To complete the proof of Theorem \ref{thm:SplittingHomo2Rect} it is enough  use \eqref{dist2} and Proposition \ref{propfond} and take  $\de=\eps$. Indeed, the transverse homoclinic  points are $\eps$-close to the non-degenerate critical point for $\mathcal{L}_0^{(0)}$. By Proposition \ref{propfond}, $\mathcal{L}_0^{(0)}$ has a non-degenerate critical point $\eps^{\nu_0}$-close to $\tau_0=0$.



\subsubsection{Proof of Proposition \ref{propfond}}\label{def:prop59}

Thanks to the exponential convergence of the homoclinic orbit to the equilibrium points $\mathfrak{e}_\pm^{(0)}$ we have that (recall \eqref{gammazero})
\[
\int_{\mathbb{R}} K_j^{(0)}(t)\,dt<\infty \qquad i=1, 2.
\]
We write $\mathcal{L}^{(0)}_0$ as
\begin{equation}\label{ellezero}
\begin{aligned}
\mathcal{L}^{(0)}_0(\tau_0)=-d_{12}\int_{\mathbb{R}} K^{(0)}_1(s+\tau_0) \,(1-K^{(0)}_2(s))\,ds+\eta_1,
\end{aligned}
\end{equation}
where
\[
 \eta_1=\eta_*+d_{12}\int_\mathbb{R}K^{(0)}_1(s)ds.
\]
Define the function
\[
 F(\psi_1,\psi_2,K_1,K_2)=K_1 (1-K_2).
\]
By \eqref{def:IntegralForMelnikov}, we have that
\[
 \begin{split}
\mathcal{F}_+(\tau_1,\tau_2)&= \int_\mathbb{R}F(\gamma_+(\tau_1+t,\tau_2+t))dt=(\tau_1-\tau_2)\frac{1}{1-e^{-\sqrt{3}(\tau_1-\tau_2)}},\\
\mathcal{F}_-(\tau_1,\tau_2)&  =  \int_\mathbb{R}F(\gamma_-(\tau_1+t,\tau_2+t))dt=-(\tau_1-\tau_2)\frac{1}{1-e^{\sqrt{3}(\tau_1-\tau_2)}},
 \end{split}
\]
which is just the integral of the function $F$ along the heteroclinic orbits $\gamma_\pm$ introduced in \eqref{gammaplus}. These functions satisfies $\mathcal{F}_\pm(\tau_1,\tau_2)=\mathcal{F}_\pm(0,\tau_2-\tau_1)$.

Since the homoclinic orbit \eqref{gammazero} is ``close'' to the concatenation of $\gamma_+$ and $\gamma_-$ in \eqref{gammaplus}, we show that there exists $\nu_0>0$ such that
 the integral in \eqref{ellezero} satisfies
\[
\begin{split}
\int_{\mathbb{R}} K^{(0)}_1(s+\tau_0) \,(1-K^{(0)}_2(s))\,ds&=\mathcal{F}_+(0,\tau_0)+\mathcal{F}_-(0,\tau_0)+\OO\left(\delta^{\nu_0}\right)\\
&=\tau_0 \mathrm{coth} \left( \frac{\sqrt{3}\tau_0}{2} \right)+\OO\left(\delta^{\nu_0}\right).
\end{split}
\]
%
%
The estimate for the error is proved in the following lemma. To state it, we define
\begin{equation}\label{def:oeffe}
\mathfrak{O}_F(\vec{\tau}):=\int_{\mathbb{R}}\left[ F(\gamma_0(\tau_1+t,\tau_2+ t))-F(\gamma_+(\tau_1+t,\tau_2+ t))-F(\gamma_-(\tau_1+t,\tau_2+ t))\right]\,dt.
\end{equation}

%
 \begin{lem}\label{propdiff}
 Let $\mathbf{K}$ be a compact subset of $\mathbb{R}^2$.
 There exists $\delta_0>0$ small, such that $\forall \delta\in (0, \delta_0)$ and $\tau\in\mathbf{K}$ there exists a positive constant $\nu_0\in (0, 1)$ such that the following holds
 \begin{equation}\label{integrand}
  \lVert \mathfrak{O}_F \rVert_{C^0(\mathbf{K})}+ \lVert \partial_{\tau_1} \mathfrak{O}_F \rVert_{C^0(\mathbf{K})}+\lVert \partial^2_{\tau_1} \mathfrak{O}_F \rVert_{C^0(\mathbf{K})}  \lesssim \delta^{\nu_0}.
 \end{equation}
 \end{lem}

This lemma implies Proposition \ref{propfond}. We devote the rest of this section to prove Lemma \ref{propdiff}.

%
%
%
\begin{proof}[Proof of Lemma \ref{propdiff}]
We write the function $\mathfrak{O}_F$ in \eqref{def:oeffe} as
\[
\mathfrak{O}_F(\vec{\tau}):=\int_{\mathbb{R}} F(\gamma_0(\tau_1-p+t, \tau_2-p+t)-F(\gamma_+(\tau_1+t, \tau_2+t)-F(\gamma_-(\tau_1-2 p+t, \tau_2- 2 p+t))\,dt
\]
where
\begin{equation}\label{delta}
p:=\dfrac{1}{\sqrt{3}}\left|\ln\frac{\delta}{3}\right|.
\end{equation}
Note that the shifts by the vector $(p, p)$ do not alter the value of the integral. These shifts are  useful to bound the integrand.
To obtain such estimates, we need the following lemmas.

\begin{lem}\label{difforb}
Let $\sigma_1\in (0, 1)$.
Consider $\gamma_\pm, \gamma_0$ in  \eqref{gammazero}, \eqref{gammaplus} and
\begin{equation}\label{def:interval}
\mathcal{I}:=\left(-\frac{\sigma_1}{\sqrt{3}}\left|\ln\frac{\delta}{3}\right|, \frac{\sigma_1}{\sqrt{3}} \left|\ln\frac{\delta}{3}\right|\right).
\end{equation}
 There exists a constant $\nu\in(0, 1)$ such that
\begin{equation}\label{differenceaction}
\left\lVert K_i^{(0)}\left(\tau_i\mp p+ t\right)-K_i^{\pm}(\tau_i+ t) \right\rVert_{C^0(\mathcal{I}\times\mathbf{K})}\lesssim  \delta^{\nu}, \qquad i=1, 2,
\end{equation}
\begin{equation}\label{differenceangle}
\left\lVert \sin\left(\psi_i^{(0)}\left(\tau_i\mp p+ t\right)\right)-\sin (\psi_i^{\pm}(\tau+ t)) \right\rVert_{C^0(\mathcal{I}\times\mathbf{K})}\lesssim  \delta^{\nu}\qquad i=1, 2.
\end{equation}

\end{lem}

\begin{proof}
To simplify the notation let us consider $i=1$.
By \eqref{def:kappasimple} we have that
\begin{equation*}
K_1^{(0)}\left(t\pm p\right)=\frac{1}{1+e^{\pm\sqrt{3}t}-\frac{\delta}{3}+\frac{\delta^2}{9} e^{\mp\sqrt{3}t}}.
\end{equation*}
Thus for $\sigma\in [-\sigma_1, \sigma_1]$,
\[
 \left\lvert K_1^{(0)}\left(\frac{\sigma}{\sqrt{3}} \left|\ln\frac{\delta}{3}\right|\pm p\right)-K_1^\mp\left( \frac{\sigma}{\sqrt{3}}\left|\ln\frac{\delta}{3}\right|\right)\right\rvert \lesssim \max\{ \delta, \delta^{2-\sigma_1} \}.
 \]
 This gives the bounds \eqref{differenceaction}.
By using \eqref{psi1} and the trigonometric identity $\sin(2\arctan(x))=2 x/(1+x^2)$ we have

\[
 \sin(\psi^{(0)}_1(t))=-\frac{\sqrt{3}\sinh\left(\sqrt{3}t\right)}{2\cosh\left(\sqrt{3}t\right)-1}=-\frac{\sqrt{3}}{2} \tanh(\sqrt{3} t)\Big( 1+\frac{1}{2 \cosh(\sqrt{3} t)-1} \Big).
\]
Then, for $t_\pm=\displaystyle \pm p+\frac{\sigma}{\sqrt{3}}\left|\ln\frac{\delta}{3}\right|$ with $\sigma\in [-\sigma_1, \sigma_1]$, we have
\[
\tanh(\sqrt{3} t_{\pm})=
\pm1+\OO\left(\delta^{2(1-\sigma_1)}\right),\qquad
\frac{1}{2 \cosh(\sqrt{3} t_{\pm})-1}=\OO\left(\delta^{1-\sigma_1}\right).
\]
To prove \eqref{differenceangle} it is enough to use these estimates and \eqref{gammaplus}, to obtain
\[
 \sin(\psi^{(0)}_1(t_\pm))=\mp \frac{\sqrt{3}}{2}+\OO\left(\de^{1-\sigma_1}\right)=\sin (\psi_i^{\mp}(\tau_i))+\OO\left(\de^{1-\sigma_1}\right).
 \]
%

\end{proof}

\begin{lem}\label{convexp}
There exists $T>0$ independent of $\varepsilon$ such that
(see \eqref{gammaplus},  \eqref{gammazero})
\[
\begin{aligned}
\lVert  \gamma_+(\tau_1+t, \tau_2+t)-\mathfrak{e}_+^{(0)} \rVert_{C^0(\mathbf{K})}&\lesssim  \,e^{\sqrt{3} t} \qquad &\text{for}\,\, t<0,&\\
\lVert  \gamma_+(\tau_1+t, \tau_2+t)-\mathfrak{e}_+^{(1)} \rVert_{C^0(\mathbf{K})}&\lesssim \,e^{-\sqrt{3}  t} \qquad &\text{for}\,\, t>0,&\\
\lVert  \gamma_-(\tau_1+t, \tau_2+t)-\mathfrak{e}_-^{(1)} \rVert_{C^0(\mathbf{K})}&\lesssim \,e^{\sqrt{3} t}  \qquad &\text{for}\,\, t<0,&\\
\lVert  \gamma_-(\tau_1+t, \tau_2+t)-\mathfrak{e}_-^{(0)} \rVert_{C^0(\mathbf{K})}&\lesssim \,e^{-\sqrt{3} t}  \qquad &\text{for}\,\, t>0,&\\
\lVert  \gamma_0(\tau_1-p+t, \tau_2-p+t)-\mathfrak{e}^{(0)}_{+} \rVert_{C^0(\mathbf{K})}&\lesssim \,e^{\sqrt{3}t}\qquad &\text{for}\,\, t<-T,&\\
\lVert  \gamma_0(\tau_1+p+t, \tau_2+p+t)-\mathfrak{e}^{(0)}_{-} \rVert_{C^0(\mathbf{K})}&\lesssim \,e^{-\sqrt{3}t}\qquad &\text{for}\,\, t>T&.
\end{aligned}
\]

\end{lem}

\begin{proof}
The lemma follows by straightforward estimates and the hyperbolicity of the equilibria.


\end{proof}

%
%
%


We split $\mathbb{R}=\bigcup_{k=1}^5 \mathcal{I}_k$
\[
\mathcal{I}_1:=(-\infty, \mathtt{a}), \quad \mathcal{I}_2:=[\mathtt{a}, \mathtt{b}], \quad \mathcal{I}_3:=[\mathtt{b}, \mathtt{c}], \quad \mathcal{I}_4:=[\mathtt{c}, \mathtt{d}], \quad \mathcal{I}_5:=(\mathtt{d}, +\infty)
\]
with
\begin{equation*}
\mathtt{a}=-\frac{3}{4\sqrt{3}} \left| \ln\delta\right|,\quad \mathtt{b}=\frac{3}{4\sqrt{3}}\left| \ln\delta\right|, \quad \mathtt{c}=2p-\frac{3}{4\sqrt{3}}\left| \ln\delta\right|, \quad \mathtt{d}=2p+\frac{3}{4\sqrt{3}}\left| \ln\delta\right|.
\end{equation*}
We have
\[
\int_{\mathbb{R}}\left[ F(\gamma( \tau_1-p +t, \tau_2-p+t))-F(\gamma_+(\tau_1+t, \tau_2+t))-F(\gamma_-( \tau_1-2p +t, \tau_2-2p+t))\right]\,dt=\sum_{j=1}^5\mathtt{T}_j
\]
where
\begin{align}
\mathtt{T}_1=&\,\int_{\mathcal{I}_1} \left[ F(\gamma_0( \tau_1-p +t, \tau_2-p+t))-F(\gamma_+(\tau_1+t, \tau_2+t))\right]\,dt   \label{t1} \\
\mathtt{T}_2=&\,\int_{\mathcal{I}_2}\left[  F(\gamma_0( \tau_1-p +t, \tau_2-p+t))-F(\gamma_+(\tau_1+t, \tau_2+t))\right]\,dt     \label{t2}\\
\mathtt{T}_3=&\,\int_{\mathcal{I}_3} F(\gamma_0( \tau_1-p +t, \tau_2-p+t))\,dt-\int_{\mathtt{b}}^{+\infty} F(\gamma_+(\tau_1+t, \tau_2+t))\,dt\label{t3}\\
&- \int^{\mathtt{c}}_{-\infty} F(\gamma_-(\tau_1-2 p+t, \tau_2-2 p+t))\,dt    \notag\\
\mathtt{T}_4=&\,\int_{\mathcal{I}_4} \left[  F(\gamma_0( \tau_1-p +t, \tau_2-p+t))-F(\gamma_-(\tau_1-2 p+t, \tau_2-2 p+t))\right]\,dt \label{t4}\\
\mathtt{T}_5=&\,\int_{\mathcal{I}_5}\left[  F(\gamma_0( \tau_1-p +t, \tau_2-p+t))-F(\gamma_-(\tau_1-2 p+t, \tau_2-2 p+t))\right]\,dt. \label{t5}
\end{align}
By the symmetry of the problem it is sufficient to provide bounds for $\lvert \mathtt{T}_i \rvert$ with $i=1, 2, 3$. The idea is to use the exponentially fast convergence of the orbits $\gamma_0$, $\gamma_{\pm}$ to the saddles (see Lemma \ref{convexp}) to get bounds on the integrals over the unbounded intervals and to exploit the closeness of such orbits on the compact intervals using Lemma \ref{difforb}.
%

\begin{itemize}
\item \textbf{Bound for $\mathtt{T}_1$} (see \eqref{t1}): Let us call $S:=\{ \gamma_0( \tau_1-p +t, \tau_2-p+t) \}_{t\in\mathcal{I}_1}$. $S$ is a compact subset of $\mathbb{R}^2$.
Recalling that $F(\mathfrak{e}^{(0)}_+)=0$, by Lemma \ref{convexp} and the mean value theorem we have that (recall \eqref{gammazero})
\begin{equation}\label{stima1}
\begin{split}
\int_{\mathcal{I}_1} \lvert F(\gamma_0( \tau_1-p +t, \tau_2-p+t)) \rvert\,dt &=\int_{\mathcal{I}_1} \lvert F(\gamma_0( \tau_1-p +t, \tau_2-p+t))-F(\mathfrak{e}^{(0)}_+) \rvert\,dt\\
&\le \lVert F \rVert_{C^1(S)} \int_{\mathcal{I}_1} \lVert \gamma_0( \tau_1-p +t, \tau_2-p+t)-\mathfrak{e}_+^{(0)}\rVert\,dt\\
&\lesssim \delta^{3/4}.
\end{split}
\end{equation}
By Lemma \ref{convexp}, the compactness of the orbit $\{\gamma_+(\tau_1+t, \tau_2+t)\}_{t\in\mathcal{I}_1}$, $F(\mathfrak{e}_+^{(0)})=0$ one can reason in the same way to obtain the same bound for the term involving $\gamma_+$.

\item \textbf{Bound for $\mathtt{T}_2$} (see \eqref{t2}): We note that $\mathcal{I}_2$ is of the form \eqref{def:interval}. By using the compactness of the orbits and the mean value theorem as in the previous step, we can apply Lemma \ref{difforb} and obtain
\begin{equation}\label{stima2}
\lvert \mathtt{T}_2 \rvert\lesssim \delta^{\nu}\,\lvert \mathcal{I}_2 \rvert\lesssim \delta^{\nu}\,\lvert \ln\delta \rvert
\end{equation}
where $\nu\in(0, 1)$ is given by Lemma \ref{difforb}.

\item \textbf{Bound for $\mathtt{T}_3$} (see \eqref{t3}):  We use that $F( 1, 1)=0$. Let us denote $\mathtt{m}:=-(\mathtt{b}-p)=\mathtt{c}-p>0$ (see \eqref{delta} for the definition of $p$). By translating the variable $t$ we obtain
\begin{align*}
\int_{\mathcal{I}_3} \lvert F(\gamma_0(\tau_1- p+t, \tau_2- p+t)) \rvert\,dt& =\int_{\mathcal{I}_3} \lvert F(\gamma_0(\tau_1- p+t, \tau_2- p+t))-F( 1, 1) \rvert\,dt\\
&=\int_{-\mathtt{m}}^{\mathtt{m}}  \lvert F(\gamma_0(\tau_1+t, \tau_2+t))-F(1, 1) \rvert\,dt\\
&\lesssim \lVert F \rVert_{C^1(S)} \sum_{i=1}^2 \int_{-\mathtt{m}}^{\mathtt{m}} | K_i^{(0)} (t)-1 |\,dt.
\end{align*}
Now, using \eqref{def:kappasimple}, one can see that on the  interval $[-\mathtt{m}, \mathtt{m}]$, one has that $| K_i^{(0)} (t)-1 |\lesssim \delta^{3/4}$, which implies
\[
\int_{\mathcal{I}_3} \lvert F(\gamma_0(\tau_1- p+t, \tau_2- p+t)) \rvert\,dt\lesssim \delta^{3/4} \,\lvert \ln\delta\rvert.
\]
By Lemma \ref{convexp} we have
\[
\begin{split}
\int_{\mathtt{b}}^{\infty} \lvert F(\gamma_+(\tau_1+t, \tau_2+t))\rvert\,dt&\lesssim  \delta^{3/4}\\
\int^{\mathtt{c}}_{-\infty} \lvert F(\gamma_-(\tau_1- 2p+t, \tau_2- 2p+t)) \rvert\,dt& = \int^{\mathtt{a}}_{-\infty} \lvert F(\gamma_-(\tau_1+t, \tau_2+t)) \rvert\,dt\lesssim  \delta^{3/4}.
\end{split}
\]
Hence,
\begin{equation}\label{stima3}
\lvert \mathtt{T}_3\rvert\lesssim  \delta^{3/4} \,\lvert \ln \delta \rvert.
\end{equation}
%
\end{itemize}

By \eqref{stima1}, \eqref{stima2}, \eqref{stima3}
we have that
\[
\lVert \mathfrak{O}_F \rVert_{C^0(\mathcal{A})}\lesssim \delta^{\nu}\lvert \ln \delta \rvert
\]
(changing $\nu$ if necessary). Now we observe that (recall \eqref{calham0}, \eqref{bfham0})
\begin{align*}
\partial_{\tau_1} F(\gamma_0(\tau_1- p+t, \tau_2- p+t))&=\{ F, \mathbf{H}^{(1)}_0\} (\gamma_0(\tau_1- p+t, \tau_2- p+t))\\
\partial_{\tau_1} F(\gamma_+(\tau_1+t, \tau_2+t))&=\{ F, \mathcal{H}^{(1)}_0 \}(\gamma_+(\tau_1+t, \tau_2+t)), \\
\partial_{\tau_1} F(\gamma_-(\tau_1-2 p+t, \tau_2-2 p+t))&=\{ F, \mathcal{H}^{(1)}_0\}(\gamma_-(\tau_1-2 p+t, \tau_2-2 p+t)).
\end{align*}
By the particular form of the Hamiltonians $\mathcal{H}^{(1)}_0$ and $\mathbf{H}^{(1)}_0$ in  \eqref{calham0}, \eqref{bfham0} one can check that $\{ F, \mathcal{H}^{(1)}_0\}=\{ F, \mathbf{H}^{(1)}_0\}$. Let us call $G:=\{ F, \mathcal{H}^{(1)}_0\}$. Clearly $G(\mathfrak{e}_\pm^{(0)})=G(\mathfrak{e}_\pm^{(1)})=0$. Now we can repeat the same strategy to get the bounds for the associated $\mathtt{T}_i$. The only difference is that when we compare the orbits $\gamma_0, \gamma_{\pm}$ on compact intervals we need to use also \eqref{differenceangle}. Then we obtain $\lVert \mathfrak{O}_{G} \rVert_{C^0(\mathbf{K})}\lesssim \delta ^{\nu}\lvert \ln \delta \rvert$ (recall \eqref{def:oeffe}).

%
Regarding the second derivatives in $\tau_1$ we have
\begin{align*}
\partial_{\tau_1}^2 F(\gamma_0(\tau_1- p+t, \tau_2- p+t))&=\{ G, \mathbf{H}^{(1)}_0\} (\gamma_0(\tau_1- p+t, \tau_2- p+t))\\
\partial_{\tau_1}^2 F(\gamma_+(\tau_1+t, \tau_2+t))&=\{ G, \mathcal{H}^{(1)}_0 \}(\gamma_+(\tau_1+t, \tau_2+t)), \\
\partial_{\tau_1}^2 F(\gamma_-(\tau_1-2 p+t, \tau_2-2 p+t))&=\{ G, \mathcal{H}^{(1)}_0\}(\gamma_-(\tau_1-2 p+t, \tau_2-2 p+t)).
\end{align*}
We observe that on a compact set $\lvert \{G, \mathbf{H}^{(1)}_0\}-\{ G, \mathcal{H}^{(1)}_0\}\rvert\lesssim \delta$. Then we can consider the function $E:=\{G, \mathcal{H}^{(1)}_0\}$ and repeat the same arguments above to prove that $\lVert \mathfrak{O}_{E} \rVert_{C^0(\mathbf{K})}$ has a bound like \eqref{integrand}. We conclude by noting that
\[
 \lVert \mathfrak{O}_F \rVert_{C^0(\mathbf{K})}+ \lVert \partial_{\tau_1} \mathfrak{O}_F \rVert_{C^0(\mathbf{K})}+\lVert \partial^2_{\tau_1} \mathfrak{O}_F \rVert_{C^0(\mathbf{K})}    \lesssim \Vert \mathfrak{O}_F \rVert_{C^0(\mathbf{K})}+\lVert \mathfrak{O}_{G} \rVert_{C^0(\mathbf{K})}+\lVert \mathfrak{O}_{E} \rVert_{C^0(\mathbf{K})}.
\]
 \end{proof}

\subsection{Transversal homoclinic orbits to saddles: $N$ resonant tuples}

In this section we prove the generalization of Theorem \ref{thm:SplittingHomo2Rect} for the case of multiple resonant tuples. To break integrability we need to impose a non-degeneracy condition on the coefficients $d_{ij}$ in \eqref{coeffN}. To state it we introduce the matrix
\begin{equation}\label{def:MatrixD}
\mathcal{D}
=
\begin{pmatrix}
d_{1,N} + \sum_{j \neq 1} d_{1,j} & -d_{1,2} & \ldots & -d_{1,N-1} \\
-d_{2,1} & \ddots & \vdots & \vdots \\
\vdots & \vdots & \ddots &  -d_{N-2,N-1} \\
-d_{N-1,1} & \ldots & -d_{N-1,N-2} & d_{N-1,N} + \sum_{j \neq N-1} d_{N-1,j}
\end{pmatrix}.
\end{equation}

\begin{prop}\label{prop:homointersaddlesN}
Assume that the matrix $\mathcal{D}$ satisfies
\begin{equation}\label{cond:HessD}
\det \mathcal{D}\neq 0.
\end{equation}
Then, there exists $\varepsilon_0>0$ such that for all $\varepsilon\in(0, \varepsilon_0)$  the invariant manifolds $W_{\varepsilon}^-(\mathfrak{e}^{(0)}_-)$ and $W_{\varepsilon}^+(\mathfrak{e}^{(0)}_+)$ of the saddles \eqref{saddlesN} of the Hamiltonian \eqref{def:hamN} intersect transversally along an orbit (within the energy level).
\end{prop}

\begin{remark}
Note that condition \eqref{cond:HessD} is satisfied for a generic choice of coefficients $d_{ij}$. Indeed, the determinant of such matrix is a polynomial in the variables $d_{i j}$. Then, it is enough to  show that such polynomial is not identically zero. If one consider
\[
d_{i j}=\begin{cases}
1 \quad \text{if}\,\,i=1, \dots, N-1,\,\,j=N,\\
0 \quad \text{otherwise}
\end{cases}
\]
the matrix \eqref{def:MatrixD} is a multiple of the identity. This means that at some point the polynomial is not zero and therefore it is not-zero for almost every choice of $d_{ij}$.
In Section \ref{sec:ApplicationMultiple}, we prove that condition \eqref{cond:HessD} is satisfied for the resonant models associated to the Wave, Beam and Hartree equations that we consider.
\end{remark}

\begin{proof}
We proceed as for the case $N=2$ in Section \ref{sec:HomoMelnikov}. That is, we introduce a second parameter $\de$ and we define the Hamiltonian $\mathcal{H}=\sum_{j=1}^N\mathbf{H}_0^{(j)}+\varepsilon \mathbf{H}_1$ given by (recall \eqref{coeffN})
\begin{equation}\label{def:hamN}
\begin{aligned}
\mathbf{H}_0^{(j)}(\psi_j, K_j; \delta) :=&\,K_j (1-K_j)(1 +2  \,\cos(\psi_j))-\delta\, K^2_j ,\\
\mathbf{H}_1(\psi_1, \dots, \psi_N, K_1, \dots, K_N) :=&\,\sum_{j=1}^N \big(a_{j} K_j + (b_{j}+1) K^2_j +c_j\,K_j (1-K_j)\,\cos(\psi_j)\big)\\
&\,+ \sum_{i, j=1, i< j}^N d_{i j} K_i\,K_j.
\end{aligned}
\end{equation}
If $\de=\eps$, it coincides with \eqref{redhamN}.

We proceed as in the proof of Theorem \ref{thm:SplittingHomo2Rect}. For $\varepsilon=0$ the dynamics is the same described in Section \ref{sec:HomoMelnikov}.
In particular it is easy to see that, when $\varepsilon=0$, one can consider the two saddle points (recall that $\Psi_*:=2\pi/3$)
\begin{equation}\label{saddlesN}
\mathfrak{e}^{(0)}_{\pm}:=(\pm \Psi_*, \dots, \pm \Psi_*, 0, \dots, 0)
\end{equation}
connected by the $\de$-dependent homoclinic manifolds (recall \eqref{psi1}, \eqref{gammazero})
\begin{equation}\label{gammazeroN}
\gamma_0(\vec{\tau}):=(\psi_1^{(0)}(\tau_1), \ldots, \psi_N^{(0)}(\tau_N), K_1^{(0)}(\tau_1), \ldots, K_N^{(0)}(\tau_N)),\qquad \vec{\tau}=(\tau_1,\ldots,\tau_N).\nonumber \\
\end{equation}
We define the associated Melnikov potential
\begin{equation}\label{def:malnikovpotcaseN}
\mathcal{L}_{0, N}(\vec{\tau}):=\int_{\mathbb{R}} \mathbf{H}_1\circ \Phi^t_{\mathbf{H}_0} (\gamma_0(\vec{\tau}))\,dt= \sum_{i, j=1, i< j}^N d_{i j} \int_{\mathbb{R}}  K^{(0)}_i(\tau_i+t)\,K^{(0)}_j(\tau_j+t) \,dt+\eta_*,
\end{equation}
where
\[
\eta_*:=\sum_{i=1}^N\int_{\mathbb{R}} a_{i} K^{(0)}_i(t) + (b_{i}+1) (K^{(0)}_i)^2(t) +c_i\,K^{(0)}_i (1-K^{(0)}_i(t))\,\cos\psi^{(0)}_i(t)\,dt.
\]
We note that such function is the sum of terms of the form \eqref{def:melpotred}.
Thanks to the autonomous nature of the system the potential, $\mathcal{L}_{0, N}$ depends just on $\tau_1-\tau_N, \dots, \tau_{N-1}-\tau_N$. Thus one can consider the reduced Melnikov potential $\mathcal{L}_{0, N}^{(0)}$, which satisfies
\[
 \mathcal{L}_{0, N}^{(0)}\left(\tau_1-\tau_N, \dots, \tau_{N-1}-\tau_N\right)=\mathcal{L}_{0, N}(\tau_1,\ldots,\tau_N).
\]
Classical Melnikov Theory ensures that  non-degenerate critical points of this reduced Melnikov potential gives rise to transversal (within the energy level) intersections between $W_{\varepsilon}^-(\mathfrak{e}^{(0)}_-)$ and $W_{\varepsilon}^+(\mathfrak{e}^{(0)}_+)$.

Denoting  $\,\,{\tilde{\tau}}:= (\tau_1-\tau_N, \dots, \tau_{N-1}-\tau_N)$, Proposition \ref{propfond} implies that there exists a constant $\eta\in\mathbb{R}$ such that
\[
\mathcal{L}_{0, N}^{(0)}(\tilde{\tau})=\eta+\sum_{i, j=1, i< j}^{N-1} d_{i j}\, (\tilde{\tau}_i-\tilde{\tau}_j)\coth\left(\frac{\sqrt{3}}{2} (\tilde{\tau}_i-\tilde{\tau}_j) \right)+\sum_{j=1}^{N-1} d_{j N} \tilde{\tau}_j \coth\left(\frac{\sqrt{3}}{2} \tilde{\tau}_j \right)+\mathcal{O}_{C^2}(\delta^{\nu_0})
\]
for some $\nu_0>0$.
Since $x\,\coth((\sqrt{3}/2)x)$ is an even function, the origin
 $(0, \dots, 0)\in\mathbb{R}^{N-1}$ is a critical point of the first order of  $\mathcal{L}_{0, N}^{(0)}$ (that is, dropping the errors $\mathcal{O}_{C^2}(\de^{\nu_0})$). The  Hessian matrix of the first order of $\mathcal{L}_{0, N}^{(0)}$ at the origin is
\[\mathrm{Hess}=  \frac{1}{\sqrt{3}}\mathcal{D}\]
where $\mathcal{D}$ is the matrix introduced in \eqref{def:MatrixD}. Then, condition \eqref{cond:HessD} implies $\det\mathrm{Hess}\neq 0$.

The non-degeneracy of the Hessian implies that the reduced Melnikov potential $\mathcal{L}_{0, N}^{(0)}$ has a non-degenerate critical point $\de^{\nu_0}$--close to $\tilde\tau=0$. Then, taking $\de=\eps$ one can use  classical Melnikov Theory to ensure the existence of the transverse intersection between invariant manifolds stated in Proposition \ref{prop:homointersaddlesN}.
\end{proof}

\section{Proof of Theorem \ref{TeoWaveBeam}}\label{sec:Proofthm1}

The goal of this section is to prove Theorem \ref{TeoWaveBeam}. The key point of the proof is to construct symbolic dynamics (an infinite symbols Smale horseshoe) for the resonant model \eqref{calham0} which has been derived from the equations \eqref{Hartree}, \eqref{Wave}, \eqref{Beam}. In Theorem \ref{thm:SplittingHomo2Rect} we have constructed transverse homoclinic orbits to saddles for \eqref{calham0}. It is well known that the intersection of invariant manifolds of critical points in flows do not always lead to the existence of symbolic dynamics (see, for instance, \cite{Devaney78}).
Therefore, the first step of the proof is to
obtain transverse homoclinic points to certain periodic orbits. This is done in Section \ref{sec:Thm1:MelnikovPO}. Then, following \cite{Moser01}, in Section \ref{sec:Thm1:SymbDyn} we construct an invariant set of (a suitable Poincar\'e map of) the flow associated to the Hamiltonian \eqref{calham0} whose dynamics is conjugated to a shift of infinite symbols (see Section \ref{sec:heuristics}). Finally in Section \ref{sec:ProofThm1:Application} we complete the proofs of Theorem  \ref{TeoWaveBeam} by checking that the non-degeneracy conditions imposed on \eqref{calham0} are satisfied for the resonant models obtained from the PDEs \eqref{Wave}, \eqref{Hartree} and \eqref{Beam}.


%

\subsection{Transversality of invariant manifolds of periodic orbits}\label{sec:Thm1:MelnikovPO}
The main result in this section is the following.

\begin{prop}\label{prop:SplittingPO2Rect}
Consider the Hamiltonian \eqref{calham0} and assume that \eqref{tildekappas} holds.
Then there exists $\varepsilon_0>0$ such that for all $\varepsilon\in (0, \varepsilon_0)$ there exists $h_0=h_0(\varepsilon)>0$ such that for all $h\in (0, h_0)$,
\begin{enumerate}
\item[$(i)$] The hyperbolic  periodic orbits $\mathtt{P}^{\pm, 0 }_h$ in \eqref{def:PO:ext}
persist and have period $\mathtt{T}_h$. That is, Hamiltonian \eqref{calham0} has
hyperbolic periodic orbits  $\mathtt{P}^{\pm, 0 }_{h,\eps}$ that are
$\mathcal{O}(\eps)$--close to $\mathtt{P}^{\pm, 0}_h$.
\item[$(ii)$] The invariant manifolds   $W^u(\mathtt{P}^{+, 0 }_{h,\eps})$ and
$W^s(\mathtt{P}^{-, 0 }_{h,\eps})$ intersect transversally along orbits  (within the
energy level).
\end{enumerate}
\end{prop}

Note that in the coordinates introduced in \eqref{def:cartesian}, the periodic orbits $\mathtt{P}^{+, 0 }_{h,\eps}$ and $\mathtt{P}^{-, 0 }_{h,\eps}$ blow down to the same periodic orbit, which we denote by $\mathtt{P}^{0 }_{h,\eps}$. In the coordinates  \eqref{def:cartesian}, Proposition \ref{prop:SplittingPO2Rect} can be restated as that the  manifolds $W^u(\mathtt{P}^{0 }_{h,\eps})$ and $W^s(\mathtt{P}^{0 }_{h,\eps})$ intersect transversally
within the energy level.

\begin{proof}[Proof of Proposition \ref{prop:SplittingPO2Rect}]
To prove $(i)$ is more convenient to use the cartesian coordinates $\{x_j,y_j\}$ in \eqref{def:cartesian} and therefore Hamiltonian $\widetilde H_\Res$ in \eqref{redhamNCartesian} (with $N=2$) to avoid the blow up of $K_j=0$. Then, the invariant subspace $\{K_2=0\}$ corresponds to $\{x_2=y_2=0\}$. The Hamiltonian on this invariant subspace is given by
\begin{equation*}
\begin{aligned}
\mathcal{H}( x_1, 0,y_1, 0)=& \frac{1}{2}\left(3x_1^2-y_1^2\right) - \frac{1}{4}
\left(3x_1^2-y_1^2\right)\left(x_1^2+y_1^2\right)\\
&+\varepsilon\Big[ \frac{1}{2} a_{1} \left(x_1^2+y_1^2\right)
+\frac{1}{4} b_{1} \left(x_1^2+y_1^2\right)^2  +\frac{1}{4}c_1
\left(x_1^2+y_1^2\right)\left(2-x_1^2-y_1^2\right)\Big].
\end{aligned}
\end{equation*}
This Hamiltonian is integrable both for $\eps=0$ and  $\varepsilon>0$ and has the saddle $(0,0)$ at the energy level. Integrability and the particular form of $\mathcal{H}$ implies that the energy levels close to zero are given by periodic orbits. These periodic orbits are $\eps$-close to those of the unperturbed problem
(see \eqref{calham0}).

To  prove $(ii)$ we proceed as in Section \ref{sec:resonantmodel} by doing approximations of several Melnikov functions and using an auxiliary parameter $\de$. We follow the notation of Section \ref{sec:HomoMelnikov}, In particular,  we consider the Hamiltonians $\mathbf{H}_0$, $\mathbf{H}_1$ in \eqref{bfham0}, which taking $\de=\eps$ also define the Hamiltonian $\mathcal{H}$.

By the particular form of  Hamiltonian $\mathbf{H}_0^{(j)}$,
$j=1,2$ (see \eqref{bfham0}, it can be easily checked that it  has the saddles
%
$(\pm \Psi_*,0)$ (they correspond to $x_1=y_1=0$ in the blow down coordinates \eqref{def:cartesian}. These saddles are connected by the homoclinic orbits $\gamma_0^{(j)}$, $j=1, 2$, introduced in \eqref{gammazero}.

Let $h>0$ small, then the Hamiltonian $\mathbf{H}_0$ possesses
%
%
the hyperbolic periodic orbits
\[
\mathtt{P}^{\pm,0}_{\delta,h} =\left\{\big(\gamma_{\delta,h}^{(1)}(\tau),\pm \Psi_*,0 \big) : \tau\in\mathbb{R}\right\},
\]
where $\gamma_{\delta,h}^{(k)}$ is the time parametrization of the periodic orbit defined by
$\{\mathbf{H}_0^{(k)}=h\}$ (see Figure \ref{fig:PhaseSpaceHomoclinic}).
When $\varepsilon = 0$, the homoclinic manifold
$W^u_0(\mathtt{P}_{\delta,h}^{+,0})\equiv W^s_0(\mathtt{P}_{\delta,h}^{-,0})$ is
parameterized by
\begin{equation}\label{man}
\Gamma_{\delta,h,0}(\vec{\tau}) := (\gamma_{\delta,h}^{(1)}(\tau_1),\gamma_{0}^{(2)}(\tau_2)).
\end{equation}


\begin{remark}\label{rem:uniformconv}
The periodic orbits $\mathtt{P}^{\pm,0}_{\delta,h}$ converge pointwise for any fixed $\tau$  to $(\gamma_0^{(1)}(\tau),\pm \Psi_*,0)$ as $h\to 0$.
Similarly, for fixed $\tau$, the parametrization $\Gamma_{\delta,h,0}(\vec{\tau})$ converges
to $\gamma_0(\vec{\tau})$ in \eqref{gammazero} as $h\to 0$.
\end{remark}
 When $\varepsilon>0$, the periodic orbits  $\mathtt{P}_{\delta,h}^{\pm,0}$ persist
 .
Direct application of Melnikov Theory, as in Section
\ref{sec:resonantmodel}, ensures the following. There exists $\de_0>0$,
$\eps_0>0$ small enough such that for any $\de\in(0,\de_0)$ and $\eps\in
(0,\eps_0)$ small enough, the distance function between the invariant manifolds
$W^u(\mathtt{P}^{+, 0 }_{h,\eps})$ and $W^s(\mathtt{P}^{-, 0 }_{h,\eps})$ in a well
chosen transversal section is given by
\[
 d(\vec{\tau})=\eps
\frac{\mathcal{M}_h(\vec{\tau})}{\lVert\nabla\mathbf{H}_0^{(2)}(\gamma_{0}^{(2)}
(\tau_2)) \rVert}+\OO\left(\eps^2\right),
\]
where the error $\OO(\eps^2)$ is uniform in $\de$ and $h$ and $\mathcal{M}_h$ is the
Melnikov function given by
\begin{equation}\label{melfunz2}
\mathcal{M}_h(\tau_1,\tau_2)=\int_\mathbb{\R}\left\{\mathbf{H}_0^{(2)},\mathbf{H}
_1\right\} (\Gamma_{\delta,h,0}(\tau_1+t,\tau_2+t))dt.
\end{equation}
We note that as $t\to\pm\infty$ the $K_2$-component of $\Gamma_{\delta,h,0}(\tau_1+t,\tau_2+t)$ goes exponentially fast to zero, then by a direct computation it is easy to see that
\[
\lim_{\tau_2 \to \pm\infty} \{ \mathbf{H}_0^{(2)}, \mathbf{H}_1 \}(\gamma_{\delta,h}^{(1)}(\tau_1),\gamma_{0}^{(2)}(\tau_2))=0
\]
with exponentially fast convergence and \eqref{melfunz2} is well defined.
To obtain the non-degeneracy of the zeros of the Melnikov function, we compare
\eqref{melfunz2} to the Melnikov function \eqref{melfunz} associated to the homoclinic
orbits to the saddles $\mathfrak{e}_{\pm}^{(0)}$.

Let us consider the reduced Melnikov functions $\mathcal{M}_h^{(0)}(\tau_0)=\mathcal{M}_h(\tau_1, \tau_2)$ and $\mathcal{M}_0^{(0)}(\tau_0)=\mathcal{M}_0(\tau_1, \tau_2)$, where $\tau_0:=\tau_1-\tau_2$ (recall \eqref{melfunz}).

\begin{lem}\label{lem:periodicomo}
Let $\mathbf{K}\subset\mathbb{R}$ be a closed interval and
 let us define
\[
\mathfrak{O}_h(\tau_0):=\mathcal{M}_0^{(0)}(\tau_0)-\mathcal{M}_h^{(0)}(\tau_0).
\]
There exists $\de_0>0$ small such that $\forall \de\in (0, \de_0)$ there exist $h_0=h_0(\de)$, positive and small, such that $\forall h\in (0, h_0)$ there exists $\nu_*\in (0, 1)$ such that the following holds
\begin{equation*}
 \lVert \mathfrak{O}_h \rVert_{C^0(\mathbf{K})}+ \lVert \partial_{\tau_0} \mathfrak{O}_h \rVert_{C^0(\mathbf{K})} \lesssim \de^{\nu_*}.
\end{equation*}
\end{lem}

\begin{proof}
We consider the splitting
\begin{align}
\mathfrak{O}_h(\tau_0)=& \int_{\lvert t \rvert>\mathtt{c}|\ln\de|}  \{ \mathbf{H}_0^{(2)}, \mathbf{H}_1 \}(\gamma_0(t,\tau_0+t))- \{ \mathbf{H}_0^{(2)}, \mathbf{H}_1 \}(\Gamma_{\delta,h,0}(t,\tau_0+t))\,dt
\label{pippo1} \\ \label{pippo2}
&+\int_{\lvert t \rvert\le \mathtt{c}|\ln\de|}  \{ \mathbf{H}_0^{(2)}, \mathbf{H}_1 \}(\gamma_0(t,\tau_0+t))- \{ \mathbf{H}_0^{(2)}, \mathbf{H}_1 \}(\Gamma_{\delta,h,0}(t,\tau_0+t))\,dt,
\end{align}
where $\mathtt{c}$ is some positive constant. We observe that
\[
\{ \mathbf{H}_0^{(2)}, \mathbf{H}_1\}_{|_{\{K_1=0\}}}=0 \qquad \mbox{and} \qquad \mathtt{P}_{\delta, h}^{\pm,0}\subset \{ K_2=0\}.
\]
Then, by the exponential convergence of the flow to the hyperbolic saddles $\mathfrak{e}^{\pm}_0$ and the hyperbolic periodic orbits $\mathtt{P}_{\delta, h}^{\pm,0}$,
 the term on the r. h. s. of \eqref{pippo1} is bounded by $C\de^{\mathtt{d}}$ where $C, \mathtt{d}>0$ are two constants independent of $h$. Let us call $\mathcal{I}:=[-\mathtt{c}\lvert \ln\de \rvert, \mathtt{c}\lvert \ln\de\rvert]$. By Remark \ref{rem:uniformconv} we have
\[
\lim_{h\to 0} \sup_{(\tau_0, t)\in \mathbf{K}\times \mathcal{I}} \lvert \gamma_0(t,\tau_0+t)-\Gamma_{\delta,h,0}(t,\tau_0+t) \rvert=\lim_{h\to 0} \sup_{(\tau_0, t)\in \mathbf{K}\times \mathcal{I}} \lvert  (\gamma_0^{(1)}(t)-\gamma_{\delta, h}^{(1)}(t),0) \rvert=0.
\]
Hence there exists $h_0=h_0(\de)>0$ such that if $h\in (0, h_0)$ then \eqref{pippo2} is bounded, up to constant factors, by $\de$. The derivative $\partial_{\tau_0} \mathfrak{O}_h$ has the expression \eqref{pippo1}, \eqref{pippo2} with the double Poisson $\{ \mathbf{H}_0^{(2)}, \{ \mathbf{H}_0^{(2)}, \mathbf{H}_1 \} \}$ instead of $\{ \mathbf{H}_0^{(2)}, \mathbf{H}_1 \}$. Clearly it is still true that this Poisson vanishes at $\{K_2=0\}$. Then one can repeat the same argument to get a bound as for $\mathfrak{O}_h$.
\end{proof}

Lemma \ref{lem:periodicomo} and Proposition \ref{propfond} imply that $\mathcal{M}_h$ has a non-degenerate zero. Then, proceeding as in Section \ref{sec:HomoMelnikov} and taking $\eps=\de$ one obtains transverse heteroclinic orbits between the periodic orbits $\mathtt{P}^{+, 0 }_{h,\eps}$ and $\mathtt{P}^{-, 0 }_{h,\eps}$.
This
completes the proof of Proposition \ref{prop:SplittingPO2Rect}.
\end{proof}

\subsection{Symbolic dynamics of infinite symbols}\label{sec:Thm1:SymbDyn}
To construct symbolic dynamics for Hamiltonian \eqref{calham0} we consider a
section transverse to the flow, within a given energy level, and the associated
Poincar\'e map.

We proceed as
in \cite{Moser01}. Fix $0<\eps\ll1$ and $0<h\ll 1$.  We build an invariant set with points
arbitrarily close
to the transverse homoclinic orbit to the $\mathtt{T}_h$-periodic orbit
$\mathtt{P}^{0 }_{h,\eps}$ obtained in Proposition \ref{prop:SplittingPO2Rect}.

We define the section within the energy level $\{\mathcal{H}=h\}$,
\begin{equation}\label{def:SectionSh}
\mathcal{S}_h=\{K_2=1/2,\,\, \psi_2\in (-2\pi/3-m,-2\pi/3+m),\,\,
\mathcal{H}(\Psi_1,\Psi_2,K_1,K_2)=h \},
\end{equation}
for some small $m>0$ (see Figure~\ref{fig:horseshoe}). This section is transverse to the unperturbed flow
($\eps=0$) and therefore also transverse, for $\eps>0$ small enough, to the perturbed one. In particular, by Proposition \ref{prop:SplittingPO2Rect}, it contains points in $W^u(\mathtt{P}^{0 }_{h,\eps})\cap W^s(\mathtt{P}^{0 }_{h,\eps})$ (classical perturbative arguments ensure that the perturbed invariant manifolds are $\OO(\eps)$ to the unperturbed ones).

Denote by $\Phi^t_\mathcal{H}$ the flow associated to the Hamiltonian \eqref{calham0}. For a point
$z\in\mathcal{S}_h$, we define $T(z)>0$ the first (forward) return time of the trajectory $\Phi^t_\mathcal{H}(z)$ to this
section whenever it is defined. For those points whose forward
trajectory never hits again $\mathcal{S}_h$ we can take $T(z)=+\infty$. Note that
this happens in particular for the points in $W^s(\mathtt{P}^{-, 0 }_{h,\eps})$
(note that by the perturbative results in Section \ref{sec:Thm1:MelnikovPO} this
intersection is not empty).

Then, we define the open set $\mathcal{U}\subset\mathcal{S}_h$ as
$\mathcal{U}=\{z\in\mathcal{S}_h:T(z)<+\infty\}$ and the associated Poincar\'e
map $\mathtt{P}\colon \mathcal{U}\subset\mathcal{S}_h\to\mathcal{S}_h$ defined by
\[
\mathtt{P}(z):=\Phi^{T(z)}_{\mathcal{H}} (z).
\]

\begin{figure} 
\begin{center}
\includegraphics[width=1.\textwidth]{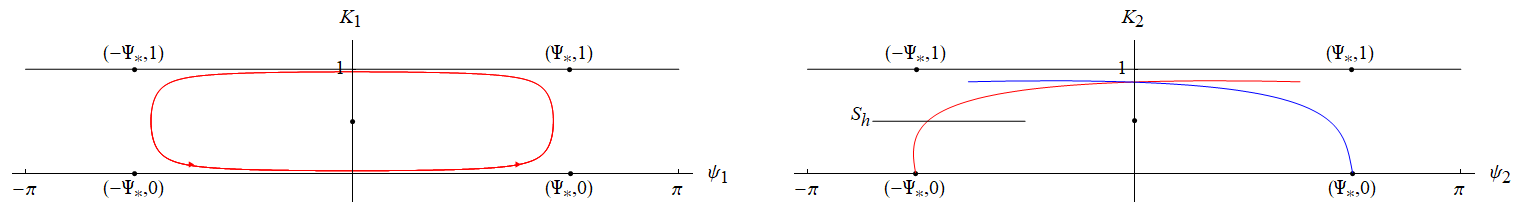}
\end{center}
\caption{The periodic orbit in the $(\Psi_1,K_1)$-plane, its invariant manifolds and the section $\mathcal{S}_h$.}
\label{fig:horseshoe}
\end{figure}

\begin{prop}[Existence  of Horseshoe]\label{horseshoeteo}
 Assume \eqref{tildekappas}.
Then there exists $\varepsilon_0>0$ such that $\forall \varepsilon\in (0,
\varepsilon_0)$ the Poincar\'e map $\mathtt{P}$ possesses an invariant set
$Y\subset\mathcal{U}$ whose dynamics is conjugated to the infinite symbols
shift. Namely, there exists a homeomorphism $h:\Sigma\to Y$, where
\[
\Sigma=\N^\Z=\left\{\{\omega_k\}_{k\in\Z}:\omega_k\in\N\right\},
\]
such that $\mathtt{P}|_Y=h\circ\sigma\circ h^{-1}$ where
$\sigma:\Sigma\to\Sigma$ is the shift, that is
\[
 (\sigma\omega)_k=\omega_{k+1},\quad k\in\Z.
\]
Moreover, $h^{-1}$ can be defined as follows. Fix $z^*\in Y$ and define
$\omega^*=h^{-1}(z^*)$. Associated to $z$ one can define the sequence of
hitting times
\[
t_0=0,\qquad t_k=T\left(\mathtt{P}^{k-1}(z)\right)\quad\text{for}\quad k\geq
1,\qquad  t_k=T\left(\mathtt{P}^{k}(z)\right)\quad\text{for}\quad k\leq -1.
\]
Then, there exists $C^*\in\mathbb{N}$ independent of $z^*$ such that
\begin{equation}\label{def:omegastar}
 \omega^*_k=\left\lfloor\frac{t_k-t_{k-1}}{\mathtt{T}_h}\right\rfloor-C^*
\end{equation}
where $\mathtt{T}_h$ is the period of the periodic orbit $\mathtt{P}^{\pm, 0
}_{h,\eps}$.
 \end{prop}

This proposition gives symbolic dynamics for a Poincar\'e map associated to
Hamiltonian \eqref{calham0}. Note that it is constructed in a way that higher
symbols in $\Sigma$ imply longer return times. In particular those can be
unbounded. The proof of this proposition follows the same lines as the
construction of symbolic dynamics done  by Moser in Chapter 3 of
\cite{Moser01}. Note that the natural $C^*$ in \eqref{def:omegastar} is just to normalize and have as symbols $\N$ (since the horseshoe is build close the homoclinic orbit, the hitting times satisfy $|t_k-t_{k-1}|\gg 1$).

We remark that condition \eqref{tildekappas} is necessary, indeed
the term that breaks the integrability in the Hamiltonian \eqref{calham0} has
the form
$d_{12} \,K_1\,K_2$ (see for instance \eqref{bfham0},
\eqref{trueperturbation}). Hence if the condition \eqref{tildekappas} does not
hold then the Hamiltonian \eqref{calham0} is integrable.

\subsection{Application to the Wave, Beam and Hartree
equations}\label{sec:ProofThm1:Application}

To proof  Theorem \ref{TeoWaveBeam} (and also the result for the Hartree equation \eqref{Hartree}) by applying Proposition \ref{horseshoeteo},  one needs to
check that the condition  \eqref{tildekappas} is satisfied by the resonant models
derived from the Hartree, Beam and Wave equations.
To thus end, recall the definitions \eqref{def:commonham}, \eqref{def:Cj1}, \eqref{def:Cj2} and the symplectic reduction performed in Section \ref{sectionsymplecticreduction}.
Next lemmas check condition \eqref{tildekappas}
under the hypotheses considered for these three equations.

\begin{lem}\label{lemma:WaveBeamMelnikovCond}
Let us consider Hamiltonian \eqref{calham0} associated to either the Wave
equation \eqref{Wave} or the Beam equation \eqref{Beam} and to a set
$\Lambda$ satisfying Proposition \ref{prop:LambdaSetWaveHom} . Then, the
condition  \eqref{tildekappas} is satisfied.
\end{lem}

\begin{lem}\label{lemma:HartreeMelnikovCond}
Let us consider Hamiltonian \eqref{calham0} associated to the Hartree equation
\eqref{Hartree} with a potential $V$ as in \eqref{def:potential} and to a set $\Lambda$
satisfying Proposition \ref{prop:LambdaSetHartreeHet}
Then for a generic choice of the $\{\gamma_{n}\}_{n\in\Lambda}$, the condition  \eqref{tildekappas}  is satisfied.
\end{lem}

These lemmas, together with Proposition \ref{horseshoeteo}, complete the proof of Item $(i)$ of Theorem \ref{thm:geometric}.

\begin{proof}[Proof of Lemma \ref{lemma:WaveBeamMelnikovCond}]
Recall \eqref{def:g}, \eqref{coeffN}, \eqref{def:coeffmatrixA}, \eqref{def:Cj2}.
For the Wave and Beam equations \eqref{Wave}, \eqref{Beam},
\begin{equation*}
 d_{12}=\frac{3}{32 \mathtt{g}}\,\sum_{\substack{1 \leq i \leq 4 \\ 5 \leq j \leq 8}}
\frac{(-1)^{i+j}}{|n_i|^{\kappa} |n_j|^\kappa}=\frac{3}{32 \mathtt{g}}\left(\sum_{i=1}^4\frac{(-1)^i}{|n_i|^{\kappa}}\right)\left(\sum_{j=5}^8\frac{(-1)^j}{|n_j|^{\kappa}}\right)
 \end{equation*}
where $\kappa=1$ for the Wave equation and  $\kappa=2$ for the Beam equation.

We write $d_{12}$ in a different form. To this end, we introduce the following notations. For each finite set of indexes
$I=\{i_1,\ldots,i_n\} \subset \{ 1, \dots, 8 \}$ and for any pair of positive
integers $i_1, i_2 \in \{ 1, \dots, 8 \}$, we define
\begin{equation}\label{def:IndexesPiDelta}
\prod^{I}= \frac{1}{\prod_{k=1}^n |n_{i_k}|^\kappa},\qquad
\Delta_{i_1,i_2}= |n_{i_1}|^\kappa - |n_{i_2}|^\kappa.
\end{equation}
Using the  identities
\[\prod^{i,j} - \prod^{i,k} = \prod^{i,j,k} \Delta_{k,j},\qquad
\prod^{i,j,k} - \prod^{l,j,k} = \prod^{i,j,k,l} \Delta_{l,i},\]
and the fact that the resonance relations (see \eqref{def:Abh},\eqref{def:Aw})
imply  $\Delta_{1,2}
= \Delta_{4,3}$, $\Delta_{5,6} = \Delta_{8,7}$,
one can see that
\begin{equation}\label{eq:numbeam3}
d_{12}= \Delta_{2,1} \Delta_{6,5} \left( |n_3|^\kappa |n_4|^\kappa - |n_1|^\kappa |n_2|^\kappa
\right) \left( |n_7|^\kappa |n_8|^\kappa - |n_5|^\kappa |n_6|^\kappa \right)
\prod^{1,2,3,4,5,6,7,8}.
\end{equation}
%
Therefore $d_{12}$ vanishes if one of the following
conditions holds
\begin{equation}\label{NotAdmissibleBeam}
|n_1| = |n_2|,\qquad|n_5| = |n_6|,\qquad
 |n_1||n_2| = |n_3||n_4|,\qquad
|n_5||n_6| = |n_7||n_8|.
\end{equation}
\begin{remark}
We point out that the conditions \eqref{NotAdmissibleBeam} do not involve at the same time modes belonging to two different $4$-tuple resonances.
\end{remark}
Condition \eqref{eq:Wave} implies that the two first
conditions cannot be satisfied.
We check now that under the hypotheses of Proposition  \ref{prop:LambdaSetWaveHom}, one has
\begin{equation}\label{def:CondExtraForPk}
|n_1||n_2| \neq |n_3||n_4|
\end{equation}
(the condition $|n_5||n_6| \neq |n_7||n_8|$ can be checked analogously).

We start with the Beam equation, that is $\kappa=2$. Arguing by contradiction,
assume that $n_1,n_2,n_3,n_4$ satisfy  $|n_1||n_2| = |n_3||n_4|$, \eqref{eq:Wave} and the resonance condition
\begin{equation}\label{cond:resforPk}
 |n_1|^2-|n_2|^2=-|n_3|^2+|n_4|^2
\end{equation}
The resonance relation can be written as
\begin{align*}
(|n_1|-|n_2|)(|n_1|+|n_2|) &= (|n_4|-|n_3|)(|n_4|+|n_3|)
\end{align*}
Squaring each side, one has
\begin{align*}
(|n_1|^2 + |n_2|^2)^2 - 4 |n_1|^2 |n_2|^2= (|n_3|^2 + |n_4|^2)^2 - 4 |n_3|^2 |n_4|^2.
\end{align*}
Therefore, since we are assuming $|n_1||n_2| = |n_3||n_4|$,  we get $|n_1|^2 +
|n_2|^2= |n_3|^2 +
|n_4|^2$, which combined with the resonance relation \eqref{cond:resforPk} leads to $|n_2|^2=
|n_3|^2$,
which contradicts assumption \eqref{eq:Wave}.

For the Wave equation \eqref{Wave}, that is $\kappa=1$, one can proceed
analogously, arguing by contradiction. Assume that
$n_1,\ldots,n_4$ satisfy \eqref{eq:Wave},
the resonance condition
\[
 |n_1|-|n_2|=-|n_3|+|n_4|
\]
and $|n_1||n_2| = |n_3||n_4|$. Squaring the resonance condition and using this
last assumption, one has
\[|n_1|^2 + |n_2|^2 = |n_3|^2 + |n_4|^2.\]
Multiplying both sides by $|n_4|^2$ and using again $|n_1||n_2| = |n_3||n_4|$
one obtains
$(|n_1|^2 -|n_4|^2)(|n_4|^2 - |n_2|^2) = 0$,
which contradicts \eqref{eq:Wave}.
\end{proof}

\begin{proof}[Proof of Lemma \ref{lemma:HartreeMelnikovCond}]
Recall \eqref{coeffN}, \eqref{def:coeffmatrixA}.
For the Hartree equation \eqref{Hartree}, $d_{12}$ is of the form
\[
 d_{12}=\sum_{k\in I}\al_k V_k
\]
where $\al_k\neq 0$ and
\[
I:=\{ k\in\mathbb{Z}^2 : k=n_i-n_j \,\,\mbox{for some}\,\,n_i\in\mathcal{R}_1, \,\,n_j\in\mathcal{R}_2\}.
\]
We observe that the cardinality of $I$ is bounded by $4N(4N-1)/2$. Therefore, by condition
\eqref{nondegLambda}, $d_{12}$ is a polynomial in the $4N(4N-1)/2$ variables $\gamma_k$, $k\in I$. Such polynomial is not identically zero because if we set one of the $\gamma_k$'s equal to one and all the others at zero then $d_{12}\neq 0$.

\end{proof}

\subsection{End of the proof of Theorem \ref{TeoWaveBeam}}

 Lemmas \ref{lemma:WaveBeamMelnikovCond}, \ref{lemma:HartreeMelnikovCond} imply that condition \eqref{tildekappas} holds and, therefore,  Proposition \ref{horseshoeteo} can be applied to the resonant models associated to the Wave \eqref{Wave}, Beam \eqref{Beam} and Hartree \eqref{Hartree} equations. This proposition gives certain orbits of these resonant models. These orbits will be the first order (up to changes of coordinates) of orbits of equations \eqref{Wave}, \eqref{Beam} and \eqref{Hartree}.

Fix $0<\eps\ll 1$ and $0<h\ll 1$ and consider the periodic orbit $\mathtt{P}^{0 }_{h,\eps}$ given by Proposition \ref{prop:SplittingPO2Rect}, which has period $\mathtt{T}_h$. By Proposition \ref{horseshoeteo} there exist a set  $Y\subset\mathcal{S}_h$ which is an invariant hyperbolic set (a Smale horseshoe) for the Poincar\'e map associated to the Hamiltonian $\mathcal{H}$ in \eqref{calham0}. This set can be built arbitrarily close to homoclinic points of  $\mathtt{P}^{0 }_{h,\eps}$. Fix $\omega\in\Sigma$  such that $|\omega_k|\geq M_0\,\mathtt{T}_{h}$, where $M_0$ satisfies $M_0\gtrsim\log\eps$ and $\mathtt{T}_h$ is the period of the periodic orbit $\mathtt{P}^{0}_{h, \varepsilon}$.
Then, Proposition \ref{horseshoeteo} ensures that there exists an orbit $\gamma(t)$  of $\mathcal{H}$ with initial condition in $Y$,
\[
\gamma(t):=(\Psi_1(t), \Psi_2(t), K_1(t), K_2(t)), \quad t\in [0, T] \quad \mbox{for some}\,\,\,T>0,
\]
 which satisfies the following. There exists a sequence of times $\{t_k\}_{k\in\ZZ}$ satisfying \eqref{def:omegastar} such that $\gamma(t_k)\in \mathcal{S}_h$ where $\mathcal{S}_h$ is the section defined in \eqref{def:SectionSh}. Note that, by \eqref{def:omegastar}, the times $t_k$ satisfy
\[
 t_{k+1}=t_{k}+ \mathtt{T}_h (\omega^*_k+C^*+\theta_k)\qquad \text{ for some }\quad \theta_k\in (0,1) \quad \text{ and }C^*\in\N.
\]
By construction, there exists another sequence of times $\{\bar t_k\}_{k\in\ZZ}$ with $\bar t_k\in (t_k,t_{k+1})$ such that $\gamma(\bar t_k)$ satisfies
\[
 K_2(\bar t_k)=\frac{1}{2},\qquad  \left|\Psi_2(\bar t_k)-\frac{2\pi}{3}\right|\ll 1.
\]
The Smale horseshoe, can be built arbitrarily close to the invariant manifolds of $\mathtt{P}^{0}_{h, \varepsilon}$ and therefore, one can ensure that there exist intervals
\begin{itemize}
\item  $I_k\subset (\overline{t_k},t_{k+1})$  such that, for $t\in I_k$, $\ga(t)$ belongs to a $\eps$-neighborhood of  $\mathtt{P}^{0}_{h, \varepsilon}$;
\item  $J_k\subset (t_k, \overline{t_k})$  such that for $t\in J_k$ the orbit $\gamma(t)$ belongs to a $\OO(\eps)$-neighborhood of $K_2=1$, since the homoclinic orbit obtained in Proposition \ref{prop:SplittingPO2Rect} have points $\OO(\eps)$-close to $K_2=1$.
\end{itemize}
This behavior implies estimates \eqref{eq:TheoRandomTimeOsc1} and \eqref{eq:TheoRandomTimeOsc2} in Theorem \ref{TeoWaveBeam}, once we undo the symplectic reductions, the changes of coordinates and we add the error terms as it is  explained below.
By Proposition \ref{prop:SplittingPO2Rect} the parameterization of the periodic orbit  $\mathtt{P}^{0}_{h, \varepsilon}$ is $\varepsilon$-close to \eqref{def:PO:ext}, hence we have that
\[
 K_1(t)=Q(t)+\tilde{R}_2(t)
\]
where $Q(t)$ is the time parameterization of  $\mathtt{P}^{0 }_{h,\eps}$ and thus is   $\mathtt{T}_h$-periodic,  and $\sup_{t\in [0, T]} |\tilde{R}_2(t)|\le \varepsilon$.

%
%
%
%

By the symplectic reduction performed in Section \ref{sectionsymplecticreduction} there exists $r(t)$ solution of $H_{\Res}$ in \eqref{Hres} with Fourier support $\Lambda$ such that
\[
|r_{n_1}(t)|^2=| K_{1}(t) |^2, \qquad |r_{n_5}(t)|^2=| K_{2}(t) |^2,\qquad \text{for}\qquad t\in [0,T].\]
This can be seen using Remark \ref{rem:Kcartcoord}, which gives also the behavior of the other actions.

Since the solutions of $H_{\Res}$ are invariant under the scaling \eqref{def:scaling}, we can consider $r^{\delta}(t):=\delta r(\delta^2 t)$. Then, $r^{\delta}(t)$ is also a solution of $H_{\Res}$ for  $t\in [0, \delta^{-2} T]$.

Now it only remains to obtain an orbit for the equations \ref{Wave}, \ref{Beam} and \ref{Hartree} which is close (up to certain changes of coordinates) to $r^{\delta}(t)$. First step is to apply Proposition \ref{approximationargument}. It ensures that there exists $0<\delta_2\ll 1$ such that for all $\delta\in (0, \delta_2)$, there exists a solution $w(t)$ of $H\circ \Gamma\circ \Psi=H_{\Res}+\mathcal{R}'$   such that $w(t)=r^{\delta}(t)+\widetilde{R}(t)$ with $\widetilde{R}(0)=0$, $\| \widetilde{R}(t) \|_{\rho}\lesssim \delta^2$ for $t\in [0, \delta^{-2} T]$. We note that,  by Item $(ii)$ of Proposition \ref{WBNF}, the Birkhoff map $\Gamma$ is $\delta^3$-close to the identity. Finally the transformations \eqref{rotatingcoord1} and \eqref{fase} preserve the modulus of the Fourier coefficients. The last change of coordinates that one has to apply (for the Wave \eqref{Wave} and Beam \eqref{Beam} equations) is  passing from complex coordinates \eqref{def:complexcoordinates} to the original ones. We remark that by \eqref{eq:Wave} if $n_i\in\Lambda$ then $-n_i\notin \Lambda$. Thus
  \[
  u_{n_i}=\frac{1}{\sqrt{2 |j|}} \Psi_{n_i} \qquad n_i\in\Lambda.
  \]

\section{Transfer of beating effects: Proof of Theorem \ref{thm:traveling_periodic_beating}}\label{sec:transfer}
We devote this section to prove Theorem \ref{thm:traveling_periodic_beating}. First, in Section \ref{sec:MultiSplitting} we prove the transversality of the stable and unstable invariant manifolds of different periodic orbits of the Hamiltonian \eqref{redhamN}. As a consequence of this transversality, we construct orbits which shadow these invariant manifolds for infinite time. Then, in Section \ref{sec:ApplicationMultiple} we prove that the resonant models associated to the Wave, Beam and Hartree equations that we consider fit into the framework of Section \ref{sec:MultiSplitting} and we complete the proof of Theorem \ref{thm:traveling_periodic_beating}.

\subsection{Heteroclinic connections between periodic orbits and their shadowing}\label{sec:MultiSplitting}

Reasoning as in Proposition \ref{prop:SplittingPO2Rect}, the Hamiltonian $\mathcal{H}$ in \eqref{redhamN}  possesses hyperbolic periodic orbits
$\mathtt{P}_{\varepsilon,h,k}^\pm$ at the energy level $h$ whose time parameterization is of the form
\begin{equation*}
\gamma_{\varepsilon,h,k}^{\pm,p}(\tau_k) = (\Psi_{\pm,\varepsilon,1}^*, \dots,\Psi_{\pm,\varepsilon,k-1}^*, \Psi_k^{(h)}(\tau_k),\Psi_{\pm,\varepsilon,k+1}^*, \dots,
\Psi_{\pm,\varepsilon,N}^*, 0, \dots,0, K_k^{(h)}(\tau_k),0, \dots,0)
\end{equation*}
where
\[
 \Psi_{\pm,\varepsilon,1}^*=\pm\Psi_\ast+\OO(\eps)
\]
(see \eqref{def:psi}) and $( \Psi_k^{(h)}, K_k^{(h)})$ is $\eps$-close to the periodic orbit $\mathtt{P}_h$ (see \eqref{def:Gammah}).

When $\varepsilon = 0$, the  invariant manifolds $W^u (\mathtt{P}_{0,h,k}^+)$ and  $W^s(\mathtt{P}_{0,h,k}^-)$ coincide.
%

\begin{prop}\label{prop:SplittingMulti}
Take any $i,j=1,\ldots, N$, $i\neq j$. Assume that the condition  \eqref{cond:HessD} is satisfied (see \eqref{def:MatrixD},  \eqref{coeffN}). Then, there exists $\varepsilon_0>0$ such that for
$\varepsilon\in(0, \varepsilon_0)$ and $h_0>0$
such that for any $h \in(0, h_0)$,
the manifolds
$W^u(\mathtt{P}_{\varepsilon, h, i}^-)$ and
$W^s(\mathtt{P}_{\varepsilon, h, j}^+)$ intersect transversally
within the energy level.
\end{prop}

The transversality of the invariant manifolds allows to construct orbits which shadow them. Note that in the coordinates introduced in \eqref{def:cartesian}, the periodic orbits $\mathtt{P}_{0,h,k}^+$ and $\mathtt{P}_{\eps,h,k}^-$ blow down to the same periodic orbit, which we denote by $\mathtt{P}_{0,h,k}$. In the coordinates  \eqref{def:cartesian}, Proposition \ref{prop:SplittingMulti} can be restated as that the  manifolds $W^u(\mathtt{P}_{\varepsilon, h, i})$ and $W^s(\mathtt{P}_{\varepsilon, h, j})$ intersect  transversally along an orbit
within the energy level.

%
\begin{defi}
\label{def:transition_chain} We will say that a family of hyperbolic periodic
orbits $\{\mathtt{P}_{\ell}\}_{\ell \in \N}$ of a system of differential equations,  is a \emph{transition
chain} if  $W^u(\mathtt{P}_{\ell}) \pitchfork
W^s(\mathtt{P}_{{\ell+1}})$, for all $\ell \in \N$.
\end{defi}
Note that Proposition \ref{prop:SplittingMulti} gives full transversality between the invariant manifolds on the energy level. Thus, recalling that  $\mathcal{H}(\mathtt{P}_{i_\ell})=h$, from now on, we restrict the flow to this energy level, which is a regular manifold.
%

\begin{cor}
\label{cor:transition_chain}
Let $(i_\ell)_{\ell \in \N}$, with $i_\ell \in \{1,\dots,N\}$, be any sequence.
Then, if $\varepsilon>0$ is small enough, there exists $h_0$ such
 that for any $0<h<h_0$, $\{\mathtt{P}_{\eps,h,i_\ell}\}_{\ell \in \N}$ is a transition chain
 of Hamiltonian $\mathcal{H}$ on the manifold $\mathcal{H}=h$.
\end{cor}

\begin{prop}
\label{prop:diffusion_orbits}
Let $(i_\ell)_{\ell \in \N}$, with $i_\ell \in \{1,\dots,N\}$, be any sequence. Assume that
$\varepsilon>0$ is small enough such that $h_0$ in Corollary~\ref{cor:transition_chain} exists.
Let $\{\mathtt{P}_{\eps,h,i_\ell}\}_{\ell \in \N}$ be a transition chain
 of Hamiltonian $\mathcal{H}$. Let $(\nu_\ell)_{\ell \in \N}$, with $\nu_\ell >0$, be an arbitrary sequence.
Let $N_{\ell} := \{z\mid d(z,\mathtt{P}_{\eps,h,i_\ell})<\nu_{\ell}\}$. Then, there exists a trajectory $\gamma(t)$ of  Hamiltonian $\mathcal{H}$ in \eqref{redhamN}
and an increasing sequence  $(t_\ell)_{\ell \in \N}$ of times such that $\gamma(t_\ell) \in N_{\ell}$, for all $\ell\in\N$.
\end{prop}

\begin{proof}
Since $\{\mathtt{P}_{\eps,h,i_\ell}\}_{\ell \in \N}$ is a transition chain
 then the Inclination Lemma in~\cite{FontichM01} (Theorem $4.5$) ensures that $W^s_\varepsilon( \mathtt{P}_{\eps,h,i_{\ell}})\subseteq \overline{ \cup_{t\leq 0} \Phi^t_{\mathcal{H}}(W_{\varepsilon}^{s} (\mathtt{P}_{i_{\ell}+1}))}$ for all $\ell\in\mathbb{N}$ \footnote{More precisely we apply Theorem 4.5 in~\cite{FontichM01}to the flow map $f:=\Phi^{\tau}_{\mathcal{H}}$, where $\tau>0$ is chosen to be not a multiple of any frequency of the periodic orbits $\mathtt{P}_{{\eps,h,i_\ell}}$. Note that the Inclination Lemma stated in~\cite{FontichM01} is stated for the unstable manifold; in order to deduce the statement for the stable manifold it suffices to replace $f$ by $f^{-1}$.} .

  Let $x\in W^s_\varepsilon(\mathtt{P}_{i_0})$. We can find a closed ball $B_0$ centered at $x$ such that
  \begin{equation}\label{eq:inclusion}
  \Phi_{\mathcal{H}}^{t_0}(B_0)\subset N_0
  \end{equation}
  for some $t_0>0$. By the inclination Lemma we have that $W_{\varepsilon}^s(\mathtt{P}_{i_1})\cap B_0\neq \emptyset$. Hence we can find a closed ball $B_1$ centered at a point in $W_{\varepsilon}^s(\mathtt{P}_{i_1})\cap B_0$ such that, besides satisfying \eqref{eq:inclusion},
 \begin{equation*}
  \Phi_{\mathcal{H}}^{t_1}(B_1)\subset N_1
 \end{equation*}
 for some $t_1>t_0$. Proceeding by induction we can construct a sequence of closed nested balls $B_{{i+1}}\subset B_{i} \subset \dots\subset B_0$ and times $t_{i+1}>t_i>\ldots >t_0$ such that
 \[
 \Phi_{\mathcal{H}}^{t_j} (B_i)\subset N_{j}, \qquad i\ge j.
 \]
 Since the balls are compact, the intersection $\cap_{n\geq 0} B_n$ is non-empty, and we can consider $\gamma(t)$ as an orbit with initial datum in that set.

\end{proof}

\subsubsection{Proof of Proposition \ref{prop:SplittingMulti}}


We proceed as in Section \ref{sec:HomoMelnikov} by considering an auxiliary parameter $\de$ and the Hamiltonian $\mathcal{H} = \sum_{j=1}^N \mathbf{H}_0^{(j)}
+ \varepsilon \mathbf{H}_1$ defined in \eqref{def:hamN}. The Hamiltonian \eqref{def:hamN} has two saddle points,
\[
\mathfrak{e}^{(0)}_{\pm,\varepsilon}=  (\Psi_{\pm, \varepsilon,1}^*,\dots, \Psi_{\pm, \varepsilon,N}^*, 0,\dots,0),
\]
which, for $\varepsilon = 0$ are
$\mathfrak{e}^{(0)}_{\pm}$ (see~\eqref{saddlesN}). For $\varepsilon = 0$ and any $\delta>0$ small, they
 are connected by the homoclinic manifolds
\begin{equation}
\label{def:homoclinic_manifold_ndim}
\gamma_0(\vec{\tau} ) = (\Psi_1^{(0)}(\tau_1), \ldots, \Psi_N^{(0)}(\tau_N), K_1^{(0)}(\tau_1), \ldots, K_N^{(0)}(\tau_N)), \quad \vec{\tau}:=(\tau_1, \dots, \tau_N),
\end{equation}
where $\Psi_j^{(0)}$, $K_j^{(0)}$, $j=1,\dots, N$, have been introduced in~\eqref{gammazero}.
This parametrization of the homoclinic manifold satisfies ${\Phi_{\mathcal{H}}^t}_{\mid \varepsilon = 0} \gamma_0(\vec{\tau} )
= \gamma_0(\tau_1+t, \dots,\tau_n+t)$.
Fix $1\le k \le N$. The set
\[
\pi_k = \{(\Psi_1,\dots, \Psi_N, K_1,\dots,K_N): \;K_\ell = 0, \; \ell \neq k\}
\]
is invariant by the flow of $\mathcal{H}$ for any $\varepsilon$ and $\delta$ (this is properly seen in the coordinates \eqref{def:cartesian}, since then $\pi_k$ corresponds to $x_\ell=y_\ell=0$, $\ell\neq k$, see \eqref{redhamNCartesian}).

The dynamics on the  $\pi_k$ plane is integrable and is given by the $1$-d.o.f. Hamiltonian
$\mathbf{H}_0^{(k)} + \varepsilon {\mathbf{H}_1}_{\mid \pi_k}$. This Hamiltonian has two saddles $(\Psi_{\pm,\varepsilon,\ell}^*,0)$
 $\varepsilon$-close to $(\pm \Psi_*,0)=(\pm 2\pi/3,0)$, at the zero energy level. For $h>0$ small, the set $\{\mathbf{H}_0^{(k)} + \varepsilon {\mathbf{H}_1}_{\mid \pi_k}=h\}$ is a periodic orbit, whose period tends to infinity when $h$ goes to $0$. Let $(\Psi_k^{(h)}(\tau_k), K_k^{(h)}(\tau_k))$ be a time parametrization of this periodic orbit satisfying
\begin{equation}
\label{eq:choice_param_periodic_orbit}
\lim_{h\to 0} (\Psi_k^{(h)}(0), K_k^{(h)}(0)) =
(\Psi_k^{(0)}(0), K_k^{(0)}(0)),
\end{equation}
where $(\Psi_k^{(0)}, K_k^{(0)})$ are components of the homoclinic manifold introduced in \eqref{def:homoclinic_manifold_ndim}.

Then,
the Hamiltonian $\mathcal{H}$  possesses two hyperbolic periodic orbits
$\mathtt{P}_{\varepsilon,h,k}^\pm$ at the energy level $h$, whose time parametrization is given by
\begin{equation}\label{def:timeparapo}
\gamma_{\varepsilon,h,k}^{\pm,p}(\tau_k) = (\Psi_{\pm,\varepsilon,1}^*, \dots,\Psi_{\pm,\varepsilon,k-1}^*, \Psi_k^{(h)}(\tau_k),\Psi_{\pm,\varepsilon,k+1}^*, \dots,
\Psi_{\pm,\varepsilon,N}^*, 0, \dots,0, K_k^{(h)}(\tau_k),0, \dots,0).
\end{equation}
When $\varepsilon = 0$, the  invariant manifolds $W^u (\mathtt{P}_{0,h,k}^+)$ and  $W^s(\mathtt{P}_{0,h,k}^-)$ coincide. This homoclinic manifold can be parameterized as
\begin{equation}\label{def:omoclinicpo}
\gamma_{h,k}(\vec{\tau} ) = (\Psi_1^{(0)}(\tau_1), \ldots, \Psi_k^{(h)}(\tau_k), \dots,  \Psi_N^{(0)}(\tau_N), K_1^{(0)}(\tau_1), \ldots, K_k^{(h)}(\tau_k), \dots, K_N^{(0)}(\tau_N)),
\end{equation}
where $(\Psi_k^{(0)}, K_k^{(0)})$ are components of the homoclinic manifold introduced in \eqref{def:homoclinic_manifold_ndim}.

Now, fix $i,j \in \{1,\dots, N\}$.
For small $\varepsilon>0$, the periodic orbits  $\mathtt{P}_{\varepsilon,h,i}^-$, $\mathtt{P}_{\varepsilon,h,j}^+$ and their invariant manifolds,
$W^s(\mathtt{P}_{\varepsilon,h,i}^+)$ and $W^u(\mathtt{P}_{\varepsilon,h,j}^-)$ persist slightly deformed. We show now that the perturbation allows them to intersect.

In order to analyze the possible intersection,
we introduce a $N$-dimensional section in the following way.
We define, taking into account~\eqref{def:hamN},
\begin{equation}\label{def:hamtrick1}
\widetilde{\mathbf{H}}_0^{(k)} = \mathbf{H}_0^{(k)} + \varepsilon \widehat{\mathbf{H}}_0^{(k)}, \qquad k = 1,\dots,N,
\end{equation}
where
\begin{equation*}
\widehat{\mathbf{H}}_0^{(k)}(\psi_k,K_k) =
a_{k} K_k + (b_{k}-1) K^2_k +c_k\,K_k (1-K_k)\,\cos(\psi_k)
\end{equation*}
only depends on $(\psi_k,K_k)$. We have that $\widetilde{\mathbf{H}}_0^{(k)}$ is integrable and $\mathcal{H}$ can be also written as
\begin{equation}
\label{def:newdefinitionofmathcalH}
\mathcal{H} = \sum_{k=1}^N \mathbf{H}_0^{(k)} + \varepsilon \mathbf{H}_1 = \sum_{k=1}^N \widetilde{\mathbf{H}}_0^{(k)} + \varepsilon \widetilde{\mathbf{H}}_1
\end{equation}
where
\begin{equation}\label{def:hamtrick3}
\widetilde{\mathbf{H}}_1(K_1,\dots,K_n) = \sum_{k, \ell=1, k<\ell}^N d_{k \ell} K_k\,K_\ell.
\end{equation}
We consider the $N$-dimensional section
\begin{equation}
\label{def:transsecNrectangles}
\Sigma(\vec{\tau})=\left\{\gamma_{0}(\vec{\tau})+\sum_{k=1}^ Nr_k\,{\nabla \widetilde{\mathbf{H}}^{(k)}_0}_{\mid \gamma_0(\vec{\tau})},
r=(r_1,\dots,r_{N})\in (-m, m)^{N}\right\}
\end{equation}
where $\gamma_0$ is the homoclinic manifold introduced in \eqref{def:homoclinic_manifold_ndim}. Observe that $\gamma_{0}(\vec{\tau})$, which is $N$-dimensional, intersects transversally $\Sigma(\vec{\tau})$ at $r=0$. Then, for $h$ small, $\gamma_{h,i}$ and $\gamma_{h,j}$ (see \eqref{def:omoclinicpo}) intersect transversally
$\Sigma(0)$ at points $r_{i}$ and $r_{j}$, close to $\gamma_{0}(\vec{\tau})$. Hence, for $\eps$ small enough, the invariant manifolds  $W^u(\mathtt{P}_{\varepsilon,h,i}^-)$ and $W^s(\mathtt{P}_{\varepsilon,h,j}^+)$  intersect transversally $\Sigma(0)$ at points $r_{\varepsilon,i}$ and $r_{\varepsilon,j}$ close to $r_{i}$ and $r_{j}$, respectively.

Let $\gamma_{\varepsilon,h,i}^u$ and $\gamma_{\varepsilon,h,j}^s$ be parametrizations of the perturbed invariant manifolds $W^u(\mathtt{P}_{\varepsilon,h,i}^-)$ and $W^s(\mathtt{P}_{\varepsilon,h,j}^+)$ such that
$\gamma_{\varepsilon,h,i}^u(0) = r_{\varepsilon,i}$, $\gamma_{\varepsilon,h,j}^s(0) = r_{\varepsilon,j}$
and
%
$\Phi^t_{\mathcal{H}} \gamma(\tau_1,\dots, \tau_n) = \gamma (\tau_1+t,\dots, \tau_N+t)$, for $\gamma = \gamma_{\varepsilon,h,i}^{u},\gamma_{\varepsilon,h,j}^{s}$, where $\Phi^t_{\mathcal{H}}$ is the flow of Hamiltonian $\mathcal{H}$.
Up to a shift in the initial conditions in the periodic orbits, the parameterization of the periodic orbits and the homoclinic manifold satisfy the following property: for any $\tau$ there exists constants $\la, K, M>0$ such that
\begin{equation*}
 \begin{aligned}
\|\gamma_{\varepsilon,h,k}^{u}(\tau_1+t,\ldots,\tau_N+t) - \gamma_{\varepsilon,h,i}^{-,p}(\tau_i+t)\| \le &K e^{\lambda t}\qquad &\text{ for }\quad t\leq M\\
\|\gamma_{\varepsilon,h,k}^{s}(\tau_1+t,\ldots,\tau_N+t) - \gamma_{\varepsilon,h,j}^{+,p}(\tau_j+t)\| \le &K e^{- \lambda t}\qquad &\text{ for }\quad t\geq M.
  \end{aligned}
\end{equation*}
Let us remark that $\mathcal{H}_{\mid W^u (\mathtt{P}_{\varepsilon,h,i}^-)}=\mathcal{H}_{\mid W^s (\mathtt{P}_{\varepsilon,h,j}^+)}=h$. Therefore, to analyze their intersections it is enough to measure their distance along $(N-1)$-directions of those defining the section $\Sigma$ in \eqref{def:transsecNrectangles}. That is, the manifolds $W^u (\mathtt{P}_{\varepsilon,h,i}^-)$ and $W^s (\mathtt{P}_{\varepsilon,h,j}^+)$
 intersect transversally along an orbit at the non-degenerate zeros of the vector function (see \eqref{def:hamtrick1})
\begin{equation}
\label{def:distance_between_manifolds_of_periodic_orbits}
d_{\varepsilon,h}(\vec{\tau}) =
\begin{pmatrix}
\widetilde{\mathbf{H}}_0^{(1)}(\gamma_{\varepsilon,h,i}^u(\vec{\tau}))-
\widetilde{\mathbf{H}}_0^{(1)}(\gamma_{\varepsilon,h,j}^s(\vec{\tau})) \\
\vdots \\
\widetilde{\mathbf{H}}_0^{(N-1)}(\gamma_{\varepsilon,h,i}^u(\vec{\tau}))-
\widetilde{\mathbf{H}}_0^{(N-1)}(\gamma_{\varepsilon,h,j}^s(\vec{\tau}))
\end{pmatrix}.
\end{equation}

\begin{lem}
The function $d_{\varepsilon,h}$ in~\eqref{def:distance_between_manifolds_of_periodic_orbits} can be written as
\begin{equation}
\label{eq:difMelnikov}
d_{\varepsilon,h}(\vec{\tau}) = d_{0,h}+ \varepsilon \mathcal{M}_h(\vec{\tau}) + \OO\left(\varepsilon^2\right),
\end{equation}
where  the vector $d_{0,h}=(d_{0,h}^1, \ldots d_{0,h}^{N-1})^\top$ is of the form
\[
 d_{0,h}^i=h,\quad d_{0,h}^j=-h\quad \text{ and }\quad d_{0,h}^k=0\quad\text{ for }\quad k\neq i,j
\]
%
%
and
$\mathcal{M}_h(\vec{\tau}) = ( \mathcal{M}_{h}^1(\vec{\tau}) ,\dots,
\mathcal{M}_{h}^{N-1}(\vec{\tau}))^\top$, with
\[
\mathcal{M}_{h}^k(\vec{\tau}) :=\int_{-\infty}^{0} \{ \mathbf{H}_0^{(k)}, \widetilde{\mathbf{H}}_1\} \, \circ \, \Phi^t_{\mathbf{H}_0}(\gamma_{h,i}(\vec{\tau}))\,dt+\int_0^{\infty} \{ \mathbf{H}_0^{(k)}, \widetilde{\mathbf{H}}_1\}\, \circ \,\Phi^t_{\mathbf{H}_0}(\gamma_{h,j}(\vec{\tau}))\,dt.
\]
\end{lem}
\begin{proof}
We will compute the formula for $\widetilde{\mathbf{H}}_0^{(k)}(\gamma_{\varepsilon,h,i}^u(\vec{\tau}))$, $k=1,\dots, N-1$,
being the derivation for the one of $\widetilde{\mathbf{H}}_0^{(k)}(\gamma_{\varepsilon,h,j}^s(\vec{\tau}))$ analogous.

We first observe that, since $\mathcal{H}_{\mid \mathtt{P}_{\varepsilon,h,i}^-}= h$ and
$\widetilde{\mathbf{H}}_{1\,\mid \mathtt{P}_{\varepsilon,h,i}^-}= 0$,
 $\widetilde{\mathbf{H}}_0^{(k)}(\gamma_{\varepsilon,h,i}^{-,p}(\vec{\tau})) = \delta_{ik}h$, being $\delta_{ik}$ the Kronecker's delta. Then, taking into account~\eqref{def:newdefinitionofmathcalH}, it is immediate that
\[
\begin{aligned}
\widetilde{\mathbf{H}}_0^{(k)}(\gamma_{\varepsilon,h,i}^u(\vec{\tau})) &
= \widetilde{\mathbf{H}}_0^{(k)}(\gamma_{\varepsilon,h,i}^u(\vec{\tau})) - \widetilde{\mathbf{H}}_0^{(k)}(\gamma_{\varepsilon,h,i}^{-,p}(\vec{\tau})) + \delta_{ik} h \\
& = \int_{-\infty}^0 \frac{d}{dt} \widetilde{\mathbf{H}}_0^{(k)} \circ \Phi_{\mathcal{H}}^t \gamma_{\varepsilon,h,i}^u(\vec{\tau})\, dt + \delta_{ik} h \\
& = \varepsilon \int_{-\infty}^0 \{ \widetilde{\mathbf{H}}_0^{(k)}, \widetilde{\mathbf{H}}_1\} \circ \Phi_{\mathcal{H}}^t( \gamma_{\varepsilon,h,i}^u(\vec{\tau}))\, dt + \delta_{ik} h \\
& = \varepsilon \int_{-\infty}^0 \{ \widetilde{\mathbf{H}}_0^{(k)}, \widetilde{\mathbf{H}}_1\} \circ \Phi_{\mathbf{H}_0}^t ( \gamma_{h,i}(\vec{\tau}))\, dt + \delta_{ik} h + \mathcal{O}(\varepsilon^2) \\
& = \varepsilon \int_{-\infty}^0 \{ \mathbf{H}_0^{(k)}, \widetilde{\mathbf{H}}_1\} \circ \Phi_{\mathbf{H}_0}^t (\gamma_{h,i}(\vec{\tau}))\, dt + \delta_{ik} h +  \mathcal{O}(\varepsilon^2),
\end{aligned}
\]
where $\gamma_{h, i}$ is defined in \eqref{def:omoclinicpo}.
\end{proof}
We observe that, since the components of $d_{0,h}$ are either $0$ or $\pm h$,
if we consider $h\ll \varepsilon$, the main order of the difference in~\eqref{eq:difMelnikov} is given by $\mathcal{M}_h(\vec{\tau})$. Thus we shall prove that this function has a non-degenerate zero, so that we can conclude by the Implicit Function Theorem that the manifolds $W^s_\varepsilon(\mathtt{P}^-_{\varepsilon, h, i})$ and $W^u_\varepsilon(\mathtt{P}^+_{\varepsilon, h, j})$ intersect transversally.

To do so, we introduce
\begin{equation}
\label{def:homoclinic_Melnikov}
\mathcal{M}_{0}(\vec{\tau}) := ( \mathcal{M}_0^1(\vec{\tau}),\dots, \mathcal{M}_0^{N-1}(\vec{\tau}))^\top,
\end{equation}
where
\begin{equation*}
 \mathcal{M}_{0}^k(\vec{\tau}) := \int_{-\infty}^{\infty} \{ \mathbf{H}_0^{(k)}, \widetilde{\mathbf{H}}_1\}\,\circ\,\Phi^t_{\mathbf{H}_0}(\gamma_{0}(\vec{\tau}))\,dt,
\end{equation*}
where $\gamma_{0}$ was introduced in~\eqref{def:homoclinic_manifold_ndim} (see also~\eqref{gammazero}), which is the Melnikov function associated
to the homoclinic between $\mathfrak{e}^{(0)}_{\pm}$. We observe that the derivative of the Melnikov potential \eqref{def:malnikovpotcaseN} with respect to the variable $\tau_k-\tau_N$ coincides with the Melnikov integral $\mathcal{M}_0^k(\vec{\tau})$ in  \eqref{def:homoclinic_Melnikov}: Indeed, recall that the Melnikov Potential integral associated to $(\mathbf{H}_1-\widetilde{\mathbf{H}}_1)$ is constant, and equivalently
\[
\mathcal{M}_{0}^k(\vec{\tau}) = \int_{-\infty}^{\infty} \{ \mathbf{H}_0^{(k)}, \widetilde{\mathbf{H}}_1\}\,\circ\,\Phi^t_{\mathbf{H}_0}(\gamma_{0}(\vec{\tau}))\,dt= \int_{-\infty}^{\infty} \{ \mathbf{H}_0^{(k)}, \mathbf{H}_1\}\,\circ\,\Phi^t_{\mathbf{H}_0}(\gamma_{0}(\vec{\tau}))\,dt.
\]
Then, by Proposition \ref{prop:homointersaddlesN}, if condition \eqref{cond:HessD} is satisfied and $\varepsilon>0$ is small enough, $\mathcal{M}_{0}(\vec{\tau})$ has a non-degenerate zero.

\begin{lem}
\label{lem:Melnikov_between_periodic_orbits} Let $\varepsilon>0$ and assume that the condition \eqref{cond:HessD} in Proposition \ref{prop:homointersaddlesN} is satisfied. Then, there exists $h_0$ such that  for any $0<h<h_0$,
$\mathcal{M}_{h}(\vec{\tau})$ has a non-degenerate zero.
\end{lem}

This lemma implies Proposition \ref{prop:SplittingMulti};
indeed, one can proceed as in Section \ref{sec:HomoMelnikov} by taking $\de=\eps$ and applying Implicit Function Theorem. We devote the rest of the Section to prove Lemma \ref{lem:Melnikov_between_periodic_orbits}.

\begin{proof}[Proof of Lemma \ref{lem:Melnikov_between_periodic_orbits}]
We have that, for any $1\le k \le N-1$,
\begin{align*}
&\int_{-\infty}^0 \{ \mathbf{H}_0^{(k)}, \widetilde{\mathbf{H}}_1\}\,\circ\,\Phi^t_{\mathbf{H}_0}(\gamma_{h,i}(\vec{\tau}))\,dt
+\int_0^{\infty} \{ \mathbf{H}_0^{(k)}, \widetilde{\mathbf{H}}_1\}\,\circ\,\Phi^t_{\mathbf{H}_0}(\gamma_{h,j}(\vec{\tau}))\,dt-\mathcal{M}_0^{k}(\vec{\tau})=\\
&=\int_{-\infty}^0 \Big(\{ \mathbf{H}_0^{(k)}, \widetilde{\mathbf{H}}_1\}\,\circ\,\Phi^t_{\mathbf{H}_0}(\gamma_{h,i}(\vec{\tau}))-\{ {\mathbf{H}}_0^{(k)}, \widetilde{\mathbf{H}}_1\}\,\circ\,\Phi^t_{\mathbf{H}_0}(\gamma_{0}(\vec{\tau}))\,\Big)\,dt\\
&+\int_0^{\infty}\Big(\{ \mathbf{H}_0^{(k)}, \widetilde{\mathbf{H}}_1\}\,\circ\,\Phi^t_{\mathbf{H}_0}(\gamma_{h,j}(\vec{\tau}))-\{ \mathbf{H}_0^{(k)}, \widetilde{\mathbf{H}}_1\}\,\circ\,\Phi^t_{\mathbf{H}_0}(\gamma_{0}(\vec{\tau}))\,\Big)\,dt \\
& = I_{h,i}(\vec{\tau})+I_{h,j}(\vec{\tau}).
\end{align*}

We prove that, for any compact set $\mathbf{K}\subset \R^{N}$, $\|I_{h,k}\|_{C^1(\mathbf{K})}$ tends to $0$ as $h\to 0$, for $k=i,j$.
We give the argument for $I_{h,i}$, being the one for $I_{h,j}$ analogous. The claim follows immediately from this convergence.

Let $\mathbf{K}\subset \R^{N}$ be a compact set. If $h$ is small enough, the parametrization $\gamma_{h,i}$ is well defined;
since the period of the periodic orbit tends to infinity when $h$ goes to $0$, $\gamma_{h,i}$ intersects $\Sigma(\vec{\tau})$ at a point close to $r=0$, for all $\vec{\tau} \in \mathbf{K}$.

For a given $T>0$, we split the integral $I_{h,i}$ as
\begin{equation}
\label{eq:split_of_Ihi}
\begin{aligned}
I_{h,i} (\vec{\tau})  
 = & \int_{-T}^0 \Big(\{ \mathbf{H}_0^{(k)}, \widetilde{\mathbf{H}}_1\}\,\circ\,\Phi^t_{\mathbf{H}_0}(\gamma_{h,i}(\vec{\tau}))
 -\{ \mathbf{H}_0^{(k)}, \widetilde{\mathbf{H}}_1\}\,\circ\,\Phi^t_{\mathbf{H}_0}(\gamma_{0}(\vec{\tau}))\,\Big)\,dt \\
& + \int_{-\infty}^{-T} \Big(\{ \mathbf{H}_0^{(k)}, \widetilde{\mathbf{H}}_1\}\,\circ\,\Phi^t_{\mathbf{H}_0}(\gamma_{h,i}(\vec{\tau}))-\{ \mathbf{H}_0^{(k)}, \widetilde{\mathbf{H}}_1\}\,\circ\,\Phi^t_{\mathbf{H}_0}(\gamma_{0}(\vec{\tau}))\,\Big)\,dt.
\end{aligned}
\end{equation}
From~\eqref{def:hamN}, \eqref{def:hamtrick3}
\begin{equation}
\label{eq:Poisson_bracket_H0H1}
\{ \mathbf{H}_0^{(k)}, \widetilde{\mathbf{H}}_1\} = - 2 K_k(1-K_k) \sin \psi_k \sum_{\ell \neq k} d_{\ell,k} K_\ell.
\end{equation}
In particular,
\[
\{ \mathbf{H}_0^{(k)}, \widetilde{\mathbf{H}}_1\}_{\mid \mathfrak{e}^{(0)}_{-}}
= \{ \mathbf{H}_0^{(k)}, \widetilde{\mathbf{H}}_1\}_{\mid\mathtt{P}^-_{0, h, i}} = 0.
\]
The hyperbolic character of the periodic orbits implies the existence of constants $C, \lambda >0$ such that, for $\ell \neq i$ and for all $\vec{\tau}\in\mathbf{K}$,
\[
|\pi_{K_\ell} \gamma_{h,i}(\vec{\tau}) | \le C e^{-\lambda(\tau_\ell + t)}, \qquad t \ge 0.
\]
Also, with the same $C, \lambda >0$, for any $1 \le \ell \le N$,
\[
|\pi_{K_\ell} \gamma_{0}(\vec{\tau}) | \le C e^{-\lambda(\tau_\ell + t)}, \qquad t \ge 0.
\]
Hence, by~\eqref{eq:Poisson_bracket_H0H1}, for any $\nu >0$,
since $\mathbf{K}$ is compact, there exists $T>0$ such that, for any
$\vec{\tau} \in \mathbf{K}$,
\[
\left| \int_{-\infty}^{-T} \{ \mathbf{H}_0^{(k)},
\widetilde{\mathbf{H}}_1\}\,\circ\,\Phi^t_{\mathbf{H}_0}(\gamma_{h,i}(\vec{\tau}))\,dt
\right|, \, \left| \int_{-\infty}^{-T} \{ \mathbf{H}_0^{(k)},
\widetilde{\mathbf{H}}_1\}\,\circ\,\Phi^t_{\mathbf{H}_0}(\gamma_{0}(\vec{\tau}))\,dt
\right| \le \nu.
\]
To bound the other part of $I_{h,i}$ in~\eqref{eq:split_of_Ihi}, we observe that,
from~\eqref{man},
\[
\gamma_{h,i}(\vec{\tau}) - \gamma_{0}(\vec{\tau})
= (0, \ldots, \Psi_i^{(h)}(\tau_i)-\Psi_i^{(0)}(\tau_i), \dots,  0, 0, \ldots, K_i^{(h)}(\tau_i)-K_i^{(0)}(\tau_i), \dots, 0).
\]
We remark that, as $h$ goes to $0$, the period of the periodic orbit
$\mathtt{P}_{0,h,i}^-$  goes to $\infty$. Then, the choice of the parametrization of the periodic orbit~\eqref{eq:choice_param_periodic_orbit} implies that,
taking $h$ small enough, $\lim_{h\to 0}(\Psi_i^{(h)}(\tau_i+t), K_i^{(h)}(\tau_i+t)) =
(\Psi_i^{(0)}(\tau_i+t), K_i^{(0)}(\tau_i+t))$, for all $(t, \vec{\tau}) \in [-T,0]\times
\mathbf{K}$ and, furthermore, this convergence is the $C^k$-norm on
$[-T,0]\times \mathbf{K}$. In particular, this implies that, if $h$ is small
enough,
\[
\left| \int_{-T}^0 \Big(\{ \mathbf{H}_0^{(k)},
\widetilde{\mathbf{H}}_1\}\,\circ\,\Phi^t_{\mathbf{H}_0}(\gamma_{h,i}(\vec{\tau}))
 -\{ \mathbf{H}_0^{(k)}, \widetilde{\mathbf{H}}_1\}\,\circ\,\Phi^t_{\mathbf{H}_0}(\gamma_{0}(\vec{\tau}))\,\Big)\,dt
 \right| < \nu.
\]
\end{proof}

\subsection{Application to the Wave, Beam and Hartree equations: Proof of Theorem \ref{thm:geometric}--$(ii)$}\label{sec:ApplicationMultiple}
Recall the matrix \eqref{def:MatrixD}.
We now check the condition \eqref{cond:HessD} in Proposition \ref{prop:homointersaddlesN} for the resonant models associated to the equations \eqref{Wave}, \eqref{Beam} and \eqref{Hartree}. For the Wave and Beam equations this corresponds to choosing suitable sets $\Lambda$ (actually a suitable modification of those obtained in Proposition \ref{prop:LambdaSetWaveHom}). For the Hartree equation this corresponds to imposing a non-degeneracy condition to the potential $V$.

\begin{lem}\label{lemma:WaveBeamMelnikovCondN}
Let us consider either the Wave
equation \eqref{Wave} or the Beam equation \eqref{Beam}. Then, there exists a set $\Lambda\subset\Z^2$ satisfying Propositions \eqref{prop:LambdaSetWaveHom} such that the associated Hamiltonian \eqref{calham0} satisfies
condition \eqref{cond:HessD}.
\end{lem}

\begin{lem}\label{lemma:HartreeMelnikovCondN}
Let us consider Hamiltonian \eqref{calham0} associated to the Hartree equation
\eqref{Hartree} with a potential $V$ as in \eqref{def:potential} and to a set $\Lambda$
satisfying Proposition \ref{prop:LambdaSetHartreeHet}
Then, for a generic choice of the $\{\gamma_{n}\}_{n\in\Lambda}$, the condition  \eqref{cond:HessD}  is satisfied.
\end{lem}

These two lemmas allow us to complete the proof of Item $(ii)$ of Theorem  \ref{thm:geometric}

\begin{proof}[Proof of Item $(ii)$ of Theorem  \ref{thm:geometric}]
Lemmas \ref{lemma:WaveBeamMelnikovCondN} and Lemma \ref{lemma:HartreeMelnikovCondN} ensure that the non-degeneracy condition  \eqref{cond:HessD}  of Proposition \ref{prop:SplittingMulti}. Therefore, any pair of periodic orbits $\mathtt{P}_{\varepsilon,h,i}$, $\mathtt{P}_{\varepsilon,h,j}$ have transverse heteroclinic connections. This implies that all infinite sequences of such periodic orbits form a transition chain in the sense of Definition \ref{def:transition_chain}. Then, to complete the proof of Item $(ii)$ of Theorem \ref{thm:geometric} it is enough to apply  Proposition \ref{prop:diffusion_orbits}.
\end{proof}

We devote the rest of this section to prove Lemmas \ref{lemma:WaveBeamMelnikovCondN} and \ref{lemma:HartreeMelnikovCondN}.
 Lemma \ref{lemma:HartreeMelnikovCondN} is proved following the same argument of the proof of Lemma \ref{lemma:HartreeMelnikovCond}. To the prove  Lemma \ref{lemma:WaveBeamMelnikovCondN}, we consider a set $\Lambda_0\subset\Z^2$ satisfying Proposition \ref{prop:LambdaSetWaveHom} and we modify it slightly. By modification, we refer to construct a set $\Lambda\in\Q^2$ arbitrarily close to $\Lambda_0\subset\Z^2$ and then to scale it  so that the set belongs to $\Z^2$.

\begin{proof}[Proof of Lemma \ref{lemma:WaveBeamMelnikovCondN}]
  Let us call $n^{(i)}_k:=n_{4 (i-1)+k}$ for $i=1,\dots, N$, $k=1, \dots, 4$.
Recall the expression of the coefficients $d_{i j}$ in \eqref{coeffN}. By  using \eqref{def:coeffmatrixA} and Lemma \ref{lem:HresLambda} we obtain
\begin{equation}\label{def:d12N}
 d_{i j}=\frac{3}{32 \mathtt{g}} \,\sum_{\substack{1 \leq r, s \leq 4}}
\frac{(-1)^{r+s}}{|n^{(i)}_r|^{\kappa} |n^{(j)}_s|^\kappa},
 \end{equation}
 where $\kappa=1, 2$ corresponds respectively to the Wave and Beam equations.
 We define  (recall formulas \eqref{def:IndexesPiDelta})
\begin{align}
P_r=\Delta_{n^{(r)}_{1}, n^{(r)}_{2}} (|n_1^{(r)}|^{\kappa} |n^{(r)}_2|^{\kappa}-|n^{(r)}_3|^{\kappa} |n^{(r)}_4|^{\kappa}) \prod_r^{1, 2, 3, 4} \quad \text{where}\quad
\prod_r^{1, 2, 3, 4} &:= \frac{1}{ \prod_{i=1}^4 |n_i^{(r)}|^{\kappa} }. \label{def:Pkk}
\end{align}
We remark that by the resonance relations \eqref{def:Abhwt} and \eqref{def:Aw} we have $\Delta_{n^{(r)}_{1}, n^{(r)}_{2}}=\Delta_{n^{(r)}_{4}, n^{(r)}_{3}}$.
Then we can express the right hand side of \eqref{def:d12N} in terms of the $P_r$'s in the following way:
\[
d_{i j}=\frac{3}{32 \mathtt{g}}\, P_{i} P_j.
\]
Then, the determinant  of the matrix $\mathcal{D}$ in \eqref{def:MatrixD} is of the form
\begin{equation*}
\det(\mathcal{D}) = \left(\frac{3}{32\mathtt{g} }\right)^{N-1}\left( \prod_{k=1}^{N-1} P_k \right) \; \text{det} \;
\begin{pmatrix}
P_N + \sum_{j \neq 1}  P_j & -P_2 & \ldots & -P_{N-1} \\
-P_1 & \ddots & \vdots & \vdots \\
\vdots & \vdots & \ddots &  -P_{N-1} \\
-P_1 & \ldots & -P_{N-2} & P_{N} + \sum_{j \neq N-1} P_j
\end{pmatrix}.
\end{equation*}
This determinant can be written as
\begin{equation}\label{detHessianNP2}
\det(\mathcal{D}) = \left(\frac{3}{32\mathtt{g} }\right)^{N-1}\left( \prod_{k=1}^{N} P_k \right) \left(\sum_{k=1}^{N}P_k\right)^{N-2}.
\end{equation}
Indeed, it is enough to modify the matrix in two steps. First replace the last column by the sum of all columns. Then, the last column is the vector with all components equal to $P_N$. Second, subtract the last row to the other rows. Then, it is very easy to obtain \eqref{detHessianNP2}.

Recall  that in the proof of Lemma \ref{lemma:WaveBeamMelnikovCond} we have shown that the sets $\Lambda$ of Proposition \ref{prop:LambdaSetWaveHom} satisfy
\[
 |n_1^{(r)}|^{\kappa} |n^{(r)}_2|^{\kappa}-|n^{(r)}_3|^{\kappa} |n^{(r)}_4|^{\kappa}\neq 0
\]
(see \eqref{def:CondExtraForPk}). Moreover, in Proposition \ref{prop:LambdaSetWaveHom} it shown that  they also satisfy \eqref{eq:Wave}. These three properties imply
\begin{equation}\label{def:Pk0}
 P_k\neq 0\quad \text{ for all}\quad k=1,\ldots, N
\end{equation}
Therefore, by \eqref{detHessianNP2}, to prove $\det(\mathcal{D}) \neq0$, it only remains to check that
\begin{equation}\label{cond:sum}
 \sum_{k=1}^{N}P_k\neq 0.
\end{equation}
If the set $\Lambda$ obtained in Proposition \ref{prop:LambdaSetWaveHom} satisfies this property, the proof is complete. Now, we show  that if the set $\Lambda$ obtained in these propositions  satisfies $\sum_{k=1}^{N}P_k= 0$, one can modify it slightly so that the new one satisfies \eqref{cond:sum}.
Assume thus that  $\Lambda$ satisfies $\sum_{k=1}^{N}P_k= 0$ and  \eqref{def:Pk0}. Then, we  modify the first resonant tuple $(n_1^{(1)},n_2^{(1)},n_3^{(1)},n_4^{(1)})$ to obtain a set $\Lambda\subset\mathbb{Q}^2$ which satisfies \eqref{cond:sum}. We consider the family of resonant tuples in $\mathbb{Q}^2$, given by
\[
 (\lambda n_1^{(1)}, \lambda n_2^{(1)}, \lambda n_3^{(1)}, \lambda n_4^{(1)}), \qquad \lambda\in\mathbb{Q}\setminus\{0\}.
\]
Then, by \eqref{def:Pkk},
\[
 P_1(\lambda n_1^{(1)}, \lambda n_2^{(1)}, \lambda n_3^{(1)}, \lambda n_4^{(1)})=\lambda^{-\kappa} P_1(n_1^{(1)},n_2^{(1)},n_3^{(1)},n_4^{(1)}).
\]
Then, since $P_1\neq 0$,  $P_1$ is strictly decreasing in $\lambda$ and therefore $\sum_{k=1}^{N}P_k= 0$ can only happen for $\lambda=1$. Thus, one can modify the first rectangle by taking $\lambda\in\mathbb{Q}$ arbitrarily close to 1 and then blowing up the $N$ rectangles so that the new rectangles belong to $\mathbb{Z}^2$. It is clear that with this modification (for $\lambda$ close enough to 1) the properties in Proposition \ref{prop:LambdaSetWaveHom} are still satisfied.

 \subsection{End of the proof of Theorem \ref{thm:traveling_periodic_beating}}
  Lemmas \ref{lemma:WaveBeamMelnikovCondN}, \ref{lemma:HartreeMelnikovCondN} imply that the assumptions required by Proposition \ref{prop:SplittingMulti} hold. Then we can use Proposition \ref{prop:diffusion_orbits} to deduce dynamical results on the resonant models  of equations \eqref{Wave}, \eqref{Beam} and \eqref{Hartree}.

 Let us fix $N\geq 2$, $k\gg 1$ and a sequence $ \omega_1, \dots, \omega_k $ with $\omega_p\in \{1, \dots, N\}$ for $k=1\ldots k$.
 We apply Proposition \ref{prop:diffusion_orbits} choosing $\nu_{\ell}=\varepsilon$ for all $\ell=1, \dots, k$. Then there exist $T>0$ and  an orbit
 \[
 \gamma(t)=(\Psi_1(t), \dots, \Psi_N(t), K_1(t), \dots, K_N (t)), \quad t\in [0, T]
 \]
 of the Hamiltonian $\mathcal{H}$ (see \eqref{redhamN}) which has the following behavior:

There exists  some $\al_p,\beta_p$ satisfying $\al_p<\beta_p<\al_{p+1}$
such that, if one splits the time interval as $[0, T]=\calI_1\cup \calJ_{1, 2}\cup \calI_2\cup \calJ_{2, 3}\cup\dots \cup \calJ_{k-1, k }\cup \calI_k$
with
\[
\calI_p=[\al_p,\beta_p],\quad \calJ_{p,p+1}=[\delta^{-2}\beta_p,\delta^{-2}\al_{p+1}],
\]
the orbit $\gamma(t)$ has two different regimes
\begin{itemize}
 \item \emph{Beating regime:} For $t\in [\alpha_{p}, \beta_{p}]$,  $\gamma(t)$ belongs to an $\varepsilon$-neighborhood of the periodic orbit $\mathtt{P}_{\varepsilon, h, \omega_{p}}$. The orbit $\ga(t)$ spends $\OO(\log\varepsilon)$-time inside this neighborhood and then it leaves it.
  \item \emph{Transition regime:} For $t\in (\beta_{p}, \alpha_{p+1})$,  the orbit $\gamma(t)$ shadows the heteroclinic connection between two hyperbolic periodic orbits $\mathtt{P}_{\varepsilon, h, \omega_{p}}$ and $\mathtt{P}_{\varepsilon, h,  \omega_{p+1}}$.
 \end{itemize}
By  \eqref{def:timeparapo}, the time parameterization of the periodic orbit  $\mathtt{P}_{\varepsilon, h, \omega_{p}}$ satisfies
\[
 K_p(t)=Q(t), \qquad K_i(t)=0\quad \text{for}\quad i\neq p,
\]
where $Q(t)$ is a periodic orbit. Then, the orbit $\gamma(t)$ satisfies that for $t\in [\alpha_{p}, \beta_{p}]$
 \[
|K_p(t)-Q(t-t_p)|\le\eps,\qquad  | K_i(t) |^2\le \varepsilon \qquad \forall  \quad i\neq p,
 \]
for some $t_p>0$.

In the time interval $(\beta_{\ell}, \alpha_{\ell+1})$, the travel along the heteroclinic connection implies that all the actions $|K_i|^2$ experience a change of order $\OO(1)$ (see the proof of Proposition \ref{prop:SplittingMulti}).

By the symplectic reduction performed in Section \ref{sectionsymplecticreduction} there exists $r(t)$ solution of $H_{\Res}$ in \eqref{Hres} with Fourier support $\Lambda$ such that the actions $|r_{n_1^{(\omega_i)}}|^2$ satisfy
\[|r_{n_1^{(\omega_i)}}(t)|^2=| K_{i}(t) |^2\qquad \text{for}\qquad t\in [0,T].\]
This can be seen using Remark \ref{rem:Kcartcoord}, which gives also the behavior of the other actions.

Since the solutions of $H_{\Res}$ are invariant under the scaling \eqref{def:scaling}, we can consider $r^{\delta}(t):=\delta r(\delta^2 t)$. Then, $r^{\delta}(t)$ is also a solution of $H_{\Res}$ for  $t\in [0, \delta^{-2} T]$.

Now it only remains to obtain an orbit for the equations \ref{Wave}, \ref{Beam} and \ref{Hartree} which is close (up to certain changes of coordinates) to $r^{\delta}(t)$. First step is to apply Proposition \ref{approximationargument}. It ensures that there exists $0<\delta_2\ll 1$ such that for all $\delta\in (0, \delta_2)$, there exists a solution $w(t)$ of $H\circ \Gamma\circ \Psi=H_{\Res}+\mathcal{R}'$   such that $w(t)=r^{\delta}(t)+\widetilde{R}(t)$ with $\widetilde{R}(0)=0$, $\| \widetilde{R}(t) \|_{\rho}\lesssim \delta^2$ for $t\in [0, \delta^{-2} T]$. We note that,  by Item $(ii)$ of Proposition \ref{WBNF}, the Birkhoff map $\Gamma$ is $\delta^3$-close to the identity. Finally the transformations \eqref{rotatingcoord1} and \eqref{fase} preserve the modulus of the Fourier coefficients. The last change of coordinates that one has to apply (for the Wave \eqref{Wave} and Beam \eqref{Beam} equations) is  passing from complex coordinates \eqref{def:complexcoordinates} to the original ones. We remark that by \eqref{eq:Wave} if $n_i\in\Lambda$ then $-n_i\notin \Lambda$. Thus
  \[
  u_{n_i}=\frac{1}{\sqrt{2 |j|}} \Psi_{n_i} \qquad n_i\in\Lambda.
  \]

\end{proof}

\appendix

\section{The set $\Lambda$: Proof of Propositions
\ref{prop:LambdaSetHartreeHet} and
\ref{prop:LambdaSetWaveHom}}\label{app:LambdaSet}
The proofs of Propositions
\ref{prop:LambdaSetHartreeHet} and
\ref{prop:LambdaSetWaveHom} are modifications of the proof of the construction
of the set $\Lambda\subset\mathbb{Z}^2$ in \cite{CKSTT}. Note that the
resonances of the cubic nonlinear
Schr\"odinger equation considered in  \cite{CKSTT} are the tuples contained in
$\widetilde\AAA_{bh}$ in \eqref{def:Abhwt}. We summarize the
ideas in
that paper and explain the main modifications.

 In \cite{CKSTT}, the set $\Lambda$ is first constructed in $\mathbb{Q}^2$ and
then scaled to $\Z^2$.  The placement of the modes in
$\Q^2$ is done
inductively: first one places the modes in  $\Lambda_1$, then those in
$\Lambda_2$, checking at each placement that conditions
$1_\Lambda$--$4_\Lambda$ are fulfilled. To this end, one has to  ensure that
the imposed non-degeneracy conditions are open and dense in $\Q^2$ and then for
``most of the placements'' are satisfied. More concretely, the placement goes
as follows
\begin{itemize}
\item \emph{First generation}: In order to place the first generation we
have to chose $2N$ points in $\mathbb{Q}^2$. We choose them inductively checking that they satisfy the non-degeneracy conditions. Condition $2_\Lambda$ and $3_\Lambda$ are satisfied if all the points are chosen different and $1_\Lambda$ will be satisfied by construction. The condition $4_\Lambda$ is equivalent to check that each new point does not make a right angle with two of the modes already placed. That is, consider any segment whose endpoints are two points already chosen. Then, this new point cannot belong either to a line orthogonal to this segment and containing one of the points nor to the circle having this segment as a diameter.

\item \emph{Second generation}: The set $\Lambda_1$ is divided into pairs of
modes, which are the parents of the $N$ nuclear families. For each of these
pairs $n_1,n_3\in\Lambda_1\subset\Q^2$, we place a pair of points
$n_2,n_4\in\Lambda_2$ in such a way that they form a rectangle with the other
pair. That is, we consider the circle having as a diameter the segment between
$n_1$ and $n_3$. Then, the new modes $n_2$,$n_4$ have to be endpoints of another
diameter of this circle. To ensure that $n_2,n_4\in\Q^2$ it is enough to chose
an angle between the two diameters which has rational tangent. Note that those angles are dense. The choice is done checking that the non-degeneracy conditions are verified $1_\Lambda$--$4_\Lambda$ following the same arguments as for the first generation.
\end{itemize}

This placement is \emph{generic} in the following sense
\begin{enumerate}
\item The first generation is placed generically in $\Q^2$, that is anywhere except in the zero set of one polynomial.
\item The placement angles $\theta$ for the second generation are any angle such that  $\tan\theta\in\Q$ except a finite number of values.
\end{enumerate}

%

We use this scheme developed in \cite{CKSTT} to prove Proposition \ref{prop:LambdaSetHartreeHet}.

\begin{proof}[Proof of Proposition \ref{prop:LambdaSetHartreeHet}]
It is a direct consequence of the  scheme developed in \cite{CKSTT}. Indeed, the only extra condition added with respect to \cite{CKSTT} is \eqref{nondegLambda}, which is certainly satisfied by a generic placement in $\Q^2$. Indeed, in placing inductively the new points one has only to avoid a finite number of points.
\end{proof}

\begin{proof}[Proof of Proposition \ref{prop:LambdaSetWaveHom} (Beam case: $\AAA=\AAA_{bh}$)]
The set $\Lambda$ from  Proposition \ref{prop:LambdaSetWaveHom} has
three differences with respect to the one in \cite{CKSTT}:  properties
\eqref{eq:Wave} and \eqref{eq:Wave2} and the fact that the
condition $4_\Lambda$ requires that the modes in $\Lambda$ do not satisfy any
of the resonance conditions in $\AAA_{bh} \setminus \widetilde\AAA_{bh}$
in \eqref{def:Abh}-\eqref{def:Abhwt}. One can easily
check that a generic placement satisfies  \eqref{eq:Wave} and the
$4_\Lambda$ condition. Indeed, in placing the
new modes one has to avoid circles centered at zero with radius equal to the
norm of the already placed modes and the circles and hyperbolas defined by
\eqref{def:Abh} when two modes are fixed.

To build a set $\Lambda$  having property \eqref{eq:Wave2}, we also
follow the ideas in \cite{CKSTT}. We first construct a prototype embedding. That
is  a ``bad'' set $\Lambda_0\in\Q^2$ which is the union of $N$ rectangles but
which however does not satisfy the non-degeneracy conditions. For instance,
consider
\[
 \Lambda_0=\cup_{i=1}^N\mathcal{R}_i,\qquad \mathcal{R}_i=\{(\pm 1,0),(0,\pm 1)\}.
\]
This embedding satisfies \eqref{eq:Wave2} but does not satisfy conditions
$1_\Lambda- 4_\Lambda$ nor \eqref{eq:Wave} (in particular is not
injective). However, by genericity one can chose points in
$\Q^2$ which are $\eps/4$-close to those of $\Lambda_0$  which define a set
$\Lambda$ satisfying that all points are different and also  conditions
$1_\Lambda-4_\Lambda$.

Finally, one needs to apply a scaling and a translation to obtain a set
$\Lambda\subset\Zodd\times\mathbb{Z}$. Indeed, consider  $R\gg 1$ such that
$R\Lambda\subset\mathbb{Z}^2$ and $R\eps\gg 1$. Then, we define
\[
 \Lambda'=2R\Lambda+(1,0)
\]
Then, one can check that for $n'\in \Lambda'$, which is of the form
$n'=2Rn+(1,0)$ with $n\in\Lambda$, taking $R$ large enough,
\[
\left||n'|-2R\right|=2R\left|\left|n+\left(\frac{1}{2R},
0\right)\right|-1\right|\leq \frac{R\eps}{2}+\OO(1)\leq  R\eps.
\]

\end{proof}

\begin{proof}[Proof of Proposition \ref{prop:LambdaSetWaveHom} (Wave case: $\AAA=\AAA_w$)]
 To prove Proposition \ref{prop:LambdaSetWaveHom} one has to take into account
that the resonance condition for the Wave equation \eqref{Wave} given in
\eqref{def:Aw} is different from that of the cubic nonlinear Schr\"odinger,
Hartree \eqref{Hartree} and Beam \eqref{Beam} equations (see \eqref{def:Abh}).
Now four resonant modes $(n_1,n_2,n_3,n_4)\in\mathbf{A}_w$ form a
parallelogram whose vertices are on an ellipse with one focus at zero.
Indeed, if one fixes the modes $n_1$ and $n_3$, then $n_2,n_4$ must belong to
the ellipse defined by
\begin{align} \label{eq:ModesWave}
\left\{n\in\mathbb{Q}^2:|n| + |n-(n_1+n_3)| = |n_1|+|n_3|\right\},
\end{align}
that is, the ellipse with foci at $0$ and $n_1+n_3$ and such that the sum of distances from any point of the ellipse to the two foci is given by $|n_1|+|n_3|$. Note that the case $n_1=-n_3$ trivially corresponds to the circle with center $0$ and radius $|n_1|$.

We need to consider $N$ ellipses of this type with dense rational points to apply the genericity arguments as in the previous cases. The standard ellipse
\begin{equation}\label{def:StandardEllipse}
 \frac{x^2}{a^2}+\frac{y^2}{b^2}=1
\end{equation}
 has dense rational points provided $a,b\in\Q$. To obtain ellipses of the form
\eqref{eq:ModesWave} from \eqref{def:StandardEllipse} one needs to apply a
translation (one could also apply a rotation, but there is no need for it). To
ensure
that the transformed ellipse has dense rational points it is enough to ensure
that the foci of the standard ellipse \eqref{def:StandardEllipse} are rational.
Assuming that $a>b$, the foci are given by $F_\pm=(\pm
c,0)=(\pm\sqrt{a^2-b^2},0)$.

 Therefore, to build  ellipses $\mathcal{E}_j$ with dense rational
points , $j=1\ldots N$, it is enough
to consider $N$ different rational Pythagorean triples
$\{(a_j,b_j,c_j)\}_{j=1}^N$, that is $a_j^2=b_j^2+c_j^2$, $a_j,b_j,c_j\in\Q$,
$a_j>b_j$. Then, one can apply a translation to place one of the foci at 0.
Let us denote by
$F_j$ the focus of the ellipse $\mathcal{E}_j$ which is not at the origin.

Having fixed these ellipses, one can prove Proposition
\ref{prop:LambdaSetWaveHom} following the scheme of
\cite{CKSTT} explained above. One first places each pair of the first generation in one of
the ellipses. To place one pair $\mathcal{E}_j$  it is enough to chose one
rational point $n_{j_1}\in \mathcal{E}_j$. Then, the other mode is obtained
through the equation
\begin{equation*}
n_{j_1}+n_{j_3}=F_j\qquad \text{ (see \eqref{eq:ModesWave}).}
\end{equation*}
Since the ellipses have dense rational points, one can
place the points  such that the conditions $2_\Lambda-4_\Lambda$ and
\eqref{eq:Wave} are satisfied as follows. Let us assume that we have
placed all modes of the first generation for the ellipses $\mathcal{E}_j$,
$j=1\dots j^*-1$ and we want to place the first generation modes in the ellipse
$\mathcal{E}_{j^*}$. We show that we only need  to avoid a finite number of points.
\begin{enumerate}
\item For property \eqref{eq:Wave}, we need to avoid the
intersection points of $\mathcal{E}_{j^*}$ with all the circles centered at the
origin and radius equal to the norm of the already placed modes.
\item For properties $2_\Lambda$, $3_\Lambda$, we need to avoid the  points at
the intersection of $\mathcal{E}_{j^*}$ with the other ellipses
$\mathcal{E}_j$, $j=1\ldots j^*-1, j^*+1\ldots N$.
\item For property $4_\Lambda$, one needs to avoid placing a mode such that
with two previous modes $m,m'$ and an extra mode may create a nuclear family.
To this end we have to avoid the following points:
\begin{itemize}
 \item Case (i) -- $m, m'$ are non adjacent vertices of the parallelogram: One has to avoid the
intersection points between $\mathcal{E}_{j^*}$ and the ellipse defined by $m,m'$,
that is
\[
 |n|+|n-(m+m')|=|m|+|m'|
\]
Note that this ellipse is different from  $\mathcal{E}_{j^*}$ since by Item 2 above, $m,m'\not \in\mathcal{E}_{j^*}$.
\item Case (ii) -- $m, m'$ are adjacent vertices of the parallelogram: One has to avoid the
intersection points  between $\mathcal{E}_{j^*}$ and the hyperbolas defined by
$m,m'$,
that is
\[
 |n|-|n-(m+m')|=\pm |m|\mp|m'|.
\]
\item One can deal analogously with the conditions which arise from avoiding
the resonances conditions in $\AAA_w$ given by
\[
 n_1+n_2+n_3-n_4=0,\quad |n_1|+|n_2|+|n_3|-|n_4|=0,
\]
which either define ellipses or hyperbolas.
\end{itemize}
Note that the two new placed modes and one already placed mode cannot be part
of a nuclear family since the already placed mode does not belong to the
ellipse defined by the two new modes.
\end{enumerate}

One can proceed analogously to place the
second generation. Note that this construction implies Property $1_\Lambda$.

To build a set $\Lambda$ satisfying also condition \eqref{eq:Wave2} it is
enough to chose the rational Pythagorean triples $\{(a_j,b_j,c_j)\}_{j=1}^N$ such that
$|a_j-1|,|b_j-1|,c_j\ll \eps$ in such a way that the ellipses are $\eps$-close
to the unit circle. Note that this is possible since, in particular, rational
Pythagorean triples are dense in the unit circle.

This construction gives a set $\Lambda$ in $\Z^2$. Note that one cannot scale
and translate to construct a set $\Lambda$ in $\Z^2_\odd$ as in the proof of
Proposition \ref{prop:LambdaSetWaveHom} for the Beam case. Indeed, the resonance condition
\eqref{def:Abhwt} is not invariant by translation. Instead, we refine the
construction of the set $\Lambda$ in $\Q^2$ by choosing more carefully the
modes.

To this end, we recall that the rational modes on the unit circle are given by
 \[
z=\left(\frac{p_1}{q}, \frac{p_2}{q}\right)=\left(\frac{m^2-n^2}{m^2+n^2}
,\frac{2mn}{m^2+n^2}\right),\quad m,n\in\Z.
 \]
If one choses $m$ odd and $n$ even  one obtains a point $z\in\Q^2$
whose denominator is odd and their numerators are odd in the first component and
even in the second component. Certainly such points are dense in the unit
circle. After a blow up by $q$ (or any odd multiple of it), one obtains a
point in $\Z^2_\odd$.

We show that one can construct a set $\Lambda\subset\Q^2$ as just done keeping
track of the rational numbers to show that all of them can be chosen of the
form
\begin{equation}\label{def:OddEven}
 z=\left(\frac{\text{odd}}{\text{odd}},\frac{\text{even}}{\text{odd}}\right).
\end{equation}
Indeed, one can choose the ellipses $\mathcal{E}_j$ with  rational Pythagorean
triples
$\{(a_j,b_j,c_j)\}_{j=1}^N$, $a_j,b_j,c_j\in\Q$, such that $a_j$, $b_j$ are of
the form $\text{odd}/\text{odd}$ and $c_j$ is $\text{even}/\text{odd}$. Then,
the rational points on the ellipse $\mathcal{E}_j$ are of the form
\[
 z=\left(c_j+a_j\frac{m^2-n^2}{m^2+n^2}
,b_j\frac{2mn}{m^2+n^2}\right),\quad m,n\in\Z.
\]
Choosing $m$ odd and $n$ even, one has a point $z$ of the form
\eqref{def:OddEven}. Since points of this form are dense in $\mathcal{E}_j$ one
can proceed the construction such that all points in $\Lambda\subset \Q^2$ are
of the form \eqref{def:OddEven}.

Finally, it only remains to multiply by the least common divisor of all points
in $\Lambda$ to obtain a set in $\Z^2_\odd$ and the same happens by the
multiplication by any odd multiple of the least common divisor.

\end{proof}

\bibliography{references}
\bibliographystyle{alpha}

\end{document}